\documentclass{article}
\usepackage[text={155mm,220mm},centering]{geometry}
\usepackage[utf8]{inputenc}

\usepackage[linktocpage=true]{hyperref} 
\hypersetup{ colorlinks=true, linkcolor=blue, filecolor=magenta, urlcolor=blue, allcolors=blue,}

\usepackage{blindtext}
\usepackage{titlesec}
\usepackage{comment}
\usepackage{amsfonts}
\usepackage{amsmath}
\usepackage{amsthm}
\usepackage{amssymb}
\usepackage{booktabs}
\usepackage{multirow}
\usepackage{multicol}
\usepackage{mathrsfs} 
\usepackage{hyperref}
\usepackage[symbol]{footmisc}
\usepackage{enumitem}

\usepackage[title, toc]{appendix}
\usepackage{times}
\usepackage{graphicx}
\usepackage{subcaption}

\usepackage{bbm}
\usepackage{natbib}
\usepackage{enumitem}
\usepackage[dvipsnames]{xcolor}
\usepackage{algpseudocode}
\usepackage{algorithm}
\usepackage{authblk}
\usepackage{amsmath}

\def\abs#1{\left|#1\right|}
\def\aff{{\sf aff}}

\def\bH{\mathbf{H}}

\def\bJ{\mathbf{J}}

\def\bX{\mathbf{X}}

\def\bTheta{{\boldsymbol\Theta}}

\def\bx{\mathbf{x}}
\def\by{\mathbf{y}}

\def\bbE{\mathbb{E}}

\def\bbH{\mathbb{H}}

\def\bbN{\mathbb{N}}

\def\bbR{\mathbb{R}}
\def\bbS{\mathbb{S}}

\def\cA{\mathcal{A}}

\def\cC{\mathcal{C}}

\def\cE{\mathcal{E}}

\def\cH{\mathcal{H}}
\def\cI{\mathcal{I}}

\def\cO{\mathcal{O}}
\def\cP{\mathcal{P}}

\def\cS{\mathcal{S}}
\def\cT{\mathcal{T}}
\def\cU{\mathcal{U}}
\def\cV{\mathcal{V}}

\def\cX{\mathcal{X}}

\def\conv{{\sf conv}}

\def\KL{{\sf KL}}
\def\extr{{\sf extr }}

\def\norm#1{\left\|#1\right\|}

\def\TV{{\sf TV}}

\newcommand{\innerprod}[1]{\left\langle#1\right\rangle}

\def\doc{\bX_{[n]}}
\def\corpus{\doc^{[m]}}
\def\bTheta{\boldsymbol{\Theta}}

\newtheorem{theorem}{Theorem}
\newtheorem{lemma}{Lemma}[section]
\newtheorem{proposition}{Proposition}
\newtheorem{definition}{Definition}[section]
\newtheorem{corollary}{Corollary}

\theoremstyle{remark}
\newtheorem{example}{Example}
\newtheorem{remark}{Remark}

\title{Learning Topic Hierarchies by Tree-Directed Latent Variable Models
}

\author{
  Sunrit Chakraborty$^*$, Rayleigh Lei$^*$, XuanLong Nguyen \\
  Department of Statistics \\
  University of Michigan \\
  Ann Arbor, MI - 48105\\
  \texttt{\{sunritc, rayleigh, xuanlong\}@umich.edu}
}

\begin{document}
\maketitle

\begin{abstract}
    We study a parametric family of latent variable models, namely topic models, equipped with a hierarchical structure among the topic variables. 
Such models may be viewed as a finite mixture of the latent Dirichlet allocation (LDA) induced distributions, but the LDA components are constrained by a latent hierarchy, specifically a rooted and directed tree structure, which enables the learning of interpretable and latent topic hierarchies of interest. A mathematical framework is developed in order to establish identifiability of the latent topic hierarchy under suitable regularity conditions, and to derive bounds for posterior contraction rates of the model and its parameters.  We demonstrate the usefulness of such models and validate its theoretical properties through a careful simulation study and a real data example using the New York Times articles.
\end{abstract}


\tableofcontents

\section{Introduction}\label{sec:introduction}
Topic models have widely been used to analyze text corpora with the goal of discovering abstract \textit{topics} that occur in the collection of documents \citep{BleiEtAlLatentDirichletAllocation2003,BleiLaffertyDynamicTopicModels2006}. A topic is a probability distribution over a vocabulary, representing a particular hidden semantic structure in the corpus. A document may be associated with multiple topics present in different proportions. Understanding such latent semantics enables automated categorization and tagging, which is required by organizations handling a large number of unstructured text documents. Apart from text corpora, the same modelling framework was developed in population genetics \citep{PritchardEtAlInferencePopulationStructure2000} for identifying ancestral population structure from gene samples, and also deployed in a vast range of domains as diverse as quantitative biomedicine (e.g., extracting information of cancers' genomic samples \citep{valle2020topic}) and audio analysis (e.g., understanding hidden structures in music and how music styles evolved over time \citep{shalit2013modeling}). 

Topic models may be broadly classified into two types -- while non-probabilistic models such as 
Non-negative Matrix Factorization \citep{6634167} and Latent Semantic Analysis \citep{8455018} focus on low-rank decomposition of the document-proportion matrix, probabilistic models such as Latent Dirichlet Allocation \citep{BleiEtAlLatentDirichletAllocation2003} and Pachinko Allocation \citep{li2006pachinko} employ a probabilistic generative model for the data using latent variables. In general, mixture models represent the simplest instance of models with latent variables, and provide a basic modelling framework for text corpora -- e.g., mixture of product multinomial distributions is a choice for a text data generating model, where the components represent the topics. However, topic models are considerably richer than the basic mixture models in that different documents contain topics in possibly different proportions, whereas in mixture models, each document has the same mixture probabilities for different components. As a result, in topic models, the mixture probabilities are different across documents -- this structure is often referred to as admixture.


Latent Dirichlet Allocation (LDA) \citep{BleiEtAlLatentDirichletAllocation2003} is the most popular representative of topic modeling and is in fact a building block of this paper. 
In a nutshell, the LDA assumes there are $K$ topics in the corpus. Document specific mixture probabilities (or \textit{allocation}) are assumed to follow a Dirichlet distribution. Given the allocation for a document, words are conditionally i.i.d. from a mixture of multinomial distributions, with parameters given by the topics themselves. Thus, each document uses all the topics. The LDA model has been extended in several directions, for instance, Dynamic Topic Models for temporal topic models \citep{BleiLaffertyDynamicTopicModels2006}, Correlated Topic Models for allowing correlation among the topics \citep{blei2006correlated}, non-parametric LDA using a Hierarchical Dirichlet Process prior over the topics \citep{TehEtAlHierarchicalDirichletProcesses2006} and Gaussian LDA for continuous observations, replacing the multinomial kernel by a Gaussian kernel \citep{das2015gaussian}.

\begin{figure}
    \centering
    \includegraphics[width=0.98\linewidth]{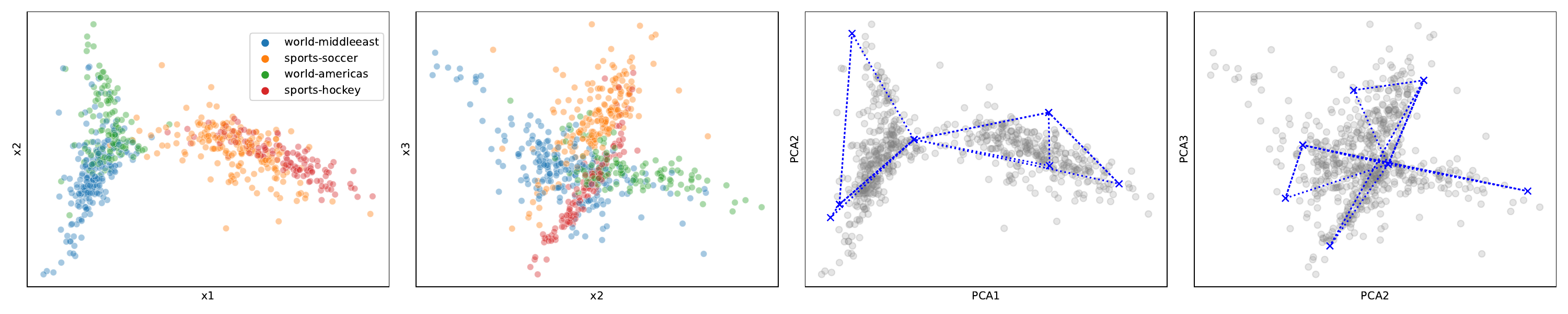}
    \caption{Illustration of a corpus (sample of 2016 New York Times articles) which demonstrates that one topic polytope is insufficient for explaining the heterogeneity in the corpus. Details about the data are in Section \ref{sec:numerical exp}. The right two pictures show estimated component topic polytopes using a tree-directed topic model (the model fitting does \emph{not} take the labeled news categories into account).}
    \label{fig: nyt example intro}
\end{figure}

Despite its versatility and adoption across a vast range of domains, the LDA model is not quite suitable for many data examples. A geometric illustration of this assertion can be seen in Figure \ref{fig: nyt example intro}. As will be elaborated in Section \ref{sec:background}, the LDA model assumes that the document means lie on a single polytope, the convex hull of the topics. This polytope is commonly called \textit{topic simplex} \citep{BleiEtAlLatentDirichletAllocation2003}, and more generally \textit{population polytope} \citep{NguyenPosteriorContractionPopulation2015}. However, a single such topic polytope might not fully explain the latent structure in the corpus. In Figure \ref{fig: nyt example intro}, the left two figures are the 2-dimensional PCA projections of the documents (color indicating category) of a subset of the 2016 New York Times corpus, where each point represents a particular document (the document-wise mean proportion of words). The figure shows that there is clearly a more complex structure in the corpus which may be not captured by using a single polytope. This motivates a "mixture of topic polytopes" approach pursued in this work. Moreover, the mixture components carry additional structures. For instance, although the \textit{sports} and \textit{world} categories roughly lie on disjoint polytopes, these two polytopes should share at least one vertex in the middle, which captures words common to both of these \textit{topics}. Furthermore, within \textit{sports}, the soccer (shown in orange) and hockey (shown in red) documents also share some topics, which are unique to sports. In addition, some have topics unique to themselves. This motivates us to use topic polytopes sharing vertices (or lower dimensional faces, using the language of convex geometry) to model the topics. Note that when fitting the topic models, we did \emph{not} take the categories into consideration. The visualization in Figure \ref{fig: nyt example intro} is only to motivate the need of more complex topic models. The two right images in the figure represent the unlabelled documents along with learnt topic polytopes (shown with dotted blue edges) using our model, which exhibits a finer understanding of the organization of topics in this corpus, and demonstrates that such learning from unlabelled data is possible in practice.

A different and interesting perspective of topic modeling extensions came from the original proposal of \cite{BleiEtAlNestedChineseRestaurant2010}, which focused on the learning of potential latent hierarchies among the topics using a non-parametric Bayesian prior distribution on (infinitely branched) random tree structures. 
These authors were motivated by a notion of hierarchy among topics in which some topics may be more significant than others, and proposed to use a tree structure to represent such a hierarchy among the topics.
As an example with respect to the data used for Figure \ref{fig: nyt example intro}, \textit{soccer} and \textit{hockey} might be naturally considered as sub-topics of the topic \textit{sports}. Using the tree-theoretic language, soccer and hockey may be represented as children of the vertex sport. Thus,
tree structures are natural and economical means for representing the relationship among topics. A probabilistic model which incorporates such a latent topic hierarchy can lead to not only a better understanding of the semantics of topics, thereby improving their interpretability, but also facilitates text categorization and reduces the strong reliance on pre-processing, as required by most existing topic models. 


The primary goal of this paper is to study whether such a latent and abstract topic hierarchy can be learned from unlabelled data and how efficiently. To this end, we investigate the connection between topics arranged in a rooted tree structure in the sense of \cite{BleiEtAlNestedChineseRestaurant2010} and the geometric perspective of multiple topic polytopes sharing vertices described earlier. We ask, in particular, \textit{is the latent topic hierarchy even identifiable to begin with}? This question remains unexplored, to the best of our knowledge. It is quite challenging, even when it is relatively simple to specify a latent structure of interest, because the relationship between the latent quantities and observed data is not very transparent in the associated statistical models. Understanding identifiability in such complex latent variable models is crucial for meaningful estimation of parameters (topics arranged in a tree for example), and for obtaining the learning rates. 
The approach taken in our work is to understand the geometry of the latent topic hierarchies arising from the sharing of topics across documents, and to establish conditions under which such a latent abstract topic structure is identifiable and can be learnt consistently from the corpus.

To address the theoretical issues associated with the learning of topic hierarchies, we focus on a class of probabilistic topic models, namely \textit{tree-directed topic models}, which highlight the setting that the topic hierarchy be represented as a \emph{finite} and \emph{rooted tree} structure. 
Although our topic modeling formulation is inspired by the nonparametric modeling of \cite{BleiEtAlNestedChineseRestaurant2010}, its parametric character is relatively simpler to enable us to address rigorously the question whether such a tree-based hierarchical structure in the latent topics can be identified and learnable from the data. The tree-directed topic model is still quite complex to be a good representative of the broad class of latent structured topic models of interest.
The finite tree assumption allows the question on identifiability and topic hierarchy interpretability to be addressed satisfactorily. It is not too restrictive in practice, as most computational techniques employ parametric approximation during some steps, and interpretability of the topic hierarchy makes practical sense only for a finite number of components. 
The \emph{rooted} tree structure allows for economical and intuitive interpretations and a clean theoretical treatment, although such a constraint can be relaxed in a future work.

An instructive view of the tree-directed topic model studied in this paper is that of a mixture of the LDA-induced components, but with a distinction: the unobserved tree structured topic hierarchy imposes certain restrictions on how the topics associated to each component of this mixture are shared. Identifying and learning such a latent tree structured topic hierarchy consists of two aspects: (i) identifying and learning of the rooted tree structure and (ii) obtaining a topic map, a fundamental notion that we will define formally in the sequel. Such a topic map embeds the abstract tree structure into a space of "meanings", i.e., a collection of topics represented as probability distributions over the vocabulary, along with constraints on the sharing structure. It is the marriage of the rooted tree structures with the topic semantic formalism that constitutes the heart of this work, and requires a major portion of our technical effort.

Summarizing, in this paper we shall present the following contributions.
\begin{itemize}
    \item [(i)] Formulation of a class of rich and computationally tractable topic models, namely tree-directed topic models, which incorporates identifiable and interpretable notions of topic hierarchy.
    \item [(ii)] An identifiability theory for the latent topic hierarchy arising in the tree-directed topic models, under the large document size regime and other suitable regularity conditions.
    \item [(iii)] Establishing bounds on the posterior contraction rates for the latent variables representing the topic structure arising from the empirical data. 
\end{itemize}
We additionally complement our model and theoretical study with a collapsed Gibbs sampler-based inference algorithm and extensive simulation and a real data analysis based on New York Times corpus. It will be shown that our model can capture interesting semantic structures in the corpus. Our work establishes that such complex latent structures can indeed be learnt from data under suitable regularity conditions, and under a well-specified model setting, tools developed here can be used to provide bounds on the learning rates of the parameters of interest. For this theory, we adopt the general framework for the Bayesian asymptotics treatment \citep{GhosalVaartFundamentalsNonparametricBayesian2017}, while the considerable novelty lies in extending the approach pursued in \cite{NguyenPosteriorContractionPopulation2015} for a single convex polytope to the setting of a mixture of convex polytopes.

The remainder of the paper is organized as follows. Section \ref{sec:background} provides the necessary background on topic modeling and graph-theoretic formalism that will be crucial in the rest of the paper. The tree-directed topic models will be introduced in 
Section \ref{sec:model}, along with examples and geometric structures of the distributions such models induce. Section \ref{sec:identifiability} investigates questions of identifiability. Section \ref{sec:rate} presents results on the model's posterior contraction behavior. We provide simulation studies and numerical experiments with real data in Section \ref{sec:numerical exp}. Concluding remarks are given in Section \ref{sec:conclude}. Complete proofs are relegated to the Appendix.

\subsection*{Notations:} $\bbR^d$ represents the $d-$dimensional Euclidean space and $\Delta^{d-1}=\{\bx\in \bbR^d\mid x_i\geq 0\, \forall i,\sum_i x_i=1\}$ represents the $(d-1)-$dimensional probability simplex. We denote the set $\{1,2,\dots,n\}$ by $[n]$. For a set of points $A\subset\bbR^d$, denote the convex hull of $A$ by $\conv A$ and affine hull by $\aff A$. We use the term \textit{simplex} to denote the convex hull of $K$ affinely independent points and each of these points are called \textit{extreme points} of the simplex. We also use \textit{polytope} to refer to a generic convex polytope $\cS$, i.e. convex hull of finitely many points and \textit{extreme points} of a polytope refer to its vertices, denoted as $\extr \cS$. We recall the dimension of a polytope to be the dimension of its affine space, defined as the dimension of the unique linear subspace that is a translation of the affine space. For a tuple $A=(a_1,\dots,a_n)$, we use the notation $\text{set}(A)$ to indicate the multiset (unordered collection of objects with repetitions) containing the elements in the tuple.

For $\bx,\by\in\bbR^d$, $\norm{\bx-\by}$ denotes the $L_2$-norm (aka Euclidean norm) and $\norm{\bx-\by}_1$ denotes the $L_1$-norm. We denote the Euclidean ball centered at $x$ of radius $r$ as $B(x,r):=\{y\in \bbR^d: \norm{y-x}_2 \leq r\}$ (the ambient space is clear from context). The notation $B_p(x,r)$ indicates an Euclidean ball of dimension $p$, i.e. $B(x,r)\cap \cA$ for some $p-$dimensional affine space $\cA$ (in this case we mention the supporting affine space on which this ball lies). For probability measures $P, Q$ on $\bbR^d$, we shall denote the total variation distance, Kullback-Liebler divergence, Hellinger distance and $q-$Wasserstein distance as $d_{\TV}(P,Q)$, $\KL (P\vert Q)$, $h(P,Q)$ and $W_q(P,Q)$ respectively. We shall mostly denote a probability measure on $\bbR^d$ by an upper-case letter , e.g. $P$ and its density by the corresponding lower-case letter, e.g. $p$. If nothing else is mentioned, the density is with respect to the Lebesgue measure on $\bbR^d$. For a measureable function $f:\bbR^d\to\bbR$, denote $Pf := \int f(x) P(dx)$. Finally, for non-empty $A,B\subset \bbR^d$,  the Hausdorff metric between them is denoted by $d_{\cH}(A,B)$. For probability distributions, let $\text{Mult}(n,p)$ denote the multinomial distribution, $\text{Cat}(p)$ the categorical distribution and $\text{Dir}_K(\alpha)$ the symmetric Dirichlet distribution on $\Delta^{K-1}$ with parameter $\alpha$.

\section{Preliminaries}\label{sec:background}
In this section, we provide the necessary background on topic modeling and rooted tree formalism. Our purpose is two-fold. First, we introduce the existing ideas and relevant results in the topic modeling literature. Second, we set up further notations and the setting based on which we shall define and study our model from the following section. In particular, we recollect some notions from graph theory regarding tree and isomorphisms, that will be used extensively in the sequel.

The data corpus consists of $m$ documents, $\bX_1,\dots,\bX_m$, where each document contains $n$ words, $\bX_i=(x_{i1},\dots,x_{in})$. The words in the corpus belong to a fixed vocabulary, say $[V]$. We shall use the same notations when introducing our model in Section \ref{sec:model}.

\subsection{Latent Dirichlet Allocation}
The LDA model provides a probabilistic generative model for the corpus. It posits that there are $K$ topics, specifically, $\theta_1,\dots,\theta_K\in\Delta^{V-1}$, and each document carries all these topics in varying proportions. For each document indexed by $i=1,\ldots, m$, first a document-specific allocation (proportion) $\beta_i\in\Delta^{K-1}$ is generated (the LDA assumes $\beta_i \sim \text{Dir}_K(\alpha)$ for some $\alpha\in\bbR_+^K$ i.i.d. across documents) and then conditionally given the $\beta_i$, words in the document are i.i.d. from the categorical distribution with parameter $\Theta^\top\beta$, where $\Theta\in\bbR^{K\times V}$, obtained by stacking the topics as rows, i.e., for $i\in [m], j\in [n], v \in [V]$,
$$p(x_{ij}=v|\theta,\beta_i) = \sum_k \beta_{ik}\theta_{kv}.$$
It is more common to express the last line equivalently in terms of the latent discrete topic assignment variable, which facilitates designing Gibbs sampler based inference methods:
\begin{align*}
    z_{ij}|\beta_i &\sim \text{Cat}(\beta_i) \\
    x_{ij}|\Theta, z_{ij}=k &\sim \text{Cat}(\theta_k).
\end{align*}
From the above model specification, the joint distribution of a single document can be expressed as
\begin{align*}
    p_{\text{LDA}}(\bx|\Theta,\alpha) &= \int_{\Delta^{K-1}} \prod_{j=1}^n \prod_{v\in[V]}\left(\sum_k \beta_{k} \theta_{kv}\right)^{1(x_j=v)} \text{Dir}_{K,\alpha}(d\beta) \\
    &= \int_{\Delta^{K-1}} \prod_{v\in[V]} \innerprod{\theta_{\cdot v}, \beta}^{m_v} \text{Dir}_{K,\alpha}(d\beta),
\end{align*}
where $\theta_{\cdot v} = (\theta_{1,v}, \dots, \theta_{K,v})^\top$ and $m_v=\sum_j 1(x_j=v)$ is the number of times word $v$ appear in the document. We shall write $\bX\sim \text{LDA}(\Theta,\alpha)$ to denote that a document $\bX$ follows the above distribution. Let us write $\cS=\conv(\Theta)$ to be the topic simplex. The Dirichlet allocation induces a pushforward measure $G = L_{\#} \text{Dir}_{K,\alpha}$, where $L:\Delta^{K-1}\to \cS$ by $L(\beta) = \sum_k \beta_k\theta_k$. In terms of this measure, the joint distribution of a document can be expressed as
$$p_{\text{LDA}}(\bx|\Theta,\alpha) = \int_{\cS} \prod_{v\in[V]} \eta_v^{m_v} G(d\eta).$$
This can be visualized as follows: the Dirichlet distribution induces a probability measure on the topic simplex $\cS$, a document first chooses $\eta\sim G$ and then draws words as conditionally i.i.d. observations from a categorical distribution with parameter $\eta$. As $n\to \infty$ (document size increases), the distribution of $\hat{\eta}$, defined as $\hat{\eta}_v = m_v/n$, converges to $G$ in probability. Furthermore, based on the above display, $G$ may be viewed as the latent mixing measure for the LDA, borrowing the notion from mixture models: $p(x)=\sum_k \pi_k f(x|\theta_k)$, where $f$ is some probability kernel and the mixing measure is $G=\sum_k \pi_k\delta_{\theta_k}$ so that $p(x)=\int f(x|\eta)G(d\eta)$ -- see \cite{nguyen2013convergence,ho2016strong} for details. Note that $G$ is discrete and finitely supported for mixture models and the choice of kernel being multinomial.

The LDA has been studied extensively in recent years -- identifiability has been studied using various tools and estimation rates have been proved under various assumptions. In the non-degenerate case assuming linear independence of $\{\theta_k\}$, \cite{anandkumar2012spectral} applied the method of moments with tensor decomposition techniques to obtain parametric rates for recovering underlying topics -- they also established that under linear independent, $n=3$ words per document is enough for identifiability, if $\bar{\alpha}=\sum_k \alpha_k$ is known. Under stronger `anchor word' type assumptions, \cite{arora2012learning} developed algorithms beyond spectral decomposition of empirical tensors. However, if the frequency of anchor words was not that common, \cite{ke2024using} were able to improve on this rate with a singular value decomposition based technique. Meanwhile, under similar assumptions, \cite{bing2020fast} developed estimators and established minimax lower rates for these estimators if the topics were dense and the number of topics were unknown  or if the topic models were sparse \citep{bing2020optimal}. As an alternative approach, \cite{chen2023learning} relaxed the separability condition. Instead, they developed an estimator based on implicitly minimizing the volume of the polytope. Assuming the data is sufficiently scattered, the model is identified regardless of the number of topics or vocabulary size. Under a Bayesian perspective, \cite{NguyenPosteriorContractionPopulation2015} establish a $\left(\frac{\log n}{n} + \frac{\log m}{m} + \frac{\log m}{n}\right)^{-1/2}$ rate for estimating the topic polytope under the Hausdorff metric (thus, only topics which are extreme points of the convex hull are identified) - this approach, while requiring $n$ to increase, requires milder assumptions on the topics and applies even if the underlying Dirichlet distribution is replaced by any smooth probability distribution on $\Delta^{K-1}$, placing sufficient mass near boundary ($\alpha-$regularity). Meanwhile, \cite{lijoi2023flexible} provided theoretical insight on this type of model from the Bayesian nonparametric perspective. In contrast to the `increasing $n$' regime in \cite{NguyenPosteriorContractionPopulation2015, TangEtAlUnderstandingLimitingFactors2014}, the work in \cite{WangConvergenceRatesLatent2019} considered a `fixed $n$' setting, and studied a slightly weaker notion of finite identifiability in the LDA and established parametric rate for MLE of the topics, under the knowledge of the underlying Dirichlet distribution (their result applies if Dirichlet distribution is replaced by any symmetric probability measure $\nu_0$ on $\Delta^{K-1}$ satisfying a mild moment condition). The optimal achievable rate for estimation of all topics in the LDA with unknown $\alpha$ in terms of $m$ and $n$ (as conjectured in \cite{WangConvergenceRatesLatent2019}, the actual behavior of the MLE should be a combination of both $m$ and $n$ -- higher $n$ might deliver faster rates) is still an open question.

Algorithms for inference in topic models also have a rich literature. Variational inference was proposed for the LDA in \cite{BleiEtAlLatentDirichletAllocation2003}, and in an online fashion \citep{hoffman2010online}, however collapsed Gibbs sampler \citep{griffiths2004finding} and spectral decomposition methods \citep{anandkumar2012spectral} have also been used extensively. Furthermore, the geometric structure of the LDA has been exploited to provide much faster inference algorithms such as Voronoi Latent Admixture (VLAD) \citep{YurochkinEtAlDirichletSimplexNest2019} and Geometric Dirichlet Means (GDM) \citep{YurochkinNguyenGeometricDirichletMeans2016}, which are shown to be consistent under certain conditions. For the nonparametric topic model in \cite{BleiEtAlNestedChineseRestaurant2010}, a collapsed Gibbs sampling based inference method was proposed. On the other hand, while much faster, variational inference methods for non-parametric LDA \citep{wang2009variational,wang2011online} do not come with many theoretical guarantees. 


\subsection{Trees and Isomorphisms}
We assume familiarity with basic notions from graph theory, such as nodes/vertices, edges, directed/undirected graphs, degree, trail, cycle, tree (see, e.g., \cite{harary1969graph}). Now, we define a \textit{directed rooted tree} (DRT) as follows, which will serve in this paper as the fundamental structure controlling the topic hierarchy.

\begin{definition}
    A directed rooted tree $\cT$ is a directed graph $(\cV,\cE)$ such that the underlying undirected graph is a tree and the directed graph has a unique node $v_0$ (to be called root), with inward-degree 0. 
\end{definition}
A DRT will be denoted by $\cT=(\cV,\cE,v_0)$. We remark that this tree is an abstract graph and has no connection at this point to the data corpus. Without loss of generality, we can take $\cV=[K]$, where $K=|\cV|$ is the number of nodes and we shall use this in the remainder of the paper. Define a \textit{maximal path} of $\cT$ as $\varphi=(v_0,v_1,\dots,v_{L-1})$, where $(v_i,v_{i+1})\in\cE$ and $v_L$ is a leaf of $\cT$, i.e., any path of the tree starting at the root and ending at a leaf (nodes with 0 outward degree) -- we shall simply call these \textit{paths} by slight abuse of terminology. For such a path $\varphi$, we say $L$ is its \textit{length}. Let $I$ be the number of leaves of $\cT$, then clearly we have $I$ paths -- let $J_1,\dots,J_I$ be the lengths of these paths. Define the \textit{size} of $\cT$ as the tuple $(I,\bJ, K)$, where $\bJ=\text{set}(J_1,\dots,J_I)$ and referred to as \textit{depth set} of $\cT$. Let $\Phi(\cT)$ denote the set of all paths of $\cT$. Finally, for any function $f:\cV\to \Omega$ (for any arbitrary set $\Omega$, e.g., $\bbR$) and a path $\varphi=(v_0,\dots,v_{L-1})\in \Phi(\cT)$, we write $f(\varphi):=(f(v_0),f(v_1),\dots,f(v_{L-1}))\in \Omega^L$. Next, we state the definition of graph-isomorphism, specifically for DRTs. 

 \begin{definition}
     $\cT=(\cV,\cE,v_0)$ and $\cT'=(\cV', \cE', v_0')$ are isomorphic if and only if there exists a bijection $\sigma: \cV\to\cV'$ such that 
     $$(u,v)\in \cE \iff (\sigma(u), \sigma(v)) \in \cE'.$$
 \end{definition}
If $\cT$ and $\cT'$ are isomorphic, we write $\cT\cong \cT'$. Notice that the definition implies that $\sigma(v_0)=v_0'$ and that any such bijection (if exists) must map a leaf of $\cT$ to a leaf of $\cT'$. We say that DRTs $\cT$ and $\cT'$ have the same \textit{structure} if they are isomorphic. Since isomorphism is an equivalence relation on the space of DRTs, we wish to think of the equivalence classes as representing a particular `tree structure'. We also note that for $\cT\cong \cT'$, the bijection is typically not unique. 

Our probabilistic model, to be defined in the following section, would associate each document to a particular path of the DRT and then, that document would be following a latent Dirichlet allocation model using topics only from that path. While this will be made more precise in the next section, we make a quick remark that the paths of the DRT play important role in our model. Two paths might share a few nodes, meaning that topics associated with these nodes are shared between documents that are assigned to these paths. This topic-sharing structure as induced by the DRT is connected solely to the structure of the DRT. In particular, we wish to identify the tree structure solely from the collection of sets of nodes from each path. The following lemma shows that two non-isomorphic DRTs cannot give rise to the same collection of paths. The proof of this result is deferred to the Appendix, where we also provide further details about DRTs.

\begin{lemma}\label{lemma:tree}
    Given two DRTs $\cT, \cT'$ with the same $\cV$, we have
    $$\{\text{set}(\varphi)\mid \varphi \text{ is a path in }\cT\} = \{\text{set}(\varphi')\mid \varphi' \text{ is a path in }\cT'\} \iff \cT \cong \cT'.$$
    Moreover, $\cT=\cT'$ if additionally, either $\cT$ or $\cT'$ satisfies the following: every element $v$ in $V$ except the leaves has at least 2 children.
\end{lemma}

\section{Tree-Directed Topic Models}\label{sec:model}
\subsection{Model formulation}
Although it is simple to describe the probabilistic model rather quickly, explicating their likelihood structure and the underlying geometry requires some effort. In particular, to incorporate the DRT structure into the probabilistic model, we need to following concept.

\begin{definition}
    Given a DRT, $\cT=(\cV,\cE,v_0)$, any injection map, $\rho:\cV\to \Delta^{V-1}$, is called a topic map on $\cT$.
\end{definition}

Given a DRT, $\cT=(\cV,\cE,v_0)$, of size $(I,\bJ,K)$ and a topic map $\rho$ on it, we are ready to describe the model. Denote a generic document with $n$ words by $\bX_{[n]}$ and the corpus with $m$ such documents by $\bX^{[m]}_{[n]}$, with $X_{\ell j}, \ell\in[m], j\in[n]$ representing the $j$th word in the $i$th document. Fix an enumeration of the paths of $\cT$, say $\varphi_1,\dots,\varphi_I$, and let $\pi:\Phi(\cT)\to[0,1]$ be such that $\pi(\varphi_i)>0$ for all $i\in[I]$ and $\sum_i \pi(\varphi_i)=1$. Thus, if we abuse the notation to take $\pi_i=\pi(\varphi_i)$ and $\pi=(\pi_1,\dots,\pi_I)^\top\in\bbR^I$, then $\pi\in\Delta^{I-1}$ represents the path probability parameter. First, for each document, a path of $\cT$ is chosen according to $\pi$, i.e.,
$$c_1,\dots,c_m\overset{iid}{\sim} \text{Cat}(\pi).$$
Here, $c_\ell\in[I]$ is a discrete random variable indicating the path chosen for document $\bX_\ell$. In other words, $c_\ell=i$ suggests that the path $\varphi_i$ is chosen for the $\ell-$th document. Now, given the path associated with a document, words in that document are generated in the same way as the LDA, using only topics associated with that path under the topic map $\rho$. More precisely, let $\Theta_i = \rho(\varphi_i) \in \bbR^{J_i\times V}$ such that the rows are the topics associated to the nodes along this path. Recall that $J_i$ denotes the length of path $\varphi_i$ and $\alpha_i\in\bbR_+, i\in[I]$ the scalar parameters for the Dirichlet distribution. Let the $k$th row of $\Theta_i$ be denoted by $\Theta_{ik}$. Then, the remainder of the generate model specification  for words in document $\ell$ is:
\begin{align*}
    \beta_\ell | c_\ell = i &\sim \text{Dir}_{J_{i}}(\alpha_i) \\
    z_{\ell,j} | \beta_{\ell} &\overset{iid}{\sim} \text{Cat}(\beta_{\ell}), \quad j\in[n] \\
    x_{\ell,j} | \Theta_i, z_{\ell,j}=k &\sim \text{Cat}(\Theta_{i,k}).
\end{align*}
The above four lines defining our model can be written succinctly and equivalently as:
\begin{equation*}
    c_\ell\sim \text{Cat}(\pi), \; \text{and} \; \bX_\ell|c_\ell=i \sim \text{LDA}(\Theta_i, \alpha_i).
\end{equation*}
    We remark that $J_i$ could potentially be different and as a result, the Dirichlet distributions for different paths might not even be same dimensional. If all $J_i$ are equal to a common value $J$, then we could use $\alpha_1,\dots,\alpha_I\in\bbR^{J}$ as parameters for Dirichlet distributions without assuming symmetric. More generally, we could assume $\alpha_i\in \bbR^{J_i}_+$ as parameters. However, in this paper, we are more interested in the identifiability and rates for the topics (defined via the topic map $\rho$). Thus, we choose a simple symmetric Dirichlet distribution with a single parameter as the basis for the model, which we denote as $\text{Dir}_K$ . It is important to note from the outset that the $\Theta_i$'s share constraints imposed by the underlying DRT $\cT$. For instance,  $\Theta_{i1}=\Theta_{i'1}$ for any $i,i'\in[I]$, since it is always $\rho(v_0)$ (topic associated to the root). Thus, the "root topic" is shared across all the paths. On the other hand, $\Theta_{i J_i} \neq \Theta_{i' J_{i'}}$ for any two $i\neq i'$, since each leaf of $\cT$ appears in a unique path and $\rho$ is injective. This structure will be revisited shortly.

\begin{figure}
    \centering
    \includegraphics[width=0.9\linewidth, trim={2cm 8cm 1cm 12cm}, clip]{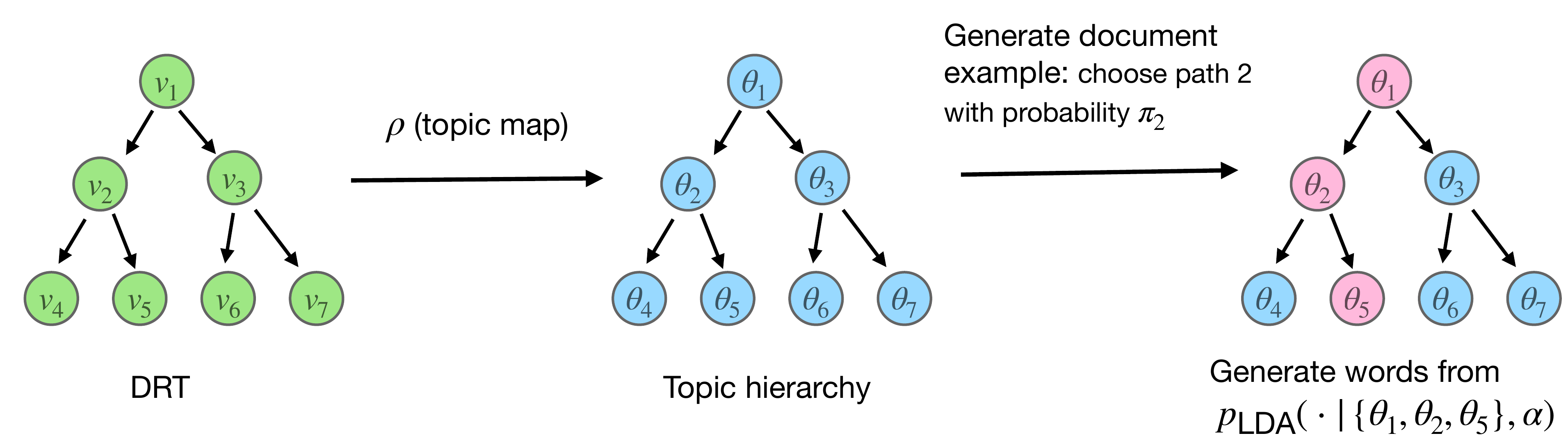}
    \caption{The basic ingredients of the tree-directed topic model. The left figure shows the underlying DRT, an abstract DRT with a single root $v_1$, the middle shows topic hierarchy resulting from the DRT and a topic map, $\rho(v_i)=:\theta_i$, and the right shows an example of the generation of a document given a path chosen categorically with probability, $\pi$.} 
    \label{fig:model}
\end{figure}

Summarizing, we assume that the documents are i.i.d. coming from the above generative model. The words in each document are conditionally i.i.d. and hence exchangeable. Marginalizing out the discrete latent variable $z$, we have the equivalent form of the model as
\begin{align*}
    c_\ell &\sim \text{Cat}(\pi), \\
    \beta_\ell | c_\ell=i &\sim \text{Dir}_{J_i}(\alpha_i), \\
    X_{\ell j} | \beta_\ell, \bTheta &\sim \text{Cat}\left(\sum_{k\in [J_i]} \beta_{\ell k} \Theta_{i,k}\right).
\end{align*}
It follows that the density of a document $\doc$ with respect to the count measure on $[V]^n$ is
\begin{equation}\label{eq:document density1}
    p(\doc | \cT, \rho, \pi, \alpha) = \sum_{i\in[I]} \pi_i p_{\text{LDA}}(\doc | \Theta_i, \alpha_i) = \sum_{i\in[I]} \pi_i \int_{\Delta^{J_i-1}} \prod_{v\in[V]} \innerprod{\Theta_{i \cdot v}, \beta}^{m_v} \text{Dir}_{J_i, \alpha_i}(d\beta),
\end{equation}
where a particular enumeration of the paths $\varphi_i, i\in[I]$ is (arbitrarily) fixed, $\Theta_i = \rho(\varphi_i)$ and $\pi_i, \alpha_i$ are the probability associated to $\varphi_i$ and Dirichlet parameter for the $i$th path respectively, and $m_v = \sum_j 1(X_j = v)$ is the count of the number of occurrences of word $v\in[V]$ in the document. We note that $\Theta_{i\cdot v} = (\Theta_{i1v},\dots,\Theta_{i J_i v})\in\bbR^{J_i}$ is the $v$-th column of $\Theta_i$. We denote this density by $p_{n,\omega}$, where $\omega$ contains all the relevant parameters and denote the corresponding probability measure on $[V]^n$ by $P_{n,\omega}$.
Thus, based on the above display, our model can be viewed as a mixture of the LDA induced distributions, but the LDA components are constrained via shared topics imposed by the latent DRT $\cT$ and topic map $\rho$ to be learned. It is this sharing that makes the model interesting and (as we shall see) the latent topic hierarchy (tree) learnable and interpretable.

\begin{remark}
A particular topic tree leads to a particular sharing structure of the associated topics across the documents. The following illustration will be helpful. For the topic tree shown in Figure \ref{fig:model}, each document only uses 3 topics from the collection of 7 total topics. However, documents cannot use \textit{any} 3 topics -- each tuple of valid 3 topics arises from the paths of the hierarchy. Thus, the topic $\theta_1$ is carried by all documents -- this root topic can be seen as a topic placing high mass on words that occur frequently across all documents in the corpus (e.g., stop words, articles etc.). Meanwhile $\theta_2$ is carried by a certain proportion of documents. On the other hand, if a document uses the topic $\theta_5$, then it has to also use the topics $\theta_2, \theta_1$. This captures the idea that a document about soccer (a type of sport) must also include a more concrete topic about sports (and also stop words) -- this general topic about sports would place higher mass on words that occur frequently in all sports-related documents. A document about basketball (another type of sport) will share the sports topic as well as stop-words topic, while differing from the previous document in one topic -- thus the topics about soccer and basketball are sub-topics of the sports topic, and capture specialized semantics within the particular category of sports. 
\end{remark}

\begin{figure}
    \centering
\includegraphics[clip, trim=3.3cm 0.5cm 2cm 1.2cm, width=0.98\textwidth]{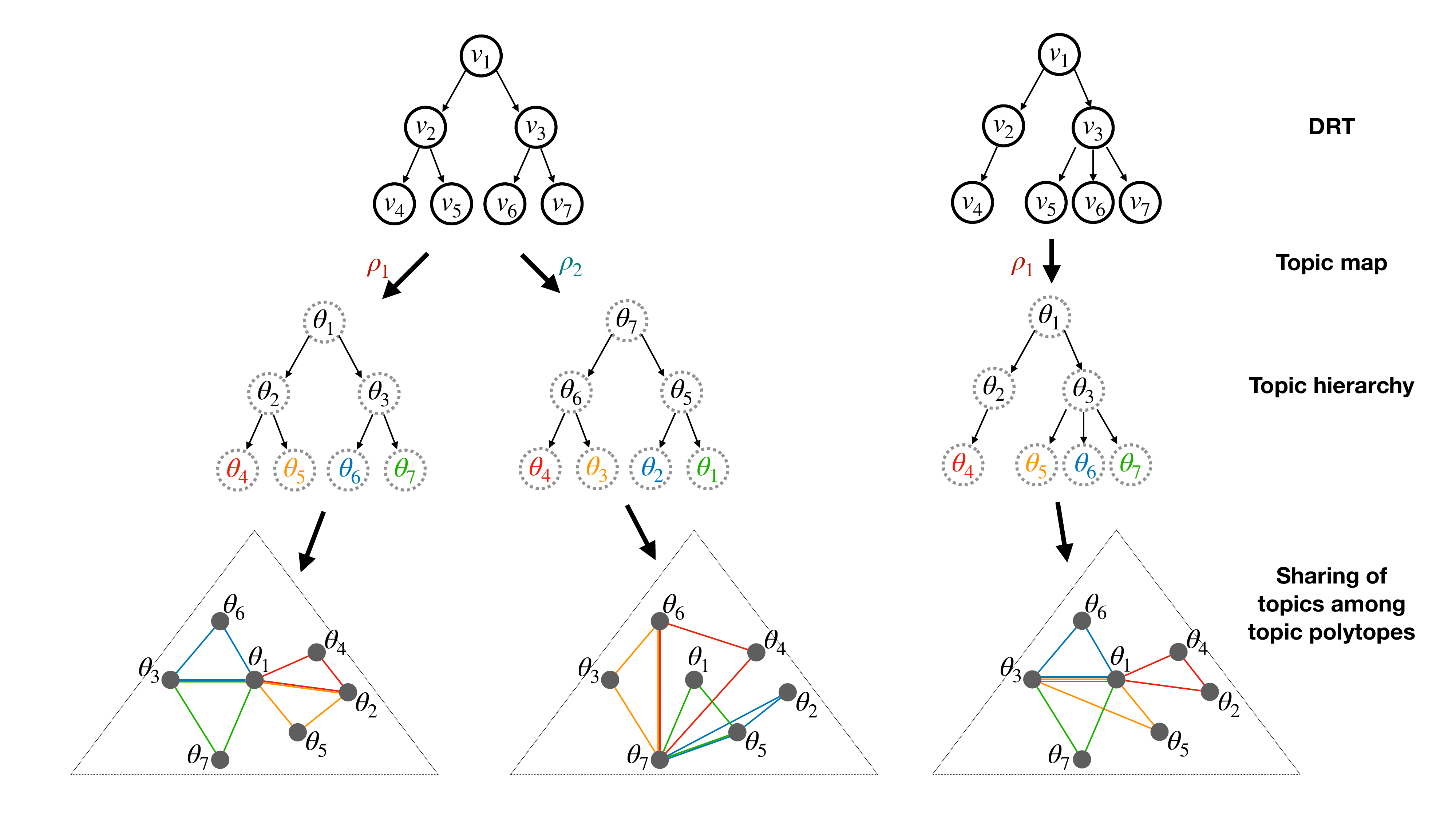}
    \caption{Illustration of DRT and topic hierarchy with fixed set of topics $\Theta$ following the example in text: (top row) two DRT of size $(I=4,J=2,K=7)$; topic maps $\rho_1:v_\ell\to \theta_\ell$ and $\rho_2:v_\ell\to\theta_{8-\ell}$ induce different topic hierarchies (middle row); (third row) shows how these topic hierarchies induce sharing of topics across documents, the topics are the black dots inside the vocabulary simplex (big bounding black triangle here). Each triangle (in different colors) represents a component topic polytope corresponding to a particular path in the tree (the respective leaf is highlighted with the same color). The size of the tree has the following effects: $I=4$ is the number of polytopes, $J=2$ is the dimension of each polytope and $K=7$ is the total number of topics.}
    \label{fig:DRT_example}
\end{figure}

\begin{remark} 
We make the following additional remarks regarding the model formulation.
\begin{enumerate}
    \item The model size is connected to the tree size as follows. If $\cT$ has size $(I,\bJ,K)$, then there are $K$ total topics for the model and $I$ components (seen as a mixture as shown above). $\bJ$ captures the dimensionality of each of the component LDA models. 
    \item The LDA is obviously a special case, with a DRT $\cT$ of size $I=1, J=K$ (a linear tree).
    \item The Dirichlet distribution plays no special role in this paper other than in designing the Gibbs sampler. Other parametric families of identifiable distributions on $\Delta^{J_i-1}$ can be used in place of the Dirichlet. This particular family is considered because of its popularity in topic models and also for ease in designing inference algorithms. 
   
\end{enumerate}
\end{remark}

\subsection{Model's geometry}\label{sec:model-geom}
Consider a particular component of the model, say $i$, corresponding to path $\varphi_i$ of $\cT$. Let $G_i = \text{Dir}_{J_i}(\alpha_i)_{\#} L_i$, where $L_i:\Delta^{J_i-1}\to \Delta^{V-1}$ is defined as $L_i(\beta) = \Theta_i^\top \beta$ for $\beta\in\Delta^{J_i-1}$. $G_i$ is the push-forward measure on $\Delta^{V-1}$ of a Dirichlet distribution under the linear transformation $L_i$. Note that $G_i$ is supported on $\cS_i:= \conv \Theta_i $, the component topic polytope corresponding to the path $\varphi_i$. The Dirichlet distribution for $\beta\in\Delta^{J_i-1}$ induces the distribution $G_i$ for $\eta=\sum_{k\in[J_i]} \beta_k \Theta_{ik}$. If $J_i \leq V$,  $\Theta_i$ are almost surely linearly independent, in this case $L_i$ has rank $J_i-1$ and $G_i$ admits the following density with respect to the $J_i-1$ dimensional Hausdorff measure $\cH^{J_i-1}$ on $\Delta^{V-1}$
$$g_i(\eta) = \text{Dir}_{J_i}(L_i^{-1}(\eta)|\alpha_i) J(L)^{-1},$$
where $J(L)$ is the Jacobian of the linear map. On the other hand, if $J_i\geq V$, then $L_i$ is generally $(V-1)-$ranked and $G_i$ has the following density with respect to the $(J_i-V)$ dimensional Hausdorff measure on $\bbR^V$, given by 
$$g_i(\eta) = \int_{L_i^{-1}(\eta)} \text{Dir}_{J_i}(\beta|\alpha_i) J(L)^{-1} \cH^{J_i-V}(d\beta).$$
Using change of variable, we can write $p_{n,\omega}$ in Eq. \eqref{eq:document density1} in the following alternate way, in terms of these latent measures $G_1,\dots,G_I$:
\begin{equation}\label{eq: document density2}
    p(\doc|\cT,\rho,\pi,\alpha) = \sum_{i\in[I]} \pi_i \int_{\cS_i} \prod_{v\in[V]} \eta_v^{m_v} G_i(d\eta) = \int_{\cup_i \cS_i} \prod_v \eta_v^{m_v} \left(\sum_{i\in[I]} \pi_i G_i\right)(d\eta).
\end{equation}

This can be visualized as follows. For each path $\varphi_i$ of $\cT_i$, there is a component polytope $\cS_i$ (formed as the convex hull of topics only from this path) and the model imposes a continuous distribution $G_i$ supported on $\cS_i$ as described above. For each document, first a component is chosen with probability $\pi$, given this choice, a document-specific $\eta\in \cup_i \cS_i \subset \Delta^{V-1}$ is chosen from the corresponding $G_i$, conditional on which, the words in the document are i.i.d. categorical. This equivalent data generating model for a single document is expressed below
\begin{align*}
    c | \pi &\sim \text{Cat}(\pi) \\
    \eta | c=i, G_1,\dots,G_I &\sim G_i \\
    X_{j} | \eta &\sim \text{Cat}(\eta).
\end{align*}

\begin{remark}
    The above representation suggests a connection to mixture models with a key distinction. Standard
    mixture models are represented by $\sum_i \pi_i f(x|\theta_i)$, where $f$ is some probability kernel's density function (e.g. Gaussian mixture model corresponds to a Gaussian kernel, $\theta_i=(\mu_i,\Sigma_i)$ capture the parameters of the components). The idea of a latent mixing measure is used in the theoretical analysis for the standard mixture model \citep{nguyen2013convergence}, this measure is $G=\sum_i \pi_i \delta_{\theta_i}$, a discrete probability measure capturing all information about the unknown parameters. In terms of $G$, the model can be written as $\sum_i \pi_i f(x|\theta_i) = \int f(x|\eta) G(d\eta)$. Based on the representation in the above display in Equation \eqref{eq: document density2}, in our case, we can consider $G=\sum_i \pi_i G_i$ as the latent mixing measure and the kernel being a product of multinomials (product comes for the $n$ words in the document). However, in our case, this measure is neither discrete nor has a density with respect to Lebesgue measure on $\Delta^{V-1}$ -- this is because the $G_i$ are supported on different affine spaces of potentially different dimensions. 
\end{remark} 

\begin{example} 
We illustrate the model and associated notions with a concrete example, illustrated in Figure \ref{fig:DRT_example}. Consider a fixed set of $K=7$ topics $\theta_1,\dots,\theta_K$ (black dots inside the vocabulary simplex $\Delta^{V-1}$ with $V=3$, shown as the big triangles in the bottom row of Figure \ref{fig:DRT_example}). These topics can lead to very different probability distributions for the corpus under our model, based on the DRT and topic map. We demonstrate two such DRTs (first row in Figure). Let us walk through the model under the left DRT with topic map $\rho_1$ (which maps $v_i\mapsto \theta_i$). There are 4 paths in the DRT, and the collections of topics along these paths are $\{\{\theta_1,\theta_2,\theta_4\}, \{\theta_1,\theta_2,\theta_5\}, \{\theta_1,\theta_3,\theta_6\}, \{\theta_1,\theta_3,\theta_7\}\}$ (leaf nodes colored red, yellow, blue, green respectively). Each of these paths induce a measure $G_i$ supported on a component polytope (corresponding colored triangle inside the vocabulary simplex in the bottom row). The DRT and topic map together impose the topic hierarchy, which in turn, induce a sharing structure of the topics among the components. Clearly, the topic $\theta_1$ is shared by all components, while $\theta_2$ is only shared between 2 components. Note the effect that a different topic map ($\rho_2$ in the figure) has on this sharing structure of the topics among the components. And clearly, a different DRT has a totally different topic hierarchy and hence sharing structure, e.g. for the DRT on the right, $\theta_3$ is shared by 3 component polytopes (no topic has this behavior for the DRT on the left). For each document, first one of these component polytopes is chosen. Then $\eta$ is generated from the corresponding $G_i$ and conditionally, words in that document are generated from this $\eta$. Thus, the sharing structure of topics under the DRT and topic map, also induce how topics are shared across documents — for two documents, one associated to the first path and one to the second in the left-most example in Figure \ref{fig:DRT_example}, they both use $\{\theta_1,\theta_2\}$, however the former uses topic $\theta_4$, while the latter uses $\theta_5$ (which are unique to these components).
\end{example}

Finally, the distribution of a corpus of i.i.d. documents $\corpus=(\doc^1,\dots,\doc^m)$ can be expressed by
$$p(\corpus | \cT, \rho, \pi,\alpha) = \prod_{\ell\in [m]} p(\doc^\ell |\cT,\rho,\pi,\alpha).$$
For brevity, for the rest of the paper this density is also denoted by $p_{\corpus|\omega}$ where $\omega$ captures all the parameters and the corresponding probability measure on $[V]^{m\times n}$ is denoted by $P_{\corpus|\omega}$. Based on the geometric viewpoint and conditional i.i.d. nature of the words, if we let $\hat{\eta}_\ell\in\Delta^{V-1}$ denote the document mean for $\doc^\ell$, i.e., $\hat{\eta}_{\ell v} = \sum_{j\in[n]} 1(X_{\ell j}=v)/n$, then $\hat{\eta}_\ell \to \eta \sim G$ almost surely as $n\to\infty$. We shall exploit this together with the representation of $G$ as a convex combination of low-dimensional probability measures in the following section.

\section{Identifiability}\label{sec:identifiability}
We are ready to investigate the identifiability of the latent structures arising in the tree-directed topic models. Our approach exploits the underlying geometry of the model and essentially shows that any potential latent measure $G=\sum_i \pi_i G_i$ can be uniquely decomposed into the $G_i'$s, from which the DRT and the topic map are identified uniquely. This approach requires that the document length is sufficiently large, and the asymptotic regime of $n\to\infty$, while having the advantage of requiring only very mild assumptions on topics and other related quantities. A possible alternative approach to study identifiability would be through the method of moments (up to a certain order). We pursue the former approach by focusing on the geometry of a collection of polytopes. 

\subsection{Topic hierarchies}
A primary object of inferential interest is the \textit{topic hierarchy}, which is composed of both the actual position of the topics in the vocabulary simplex as well as the way they are shared across components --- up to this point, we have only defined isomorphism between DRTs. We begin with a brief discussion of the notion of \textit{similar topic hierarchy} here, which would be important in the sequel. As discussed in Section \ref{sec:model-geom}, given a DRT $\cT=(\cV,\cE,v_0)$ and a topic map $\rho$, one obtains an induced topic tree $\tilde{\cT}=(\tilde{\cV}, \tilde{\cE}, \theta_0)$ where $\tilde{\cV}=(\rho(u) : u\in\cV\}$, $\theta_0=\rho(v_0)$ and $(\rho(u),\rho(u'))\in\tilde{\cE}\iff (u,u')\in \cE$. In Figure \ref{fig:model}, the tree in the middle is the induced topic tree arising due to the DRT (left) and the topic map $\rho$. Geometrically, the component polytopes of the model have precisely $\theta\in \tilde{\cV}$ as the extreme points and $\theta,\theta'\in\tilde{\cE}$ are extreme points of the same polytope iff $\theta, \theta'$ are both on some maximal path of $\tilde{\cT}$. Given $(\cT,\rho)$, let the induced topic tree be denoted by $\xi(\cT,\rho)$. The topic hierarchy is formalized as follows:
\begin{definition}[Identical Topic Hierarchies]\label{def:topic_hierarchy1}
     The two sets of DRT and associated topic maps $(\cT, \rho)$ and $(\cT', \rho')$ are said to have identical \textit{topic hierarchy} if $\xi(\cT,\rho) \cong \xi(\cT',\rho')$ and there is a unique isomorphism $\sigma:\xi(\cT,\rho)\leftrightarrow \xi(\cT',\rho')$ such that $\theta = \sigma(\theta)$ for all nodes $\theta$ in $\xi(\cT,\rho)$.
\end{definition}

\begin{figure}
    \centering
\includegraphics[clip, trim=3.3cm 9cm 2cm 6cm, width=\linewidth]{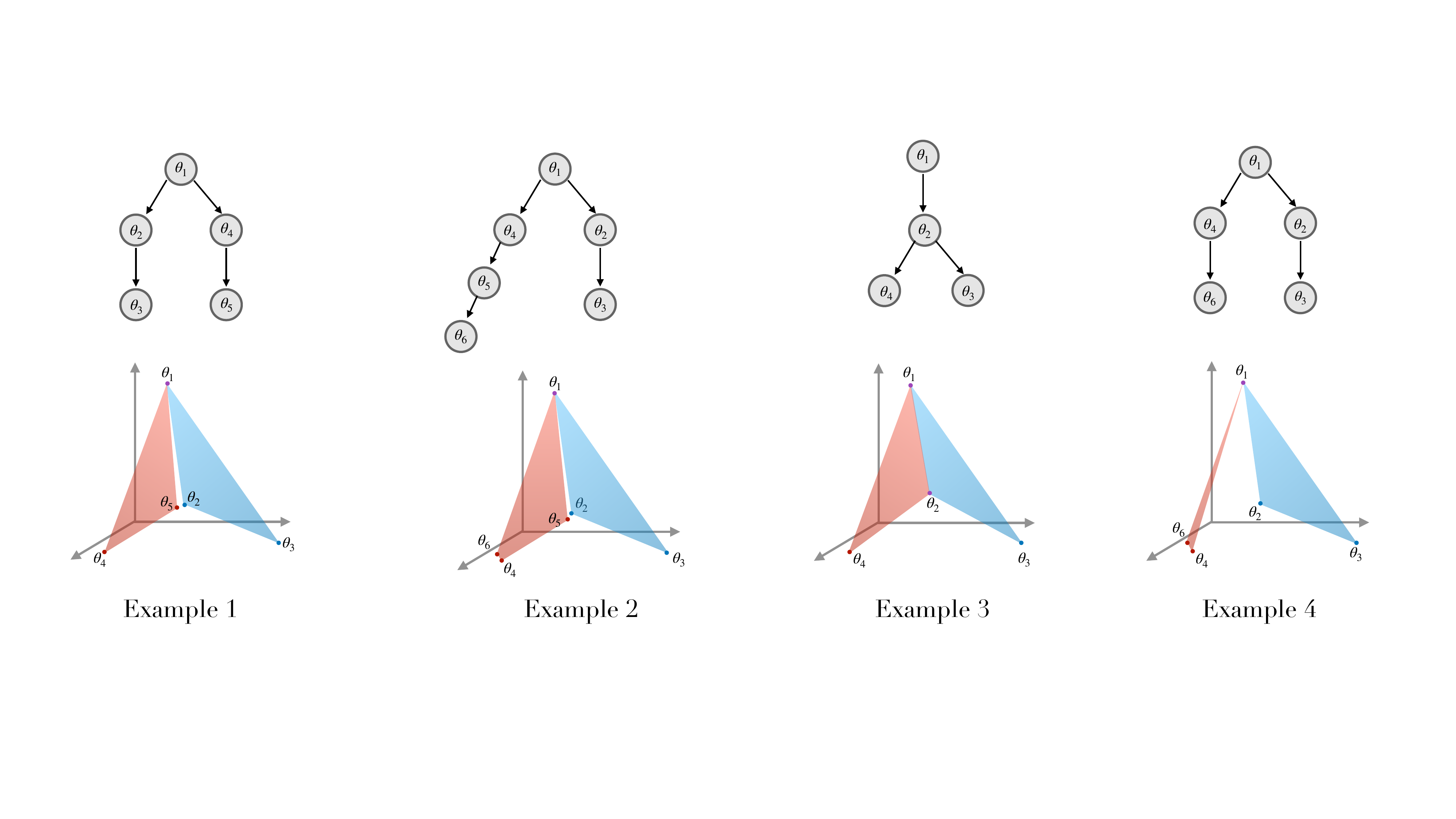}
    \caption{Illustration of topic hierarchies. The top row is the topic tree induced by an underlying DRT $\cT$ and a topic map $\rho:v_i\to\theta_i$. The bottom row shows the geometry of the component polytopes resulting due to the corresponding topic tree above. Examples 1, 2 and 3 are approximately close topic hierarchies, while example 4 is not. Topic trees need not be isomorphic while isomorphic topic trees do not necessarily yield the same topic hierarchy.}
    \label{fig:topic_hierarchy}
\end{figure}

\begin{remark}
The condition in Definition \ref{def:topic_hierarchy1} conveys the intuition that the topic (tree) hierarchies are precisely induced by the DRT and the topic maps. Requiring only isomorphic topic trees is insufficient since they capture only the sharing structure, not the actual topics. Note that while this condition implies that $\cT\cong\cT'$, i.e., the underlying DRTs are isomorphic, Figure \ref{fig:DRT_example} gives examples where the same DRT might lead to different topic hierarchies, even when the set of topics are same. 
\end{remark}

The remark suggests that this notion of topic hierarchy is rigid and does not allow flexibility. Consider the examples in Figure \ref{fig:topic_hierarchy}, and assume $\norm{\theta_2-\theta_5}<\epsilon$ and $\norm{\theta_4-\theta_6}<\epsilon$ in the figure. The first three examples show parameters which all lead to two components (all the polytopes are embedded in the vocabulary simplex $\Delta^{V-1}$), each being a triangle (or almost) and sharing an edge (or almost) and moreover, the positions of the corresponding topics are similar. This motivates our next definition.

\begin{definition}[Approximate Topic Hierarchies]\label{def:topic_hierarchy2}
    The two sets of DRT and associated topic maps $(\cT, \rho)$ and $(\cT',\rho')$ are said to share $\epsilon-$approximate topic hierarchies, for some $\epsilon>0$, if there exist partitions $\cP$ of $\tilde{\cV}$ and $\cP'$ of $\tilde{\cV}'$ with $|\cP|=|\cP'|$, where $\tilde{\cV}$ (resp. $\tilde{\cV}'$) is the set of nodes in $\xi(\cT,\rho)$ (resp. $\xi(\cT',\rho')$), such that $\exists \text{ bijection } \tau:\cP \leftrightarrow \cP'$ satisfying $\theta\in P\in \cP, \theta'\in \tau(P) \in \cP' \Rightarrow \norm{\theta - \theta'}<\epsilon$. Moreover, for every path $(\theta_1,\dots,\theta_J)$ in $\xi(\cT,\rho)$ such that $\theta_j\in P_j\in\cP$, there exists a path $(\theta_1'\dots,\theta_{J'}')$ such that $\theta_{j'}'\in \cup_{j\in[J]} \tau(P_j)$ for every $j'\in [J']$, and vice versa. 
\end{definition}

To illustrate the idea of the partition as described in the preceding definition, consider again Figure \ref{fig:topic_hierarchy}. We argue that the topic hierarchies in examples 2 and 3 are approximately close. Consider the partition $\cP = \{\{\theta_1\}, \{\theta_4,\theta_6\}, \{\theta_2,\theta_5\}, \{\theta_3\}\}$ for example 2 and $\cP'=\{\{\theta_1\}, \{\theta_4\}, \{\theta_2\}, \{\theta_3\}\}$ for example 2. The bijection $\{\theta_1\}\leftrightarrow\{\theta_1\}$, $\{\theta_4,\theta_6\}\leftrightarrow\{\theta_4\}, \{\theta_2,\theta_5\}\leftrightarrow\{\theta_2\}$ and $\{\theta_3\}\leftrightarrow\{\theta_3\}$ between $\cP$ and $\cP'$ satisfies the condition in the preceding definition, which shows that these two topic hierarchies are approximately close. Similarly, the topic hierarchies from example 1 and example 2 in Figure \ref{fig:topic_hierarchy} are also approximately close. We remark that while having the same topic hierarchy implies isomorphism of the underlying DRTs, this is relaxed for the $\epsilon-$approximate topic hierarchies: clearly, the DRTs in examples 1 and 3 are neither isomorphic, nor is one a sub-tree of the other. On the other hand, the topic hierarchy in example 4 is not approximately close to any of those in the former three examples. To see this, comparing to say, example 2, consider the partition $\{\{\theta_1\}, \{\theta_4,\theta_6\}, \{\theta_2\}, \{\theta_3\}\}$ for example 4. While it has a bijection with $\cP$ (for example 2) satisfying the first part of Definition \ref{def:topic_hierarchy2}, it violates the second part. The path $(\theta_1,\theta_4, \theta_6)$ in example 4 uses only two of the sets in the partition, while each path in the topic hierarchy in example 2 uses three distinct sets from its partition.

\begin{remark}
One of the simplest forms of ambiguity (non-identifiability) in our model arises due to the permutation of labels (aka topic map), subject to the DRT restrictions. This is similar in spirit to the issue of label switching for mixture models. However, in this model an arbitrary permutation of the labels might not lead to the same model — e.g., consider example 1 in Figure \ref{fig:topic_hierarchy}: a random permutation of the labels of the topics might remove $\theta_1$ from the root, which would lead to a very different sharing structure. For mixture models, this problem is dealt with by considering a suitable metric (e.g. Wasserstein metric on the latent mixing measures) which is agnostic to the ordering of the components. For ours, we construct a suitable metric for this case, which considers all such permutations respecting the structure imposed by the DRTs. Some non-identifiable examples, which cannot be solved by permuting the labels retaining the topic hierarchy are discussed at the end of Section \ref{sec: identifiability2}.
\end{remark}

\subsection{Metrics on topic hierarchies}\label{sec: metrics}

Define $\mathfrak{T}(I,\bJ,K)$ to be the set of DRTs $\cT=(\cV,\cE,v_0)$ of size $(I,\bJ,K)$ with $\cV=[K]$. Also let $\mathfrak{T}(I)=\cup_{\bJ, K} \mathfrak{T}(I,\bJ,K)$ and $\mathfrak{T}(I,K)=\cup_{\bJ} \mathfrak{T}(I,\bJ,K)$ stand for the set of possible DRTs with $I$ leaves and with $I$ leaves and $K$ nodes, respectively. Finally, $\mathfrak{T}(I,J)$ (for $J\in \bbN$) denotes the space of all DRTs with $I$ leaves and each of whose paths has the same length $J$. 
Given a DRT $\cT$, let $\mathfrak{R}(\cT)=\{\rho:\cV\to \Delta^{V-1}\mid \rho \text{ is injective}, \cV \text{ is the set of nodes of } \cT\}$ be the space of possible topic maps associated with $\cT$.  

The collection of all parameters $\omega$ of the tree-directed topic model is given by
$$\Omega(I,\bJ,K) = \{(\cT,\rho,\pi): \cT\in\mathfrak{T}(I,\bJ,K), \rho\in \mathfrak{R}(\cT), \pi:\Phi(\cT)\to [0,1] \ni (\pi(\varphi) : \varphi\in \Phi(\cT))\in\Delta^{I-1}\},$$
the space of tuples $(\cT,\rho,\pi)$ consisting of a DRT $\cT$ of size $(I,\bJ,K)$, topic map $\rho$ on $\cT$ and path-probability $\pi$ (seen as a map on the set of paths of $\cT$). Similarly, we take $\Omega(I)$ and $\Omega(I,K)$ to be the set of tuples subject to the appropriate restrictions on $\cT$. Finally, we define
$$\Omega(\cT) = \{(\rho,\pi): \rho\in \mathfrak{R}(\cT), \pi:\Phi(\cT)\to [0,1] \ni (\pi(\varphi) : \varphi\in \Phi(\cT))\in\Delta^{I-1}, \text{ where } I=|\Phi(\cT)|\}$$
to be the space of topic maps and path probabilities defined on a fixed DRT $\cT$. Viewing $\pi$ as a map on $\Phi(\cT)$ addresses the ambiguity in the enumeration of the paths of $\cT$.

Given $\omega\in \Omega(I,\bJ,K)$ (or any of $\Omega(I),\Omega(I,K), \Omega(\cT)$), let $p_{n,\omega,\alpha}$ be the probability density (with respect to count measure on $[V]^n$) of a length-$n$ document $\doc\in[V]^n$ drawn from the model described in Section \ref{sec:model}, and $p_{n,\omega,\alpha}$ takes the expression in Eq. \eqref{eq:document density1}, where $\alpha:\Phi(\cT)\to\bbR_+$, represents the Dirichlet parameter associated with each path of the DRT. Note that Eq. \eqref{eq:document density1} is agnostic to the choice of enumeration $\varphi_1,\dots,\varphi_I$ of the paths of the DRT $\cT$, as long as $\pi$ and $\alpha$ are defined as maps on the set of paths. $P_{n,\omega,\alpha}$ is the associated probability measure whose density is $p_{n,\omega,\alpha}$. When $\alpha_i$'s are equal to a common value $\alpha\in\bbR_+$ and is known, we drop the $\alpha$ from the subscript. We are now ready to define a metric on the parameter space just defined. 
\begin{definition}
    Given $\omega=(\cT,\rho,\pi), \omega'=(\cT',\rho',\pi')\in\Omega(I)$, define $d_{\bbH+}:\Omega(I)\times\Omega(I)\to\bbR$ as 
    \begin{equation}\label{def: augmented tree directed Hausdorff metric}
        d_{\bbH+}(\omega,\omega') = \min_{\sigma\in \bbS(I)} \sum_{i\in[I]} \left[d_{\cH}\left(\cS_i, \cS_{\sigma(i)}'\right) + |\pi_i - \pi_{\sigma(i)}'|\right]
    \end{equation}
    where $\pi_i, \cS_i$ are the path-probability and component topic polytope associated with path $i$, given a particular enumeration $\varphi_1,\dots,\varphi_I$ of the paths of $\cT$ (and similarly for $\pi_i', \cS'_i$) and $d_{\cH}$ is the Hausdorff metric on the space of non-empty subsets of $\Delta^{V-1}$, and $\bbS(I)$ is the set of permutations of $[I]$.
\end{definition}
Note that the above definition does not depend on the particular enumerations of paths chosen for $\cT, \cT'$. Thus, the defined metric seeks to find an optimal bijection between the paths of the involved DRTs which minimizes the total distance between the corresponding topic polytopes (through the Hausdorff metric) and the path probabilities. The restriction of $d_{\bbH+}$ to $\Omega(I,K)\subset \Omega(I)$ and $\Omega(\cT)$ (seen as a subset of $\Omega(|\Phi(\cT)|)$) is also denoted by $d_{\bbH+}$. We call $d_{\bbH+}$ as the \textbf{augmented tree-directed Hausdorff metric}. It is clear that $d_{\bbH+}$ is non-negative and symmetric.  The following lemma also guarantees that it satisfies the triangle inequality.

\begin{lemma}\label{lemma: dh+ triangle}
    For any $\omega_1,\omega_2,\omega_3\in\Omega(I)$, there holds
    $d_{\bbH+}(\omega_1,\omega_3) \leq d_{\bbH+}(\omega_1, \omega_2) + d_{\bbH+}(\omega_2,\omega_2).$
\end{lemma}

Although the augmented tree-directed Hausdorff metric satisfies non-negativity, symmetry and the triangle inequality, $d_{\bbH+}(\omega,\omega')=0$ does not necessarily imply that $\omega=\omega'$. Nonetheless, this metric is useful for capturing the latent structure in terms of the topic hierarchy — in particular, we will establish that $d_{\bbH+}(\omega,\omega')=0$ iff $\omega$ and $\omega'$ induce the same topic hierarchy. The following mild conditions on $\omega$ will be required.

\begin{enumerate}[label={(A\arabic*)}]
    \item For any path $\varphi$ of $\cT$, each of $\rho(u)$, for $u\in\varphi$, is an extreme point of $\conv \rho(\varphi)$.\label{assume:all_exposed}
    \item Any pair of component polytopes are distinct, i.e., for $i\neq j\in[I]$, $\cS_i\neq \cS_j$.\label{assume:distinct_components}
\end{enumerate}

These assumptions are jointly on the DRT and topic map. Assumption (A1) states that for each component topic polytope, it must have all its constituent topics as extreme points (vertices which cannot be expressed as convex combinations of other vertices). Thus, we consider topics such that no topic is a convex combination of other topics within a component polytope. Typically $K < V$, and hence with topics in general positions, this condition is almost surely met. This assumption is much weaker than requiring linear independence of the topics. It is even weaker than assuming affine independence among topics in the same component. Assumption (A2) deals with distinctness of the component polytopes. The next result shows the usefulness of the $d_{\bbH+}$ metric for understanding topic hierarchies.

\begin{proposition}\label{lemma: tree-directed identifies tree structure}
\begin{enumerate}
        \item If $\omega_0=(\cT,\rho,\pi)$ satisfying Assumptions \ref{assume:all_exposed} and \ref{assume:distinct_components} is fixed, then there exists $\epsilon_0=\epsilon_0(\omega_0)>0$ and  $C_0=C_0(\omega_0)>0$, such that for all $\omega'=(\cT',\rho',\pi')\in\Omega(I)$, whenever $d_{\bbH+}(\omega_0,\omega')<\epsilon\leq \epsilon_0$, the topic hierarchy of $(\cT',\rho')$ is $C_0\epsilon-$approximately close to that of $(\cT,\rho)$. If additionally, the number of nodes of $\cT$ and $\cT'$ are same, then $\cT\cong \cT'$.
        \item For $\omega, \omega'$ satisfying Assumption \ref{assume:all_exposed}, we have $d_{\bbH+}(\omega,\omega')=0$ if and only if $(\cT,\rho)$ and $(\cT',\rho')$ have the same topic hierarchy and $\pi(\varphi)=\pi'(\sigma(\varphi))$ for every path $\varphi$ of $\cT$, where $\sigma(\varphi)$ is the path of $\cT'$ mapped from $\varphi$ under the unique isomorphism in Definition \ref{def:topic_hierarchy1}.
    \end{enumerate}
\end{proposition}

Recall that for any isomorphism $\sigma:\cV\to\cV'$ between isomorphic DRTs $\cT=(\cV,\cE,v_0), \cT'=(\cV',\cE',v_0')$ of $\omega,\omega'$ respectively, $\sigma$ induces a bijection $\tilde{\sigma}:\Phi(\cT)\to\Phi(\cT')$ between the set of paths of $\cT$ and $\cT'$ — the last part of the above proposition makes use of this bijection (denoted as $\sigma(\varphi)$ by slight abuse of notation). 
\begin{remark}
We end our preparation with additional remarks.
\begin{enumerate}
    \item As noted before, the isomorphisms between isomorphic DRTs are not necessarily unique (see Section \ref{sec:background}). However, the particular bijection in the above proposition is special in being unique — this attests to the fact of the claim that the set of topics $\{\rho(u):u\in\cV\}$ and $\{\rho'(u):u\in\cV'\}$ are not just close but are arranged as a tree so that they have the same parent-child structures — in other words, they have the same topic hierarchy.
    \item The first part of the proposition ensures that for fixed $\omega$, whenever another parameter $\omega'$ comes in a sufficiently small neighborhood of $\omega$, in terms of $d_{\bbH+}$ metric, then the underlying topic hierarchy of $\omega$ must be a subset of the topic hierarchy of $\omega'$, i.e., the hierarchy for $\omega'$ contains the hierarchy information of $\omega$ (and possibly some very close other redundant topics). If additionally, $K=K'$ is known, then $\bJ=\bJ'$, where $\cT, \cT'$ have size $(I,\bJ,K)$ and $(I, \bJ',K')$ respectively.
    \item $d_{\bbH+}$ is a pseudo-metric on $\Omega(I)$. Based on part 2) of the above lemma, on $\{\omega\in\Omega(I):\omega \text{ satisfies \ref{assume:all_exposed}}\}$, the equivalence classes under $d_{\bbH+}$ correspond to distinct topic hierarchies. Note that Assumption \ref{assume:distinct_components} is only required in the first part. The reason is that in part 2), the condition that $d_{\bbH+}$ is 0 is strong enough to identify all the components under just assumption \ref{assume:all_exposed}.
    \item The converse direction of Proposition \ref{lemma: tree-directed identifies tree structure} is also true, subject to the fact that $d_{\bbH+}$ is only defined when the underlying DRTs have same number of leaves and the notion of topic hierarchy does not include the path probabilities. We can state the converse direction as follows: if $(\cT,\rho), (\cT',\rho')$ have $\epsilon-$approximate topic hierarchy, then
    $$\min_{\sigma\in\bbS(I)}\sum_{i\in[I]} d_{\cH}(\cS_i, \cS_{\sigma(i)}) < \epsilon,$$
    where we note that the left side of the above display is basically $d_{\bbH+}$ as in Definition \ref{def: augmented tree directed Hausdorff metric}, ignoring the path-probabilities.
\end{enumerate}
\end{remark}

\subsection{Identifiability theorems}\label{sec: identifiability2}
Having developed the requisite technical tools, we are ready to describe our main identifiability results in this section. For a standard parametric model $\{P_\theta:\theta\in\Theta\}$, the notion of identifiability boils down to injectivity of the mapping $\theta\mapsto P_\theta$. This is a necessary requirement for obtaining a consistent estimator $\hat{\theta}_m$ from i.i.d. observations $X_1,\dots,X_m\sim P_{\theta_0}$, i.e. $\hat{\theta}_m\to \theta_0$ as $m\to \infty$. For our model setting, the data size is determined by $m$ and $n$ and the i.i.d. structure is across documents, not the words inside a document. Using the notations from the previous section, we wish to establish identifiability in the `increasing $n$' regime, i.e. we want to ensure that $\omega$ can be estimated consistently given $\bX=(\doc^1,\dots,\doc^m)$ as $m,n\to\infty$. Recall from Eq. \eqref{eq: document density2} that for $\omega\in\Omega(I)$,

$$p_{n,\omega,\alpha}(\doc) = \int f(\doc|\eta) G(d\eta), \quad f(x_1,\dots,x_n|\eta) = \prod_{v \in [V]} \eta_v^{\sum_{j=1}^{n} 1(x_j=v)},$$
where the kernel $f$ is a $n$ product of multinomials and the mixing measure $G = \sum_{i\in[I]} \pi_i G_i$, where $G_i$ is as defined in Section \ref{sec:model-geom}. More specifically, for every $\omega=(\cT,\rho,\pi)\in\Omega(I)$ and $\alpha:\Phi(\cT)\to\bbR_+$ , there is a latent mixing measure $G(\omega,\alpha)$, based on which we have observations from $\doc \sim P_{n,G}$, with $p_{n,G}(\bx)=\int f(\bx|\eta)G(d\eta)$ as above. The reason we represent the model in these two steps, $(\omega,\alpha)\mapsto G(\omega,\alpha), G\mapsto P_{n,G}$, is that in order to identify the latent structure as desired, we hope to identify $\omega$ (and possibly $\alpha$), not just $G$. This leads to the question whether $(\omega,\alpha)\mapsto G(\omega,\alpha)$ is injective (under the metric we consider on the space for $\omega$) — this is the question we are interested in here. Note that this can be seen as `increasing $n$' regime, since as $n\to \infty$, the distribution of $\bar{\bX}$ (document mean), defined as $\bar{\bX}_v = \sum_j 1(X_j=v)/n$, converges weakly to $G$ itself. This is equivalent to the `noiseless' setting (similar to \cite{javadi2019nonnegative}), where we observe i.i.d. $\eta_1,\dots,\eta_m \sim G(\omega,\alpha)$ and wish to estimate $\omega$.

Our approach is to exploit the geometric structure of $G$ to identify $\omega$, noting that $G$ is supported on the union of $I$ convex polytopes. The particular nature of the underlying distribution which induces the measure $G_i$ on component polytope $\cS_i$ (the Dirichlet distribution in our case) is irrelevant for this study -- any distribution, continuous on the appropriate probability simplex, parametrized by $\alpha$, works. The following lemma, which deals with intersections of affine spaces and polytopes, is crucial in the results to follow.

\begin{lemma}\label{lemma: intersection affine space}
    Suppose $\mathcal{A},\mathcal{A}'$ are affine spaces and $\mathcal{S}$ is a convex polytope
\begin{enumerate}
    \item $\mathcal{S}\cap\mathcal{A}$ is either empty or another convex polytope.
    \item $\mathcal{A}\cap\mathcal{A}'$ is either empty or another affine space.
    \item If $\dim \mathcal{A}=\dim\mathcal{A}'$, then $\mathcal{A}\cap\mathcal{A}'$ either has strictly lower dimension or $\mathcal{A}=\mathcal{A}'$.
    \item If $\dim\mathcal{A}<\dim\mathcal{A}'$, then either $\mathcal{A}\cap\mathcal{A}'$ has dimension strictly less than $\dim\mathcal{A}$ or $\mathcal{A}\subset \mathcal{A}'$.
\end{enumerate}
\end{lemma}

The proof is straightforward and is included in the appendix for completeness. We state the key assumptions that we require in this section, which are both arguably mild (so as to avoid degenerate situations).
\begin{enumerate}[label={(B\arabic*)}]
    \item For any two maximal paths $\varphi$ and $\varphi'$, if $\dim \conv (\rho(\varphi)) = \dim \conv (\rho(\varphi'))$, then $\aff(\rho(\varphi)) \neq \aff(\rho(\varphi'))$. \label{assume:diff_affine_space}
    \item $\pi(\varphi)>0$ for any path $\varphi$ in $\cT$. \label{assume:pos}
\end{enumerate}
Now, we are ready to state the first identifiability theorem.

\begin{theorem}
    \label{thm:identifiability1}
    For any two parameter $\omega=(\cT,\rho,\pi)\in \Omega(I)$ and $\omega'=(\cT',\rho',\pi')\in \Omega(I')$, if both satisfy the assumptions \ref{assume:diff_affine_space} and \ref{assume:pos}, then the following statements are equivalent:

    \begin{itemize}
    \item [(a)]
    $d_{\TV}(P_{n,\omega,\alpha}, P_{n,\omega',\alpha'}) = 0$ for all $n=1,2,\ldots$.
    \item [(b)]
    $d_{\TV}(P_{n,\omega,\alpha}, P_{n,\omega',\alpha'}) \to 0 \text{ as } n\to\infty$. 
    \item [(c)]
    $I=I' \text{ and } d_{\bbH+}\left(\omega,\omega'\right)=0.$
    \end{itemize}
\end{theorem}
$(a)\Rightarrow (b)$ trivially, while $(c)\Rightarrow (a)$ follows immediately from the conclusion of Proposition \ref{lemma: tree-directed identifies tree structure} (part 2). The proof for $(b)\Rightarrow (c)$ is given in Appendix \ref{app:proof of identifiability1}, but here is a high-level proof sketch. As $n\to\infty$, the distribution of the sample mean of $\doc$ converges to the noiseless distribution $G(\omega,\alpha)$ of the latent variable $\eta$, as discussed at the beginning of this section. This probability measure is a mixture of components where each component is a probability measure supported on a polytope (constituent topic polytope) and these component measures are absolutely continuous with respect to the Hausdorff measure on the respective affine spaces. The assumption on the topic map guarantees that topic polytopes of the same dimension do not share a common affine hull. Suppose we start with this mixture probability measure and try to reconstruct its components. Consider the component whose polytope has the smallest number of dimensions (if multiple, select one arbitrarily). If we take the restriction of the whole measure on the affine hull of this polytope, the only contribution comes from just this component. We argue that there cannot be any component whose polytope has smaller number of dimensions and then argue that there must be a unique component which matches it. We continue this process of eliminating one component at a time until all components have been exhausted. This process uniquely identifies all the constituent topic polytopes and associated probability measures supported on such polytopes. 

\begin{remark}
There are additional remarks regarding the above theorem.
\begin{enumerate}
    \item The theorem can also be expressed as follows: for $\omega\in\Omega(I), \omega'\in\Omega(I')$ satisfying stated assumptions and $\alpha, \alpha'$ (seen as maps $\Phi(\cT)\mapsto \bbR_+$ on the associated set of paths), $G(\omega,\alpha)=G(\omega',\alpha')\Rightarrow I=I', d_{\bbH+}(\omega,\omega')=0$. Note that $d_{\bbH+}$ is a metric on $\Omega(I)$ and thus is applicable only if the underlying DRTs have same number of paths. The theorem guarantees that this is indeed the case and that the topics and their hierarchy can be identified. We emphasize the fact that the theorem requires no restriction on the number of paths in the associated DRTs — hence the model size is also identifiable in this sense.
    \item It is crucial that both $\omega, \omega'$ satisfy the assumptions (B1) and (B2). This is required since we do not make assumptions about the model size (e.g. $I$ or the underlying DRT). If only one, say $\omega$, satisfies the assumptions, it is not hard to see that there exists $\omega'$ (which would have to violate at least one of the assumptions) with underlying DRT $\cT'\ncong \cT$, along with some $\alpha,\alpha'$ such that $G(\omega,\alpha)=G(\omega',\alpha')$, where $\cT, \cT'$ are the corresponding DRTs for $\omega, \omega'$.
    \item We emphasize that we do not require specific parametric distribution for $\beta$ within each path — any continuous probability distribution on $\Delta^{J_i-1}$ would work (although in modeling practice the Dirichlet distribution is a natural choice as in the LDA) . For any such $P^{(i)}$ for path $i$, the component measure $G_i = P^{(i)}_{\#}L_i$, where $L_i(\beta)=\Theta_i^\top \beta$ (see Section \ref{sec:model-geom}) is supported on $\cS_i=\conv \Theta_i$, is absolutely continuous with respect to the Hausdorff measure on $\cA_i=\aff \cS_i$ and the density is positive in the interior of the support — this is all we require in the result. If we specialize to the case where the symmetric Dirichlet is employed in our model formulation, then the theorem entails that the $\alpha$ parameter for each component is also uniquely identifiable, i.e. $G(\omega,\alpha)=G(\omega',\alpha') \Rightarrow I=I', d_{\bbH+}(\omega,\omega')=0$ and $\alpha=\alpha'$.
\end{enumerate}
\end{remark}
The above result implies that given two distinct parameters $\omega, \omega'$ (satisfying the assumptions \ref{assume:diff_affine_space} and \ref{assume:pos}), there exists $n$ such that $d_{\TV}(P_{n,\omega,\alpha}, P_{n,\omega',\alpha'})>0$, i.e., there is a sufficiently large document size which makes these two parameters distinguishable from the data (Note here the distinct $\omega, \omega'$ are distinguishable under the metric $d_{\bbH+}$).

For parameter estimation, either by a maximizing log likelihood estimator or applying the Bayesian framework, we are interested in estimating $\omega$ based on $\doc^1,\dots,\doc^m \sim  P_{n,\omega^*,\alpha^*}$ for some \textit{true} parameter $\omega^*,\alpha^*$ and to be consistent, we want $\hat{\omega}_{m,n} \to \omega^*$ as $m,n\to\infty$, where $\hat{\omega}_{m,n}$ is some estimator based on $\doc^1,\dots,\doc^m$. For the above result to be useful, we need to restrict the parameter space to some $\Omega_A\subset \cup_I \Omega(I)$, where $\Omega_A$ is such that for any sequence $\omega_n\in\Omega_A$, any limit point of the sequence should satisfy assumptions \ref{assume:diff_affine_space} and \ref{assume:pos}. Thus, even if the true parameter $\omega^*$ satisfies the assumptions, under either the MLE or Bayesian framework this would require making stronger restrictions on the parameter space and require such restrictions be known to the statistical modeller. Instead of this requirement, we would rather want to place minimal restrictions on the parameter space (or the prior) and hope to recover unique latent structure based on data, assuming \emph{only} that the true latent structure satisfies such assumptions. The following theorem is motivated from such a practical consideration. 

\begin{theorem}
    \label{thm:identifiability2}
    Consider $\omega=(\cT,\rho,\pi)\in \Omega(I)$ satisfying assumptions \ref{assume:diff_affine_space} and \ref{assume:pos}. Then, for any $\omega'=(\cT',\rho,\pi)\in \Omega(I)$ the following statements are equivalent:

    \begin{itemize}
    \item [(a)]
    $d_{\TV}(P_{n,\omega,\alpha}, P_{n,\omega',\alpha'}) = 0$ for all $n=1,2,\ldots$.
    \item [(b)]
    $d_{\TV}(P_{n,\omega,\alpha}, P_{n,\omega',\alpha'}) \to 0 \text{ as } n\to\infty$. 
    \item [(c)]
    $d_{\bbH+}\left(\omega,\omega'\right)=0.$
    \end{itemize}
\end{theorem}

Note the difference between Theorems \ref{thm:identifiability1} and \ref{thm:identifiability2}. In Theorem \ref{thm:identifiability2}, no assumption is required on $\omega'$, while it was crucial that $\omega'$ also satisfied the assumptions in Theorem \ref{thm:identifiability1}. Instead, Theorem \ref{thm:identifiability2} requires the knowledge of the number of paths -- note that for this result, both $\omega,\omega'\in \Omega(I)$ ---- thus their underlying DRTs have the same number of paths. Since the number of paths accounts for the number of components in the overall mixture structure of the model, this assumption ensures that for each component of $\omega$ there can only be at most one component of $\omega'$ to match it.

\begin{remark}
    Assumption \ref{assume:all_exposed} is not required in either of the results in this section. The reason is that the metric $d_{\bbH+}$ only cares about the support of the individual components — thus, the component measures can be identified. Assumption \ref{assume:all_exposed} would be additionally required to identify \textit{all} the topics and the topic hierarchy, as established by Proposition \ref{lemma: tree-directed identifies tree structure}.
\end{remark} 

\begin{figure}
    \centering
\includegraphics[clip, trim=10cm 5.5cm 10cm 6.6cm, width=0.95\textwidth]{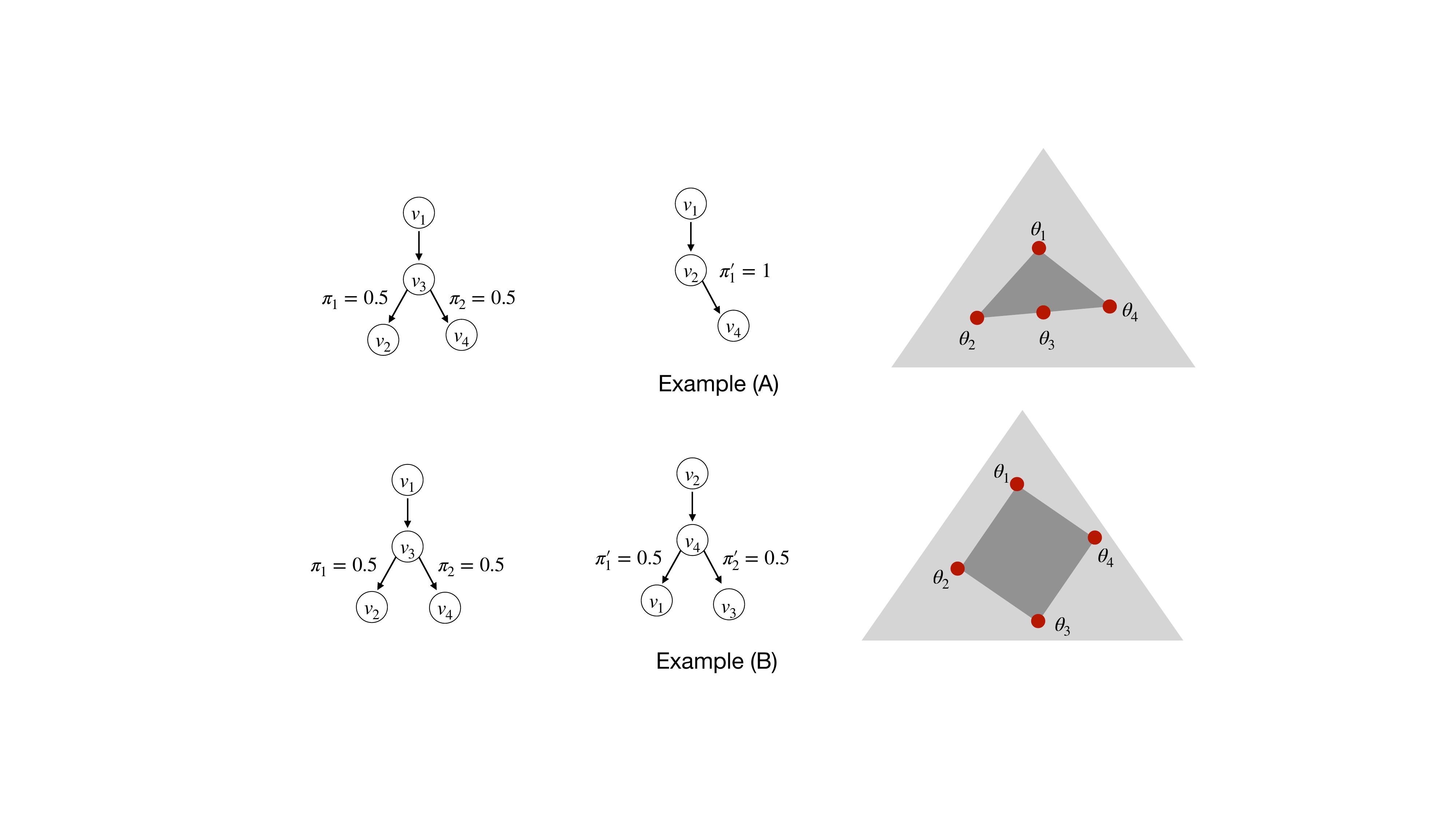}
    \caption{Examples of non-identifiability when Assumption \ref{assume:diff_affine_space} are violated — Examples (A) and (B) respectively. The outer light gray triangle is the vocabulary simplex and the dark gray region represent the noiseless probability measure under the parameter choices as shown. In both cases, the topic map is $\rho:v_\ell\to \theta_\ell$ for all $\ell$. Note that in (A) two different trees with $I\neq I'$ led to the same model, which for (B), the underlying DRT is the same, as well as the set of topics — but two very different topic hierarchies lead to the same model.} 
    \label{fig:DRT_nonidentifiable}
\end{figure}

\begin{remark} Here are examples of non-identifiability that arises if we drop the above assumptions. Two such examples are given in Figure \ref{fig:DRT_nonidentifiable}. In the first example, we see two distinct DRTs with different number of paths leading to the same model (the underlying Dirichlet is assumed with parameter 1, i.e., a uniform distribution on the associated component simplexes). The problem arises since polytopes $\conv (\theta_1,\theta_2,\theta_3)$ and $\conv (\theta_1,\theta_3,\theta_4)$ have the same affine hull. In the second example, we see the same DRT (i.e., isomorphic) yet the topic hierarchies are different. For this set of topics ($\theta_1,\theta_2,\theta_3,\theta_4$ for the vertices of a rectangle here) the decomposition of $\cup_i \cS_i$ into $\cS_1$ and $\cS_2$ become non-unique (since they have the same affine hull). We also note that if $\theta_4$ was not in the affine hull of $\conv (\theta_1,\theta_2,\theta_3)$, then the two topic hierarchies would lead to two distinct models (note that in the former $\theta_1$ is shared between the two components, while in the latter $\theta_2$). Thus, without regularity conditions such as Assumptions (B1) and (B2), the knowledge of the underlying DRT and the set of topics is not enough to determine the latent structure, the topic hierarchy, in this model.
\end{remark}

\section{Posterior Contraction Rates}\label{sec:rate}

We turn our attention to convergence behaviors of the density and parameters in our model, under a Bayesian well-specified and identifiable model setting. Given the $m\times n$ data
$$\doc^1,\dots,\doc^m \overset{iid}{\sim} P_{n,\omega, \alpha}$$
where we know the underlying DRT up to an isomorphism. Consider $\omega, \omega'$ such that $\cT\cong \cT'$, then it is enough to consider a single tree. Indeed, since they are isomorphic, let $\sigma$ be an isomorphism and $\tilde{\sigma}$ be the corresponding bijection between $\Phi(\cT)$ and $\Phi(\cT')$, then instead of $\omega'$, we can consider $(\cT, \rho'\circ\sigma, \pi'\circ \tilde{\sigma})$. Thus, it is reasonable to fix a common DRT for both $\omega$ and $\omega'$. Assume that $\omega\in\Omega(\cT)$ for a fixed $\cT$. Note that this does not mean that the topic hierarchy is known, as the identifiability theory has established in the previous section.
Furthermore, we are interested in the estimation of $\rho$ specifically, since $\rho$ and $\cT$ combine to determine the topics as well as the topic hierarchy, which is our target. Hence, in this section, we assume that $\alpha$ is known, i.e. $\alpha(\varphi)=\alpha_0>0$, known and fixed (we remark that what we just require is that the distribution $p_{\beta}$ for each path only depends on the path length — i.e., for symmetric Dirichlet case, we need $\alpha(\varphi)=\alpha_{J}$ where $J=|\varphi|$ and $\{\alpha_J\}_{J\in\bJ}$ known), as such we drop the $\alpha$ in our notation. Moreover, for simplicity of exposition, in this section we assume that each path of $\cT$ has the same length $J$.

Assume that the $m\times n$ data set is generated from true parameter $\omega^*\in \Omega(\cT)$, where $\cT\in \mathfrak{T}(I,J)$. We place an appropriate prior for $\omega\sim \Pi$ on $\Omega(\cT)$ and are interested in the posterior distribution of $p_{n,\omega,\alpha_0}$ in Section \ref{sec:density estimation}, in particular the contraction rate of the posterior to the truth $p_{n,\omega^*,\alpha_0}$ in terms of the Hellinger metric as $m,n\to\infty$. We shall adopt the standard Bayesian asymptotic framework to address this problem \citep {GhosalVaartFundamentalsNonparametricBayesian2017}. In Section \ref{sec:parameter estimation}, we are interested in the posterior distribution of $\omega$ itself and study the contraction rate of the posterior to the truth $\delta_{\omega^*}$ under an appropriate metric $d$. To this end, we establish an inverse bound of the form $d_{\TV}(p_{n,\omega}, p_{n,\omega'}) \gtrsim d_{\cU\cH}(\omega,\omega')$, which allows us to use the density estimation rate to prove an upper bound for the parameter estimation rate in the metric $d_{\cU\cH}$, which we shall define shortly. We utilize and extend the techniques developed in \cite{NguyenPosteriorContractionPopulation2015} — however, in our case, the latent measure $G$ is no longer supported on a single convex polytope as in \cite{NguyenPosteriorContractionPopulation2015}; here it is supported on a union of such convex polytopes, which requires novel modification to the inverse bound.

\subsection{Density estimation}\label{sec:density estimation}
We start our study of posterior density contraction rates by deriving an upper bound on the KL divergence in terms of the $W_1$ distance between the component measures $\{G_i:i\in[I]\}$. This result is crucial to understanding the behavior of the model in a KL-neighborhood of the truth. We aim to derive a connection between the KL divergence between the document-distributions and the distance between the associated parameters, in terms of the augmented tree-directed Hausdorff metric. This would allow us to control certain properties in the density space $\{p_{n,\omega}:\omega\in\Omega(\cT)\}$ in terms of the underlying parameter space $\Omega(\cT)$. 

\begin{lemma}\label{lemma: kl_upper_W}
    Suppose $\omega=(\rho,\pi), \omega'=(\rho',\pi')\in \Omega(\cT)$. Assume that both $\omega, \omega'$ satisfy
    \begin{enumerate}[label=({C}\arabic*)]
        \item  $\min_{u \in \cV} \rho(u)_v \geq c_0$ for all $v\in [V]$ for some constant $c_0 > 0$.\label{assume: bound topic boundary}
        \item $\pi(\varphi) \geq c_1$ for all paths $\varphi\in\Phi(\cT)$ for some constant $c_1 > 0$. \label{assume: bound path probability}
    \end{enumerate}
    Then, for any permutation $\sigma\in \bbS(I)$, we have
    $$\KL(p_{n,\omega}\Vert p_{n,\omega'}) \leq \frac{n}{c_0}\sum_{i\in[I]} \pi_i W_1\left(G_i, G_{\sigma(i)}'\right) + \frac{1}{c_1}\sum_{i\in[I]} \pi_i |\pi_i - \pi_{\sigma(i)}'|,$$
    where $G_i, \pi_i$ are the latent component measure and path probability associated to the $i-$th path of $\cT$.
\end{lemma}
The assumptions state that all associated topics are uniformly bounded away from the boundary of $\Delta^{V-1}$, a condition required to control the KL-divergence from exploding near the boundary. Similarly, the second condition assumes that $\pi$ is bounded away from the boundary of $\Delta^{I-1}$. The above lemma, 
along with a result of \cite{NguyenPosteriorContractionPopulation2015} (Lemma 7), allows us to control the KL-divergence in terms of the $d_{\bbH+}$ metric. We use the following geometric regularity condition to ensure that for each of the associated component polytopes, each `corner' of the polytope is sufficiently acute.

\begin{definition}\label{def: non-obtuse corner}
    We say that a convex polytope $\cS$ satisfies the \textit{non-obtuse corner property} if at each vertex of $\cS$, there is a supporting hyperplane whose angle formed with any edges adjacent to that vertex is bounded from below by $\delta$, where $\delta>0$ is a fixed number.
\end{definition}

\begin{corollary}\label{corollary: kl upper dH+}
    Given $\omega,\omega'$ as in Lemma \ref{lemma: kl_upper_W} such that all associated component polytopes have the non-obtuse corner property in Definition \ref{def: non-obtuse corner}. Then,
    \begin{equation} \label{eq: upper bound KL dH+}
        \KL(p_{n,\omega}\Vert p_{n,\omega'}) \leq  \left(\frac{nC_{\delta}}{c_0} \vee \frac{1}{c_1}\right) d_{\bbH+}(\omega, \omega')
    \end{equation}
    where $C_\delta>0$ is a constant depending only on $\delta$ in Definition \ref{def: non-obtuse corner}.
\end{corollary}

In the above corollary, we abuse notations slightly, treating $\omega$ as $(\rho,\pi)$ (resp. $(\rho',\pi')$) in the left-side and as $(\cT,\rho,\pi)$ (resp. $(\cT,\rho',\pi')$) in the right-side of Eq. \ref{eq: upper bound KL dH+}. We note that with sufficiently large $n$, the result reads
$$\KL(p_{n,\omega}, p_{n,\omega'}) \leq \tilde{c}n d_{\bbH+}(\omega,\omega'),$$
i.e., the bound worsens as $n$ becomes large. However, this corollary allows us to derive both prior concentration property and control of the complexity of the model class, based on which we establish the following posterior contraction rate for density estimation.

\begin{theorem}\label{thm: density contraction rate}
    Suppose $\Pi$ is a prior distribution on $\Omega(\cT)$ such that for some constants $c_0, c_1, c_2, c_3>0$ the following holds for any $\omega=(\rho,\pi)$ in the prior support:
    \begin{enumerate}
    \item Assumptions \ref{assume: bound topic boundary}, \ref{assume: bound path probability} hold.
        \item All component polytopes $\cS_i$ satisfy the non-obtuse corner property given in Defn. \ref{def: non-obtuse corner}.
        \item For any small $\epsilon>0$, there exists constants $c_2, c_3>0$ such that
        \begin{align*}
            \Pi\left(\norm{\rho(u) - \rho^*(u)}\leq \epsilon, \, \forall u\in \cV\right) &\geq c_2 \epsilon^{K(V-1)} \\
            \Pi\left(\norm{\pi(\varphi) - \pi^*(\varphi)}\leq \epsilon, \, \forall \varphi\in \Phi(\cT)\right) &\geq c_3 \epsilon^{I-1}
        \end{align*}
        and the events in the above display are independent under $\Pi$ (recall $I=|\Phi(\cT)|$ and $K=|\cV|$, size of the set of nodes of $\cT$), where $\omega^*=(\rho^*,\pi^*)$ is a specific parameter in the support of $\Pi$.
    \end{enumerate}
     Then, we have the following:
    \begin{enumerate}[label=(\alph*)]
        \item (fixed $n$ case) For sufficiently large constant $C$, as $m\to \infty$ with $n$ held fixed
        \begin{equation}
            \Pi\left(h(p_{n,\omega^*}, p_{n,\omega}) > C \sqrt{\frac{\log m}{m}} \Bigg| \doc^1,\dots,\doc^m\right) \to 0 \text{ in } P_{n,\omega^*}^\infty-\text{probability}.
        \end{equation}
        \item (increasing $n$ case) For sufficiently large constant $C$, as $m,n\to \infty$ such that $\log n = o(m)$, we have
        \begin{equation}
            \Pi\left(h(p_{n,\omega^*}, p_{n,\omega}) > C \sqrt{\frac{\log (m\vee n)}{m}} \Bigg| \doc^1,\dots,\doc^m\right) \to 0 \text{ in } P_{n,\omega^*}^\infty-\text{probability}.
        \end{equation}
    \end{enumerate}
    where $h$ is the Hellinger metric.
\end{theorem}

The proof of the above result is provided in the Appendix \ref{app: proof of density rate}. We show that the conditions required to employ Theorem 8.9 from \cite{GhosalVaartFundamentalsNonparametricBayesian2017} are indeed satisfied for our model — prior mass property in Proposition \ref{prop: KL property prior} and model entropy in Proposition \ref{prop: model entropy} in the Appendix. 

\begin{remark}
We make the following remarks regarding the preceding theorem.
\begin{enumerate}
    \item The model enjoys an almost parametric density contraction rate, up to logarithmic terms. Part (a) of the theorem shows that for density estimation, we do not require $n\to \infty$. 
    \item No additional assumptions are required for the true parameter $\omega^*$ as long as it is in the support of the prior $\Pi$.
    \item Restrictions on the prior $\Pi$: the first two are regularity conditions required to establish Corollary \ref{corollary: kl upper dH+}, which in turn is required to establish prior concentration as well as controlling entropy of the model class. The last assumption only states that the prior places sufficient mass near the true parameter. The first three regularity conditions are required for technical purposes to ensure that the associated parameters are bounded away from the boundary of the respective probability simplexes — this ensures that the multinomial density stays bounded. In practice, we use Dirichlet distribution as a prior for each of the topics and also for the path-probability vector.
    \item The constant $C$ depends only on size of the DRT $\cT$ (in particular $I,K$) and also the vocabulary size $V$. It does not depend on the true parameter $\omega^*$ in this case.
    \item The theorem shows that large $n$ affects the contraction rate of the density only slightly (in logarithmic terms). This arises due to the upper bound on the KL divergence in terms of the $W_1$ metric on the latent measures $G_i$. As $n$ increases, the conditional distributions of the document mean given that the document is generated from a particular path become singular if the underlying component polytopes have different affine spaces. This also results in worsening of the upper bound on the entropy of the model space as $n$ gets large, although the upper bound on the entropy of the parameter space does not depend on $n$.
\end{enumerate}
\end{remark}

\subsection{Latent structure estimation}\label{sec:parameter estimation}
In this section, we turn to the convergence behavior of the parameters arising in our latent structured model. This kind of question has been studied in recent works for mixture models \citep{nguyen2013convergence,guha2021posterior} and hierarchical models \citep{NguyenPosteriorContractionPopulation2015,NguyenBorrowingStrenghHierarchical2016}. A general approach considered in these works is to establish so-called inverse bounds, which are a type of lower bound for the total variation distance between distributions of the observed data in terms of a distance between the corresponding latent structures using an appropriate metric. Such an inverse bound may be of the form $d_{\TV}(P_{G}, P_{G'}) \gtrsim d(G, G')$, where $G, G'$
capture all information about the latent structure in the model and $d$ is some appropriate metric on the space of the latent $G$. For finite and infinite mixture models, $G$ is the discrete mixing measure, e.g., $G=\sum_k \pi_k \delta_{\theta_k}$ and the metric $d$ is often chosen as $W_1$ or $W_2$. For the LDA, $G$ is taken as the underlying topic polytope (called \textit{population polytope}) and the metric as the Hausdorff distance \citep{NguyenPosteriorContractionPopulation2015}. It is worth noting that such an inverse bound is the reverse direction of what we used in the previous sections. For instance, Corollary \ref{corollary: kl upper dH+} gives an upper bound for the KL divergence and by Pinsker's inequality, $d_{\TV}^2 \leq 2\KL$, which yields an upper bound for $d_{\TV}$ in terms of the $d_{\bbH+}$ metric on the parameters which represent the latent structure. However, establishing lower bounds for the total variation distance turns out to be more challenging. Given such an inverse bound, say $d_{\TV}(P_{G}, P_{G'}) \gtrsim d(G, G')$, it is easy to transfer density contraction rates to parameter contraction rates under the metric $d$.

For the tree-directed topic models, the latent structure is represented by the parameter $\omega\in\Omega(\cT)$; the metric $d_{\bbH+}$ on this space has been utilized to identify $\omega$ and to establish the data density contraction rates. It is difficult to obtain an inverse bound directly with this metric. For a given $\cT$, the topic hierarchy is captured entirely by the topic map $\rho$. Thus, in this section, the following pseudo-metric for $\rho$, to be called \textbf{union Hausdorff metric}, will be employed.

\begin{definition}
    For a fixed DRT $\cT$ and $\rho,\rho'\in \mathfrak{R}(\cT)$, define
    \begin{align*}
        d_{\cU\cH}(\rho,\rho'|\cT) &= d_{\cH}\left(\cup_{i\in [I]} \cS_i, \cup_{i\in [I]} \cS_i'\right)
    \end{align*}
    where $I=|\Phi(\cT)|$ is the number of paths of $\cT$ and $\cS_i$ is the component topic polytopes associated with path $i$ under $\rho$, given a particular enumeration of the paths of $\cT$ (and similarly for $\cS_i'$).
\end{definition}

The notation $d_{\cU\cH}(\rho,\rho')$, which hides the dependence on the underlying DRT $\cT$, will also be used  when this is clear from context. The union Hausdorff metric inherits metric properties from the Hausdorff metric directly. Since the metric is the Hausdorff distance on union of the polytopes, in general it may not be used to disentangle the individual polytope components, unless some suitable conditions hold. 
Lemma \ref{lemma: union Hausdorff identifies polytopes} in the Appendix gives such sufficient conditions to identify the complete latent structure from the union of the supports, by exploiting the special geometric structure of the model. Note that $\cup_i \cS_i$ is not a convex set in general for our model, which makes the "disentangling" challenging. However, if the measure $G$ places enough mass everywhere in $\cS$, then this issue can be overcome. Fortunately, because $G=\sum_i \pi_i G_i$, our next result demonstrates that each $G_i$ indeed places sufficient probability mass everywhere in its support, provided that the component $G_i$'s are push-forwards of Dirichlet distributions. This guarantees that if the $\pi_i'$s are also bounded from below, $G$ places sufficient mass everywhere in its support. 


\begin{lemma}\label{lemma:lb_mass_small_ball}
    For $G=\text{Dir}(K;\alpha)_{\#}L$, where $L(\beta)=\sum_{k\in[K]}\beta_k\theta_k$, then  for any $\eta\in \cS$, where $\cS=\conv(\{\theta_1,\dots,\theta_K\})$, and for all $0<\epsilon<1/K$,
    $$G(B(\eta,\epsilon)\cap \cS)\geq C(K,p,\alpha)\epsilon^{Q(K,p,\alpha)}$$
    where $p=\dim \cS$, $C(K,p,\alpha)>0$ is a constant free of $\epsilon$ and
    \begin{align*}
        Q(K,p,\alpha) = \begin{cases}
            p + \alpha(K-p-1) &\alpha \leq  1 \\
            \alpha K -1 &\alpha>1.
        \end{cases}
    \end{align*}
\end{lemma}
The above lemma guarantees that each component measure arising in the model gives sufficient mass to each $\epsilon-$ball inside the component, each such ball is of the same dimension as the corresponding component polytope. In the special case, $\alpha\leq 1$ and the topics are affinely independent, i.e., $p=K-1$,  $Q(K,p,\alpha)=K-1$. In all cases, the bound worsens as $\alpha>1$ gets larger, because the Dirichlet distribution puts lower mass near the boundary of its support. If the intrinsic dimension is small compared to $K$ and $\alpha<1$, then we find that $Q(K,p,\alpha)=\alpha(K-1)+(1-\alpha)p$, a convex combination of $K-1$ (similar to the affinely independent case) and $p$ (the intrinsic dimension). If $\alpha>1$, then we always have the worse exponent $\alpha K-1$ regardless of $p$. We note that $\max_{1\leq p \leq K-1} Q(K,p,\alpha) = (1\vee \alpha)K -1$ denotes the worst possible value of $Q(K,p,\alpha)$ for fixed $K$ and $\alpha$. Based on this result, we can prove the following inverse bound. 

\begin{theorem}\label{thm: inverse bound}
There is a constant $\epsilon_0>0$ depending on $J$ such that
    for any $\omega=(\rho,\pi)\in\Omega(\cT), \omega'=(\rho',\pi')\in\Omega(\cT)$ satisfying Assumption \ref{assume: bound path probability}, whenever $d_{\cU\cH}(\rho,\rho')<\epsilon\leq\epsilon_0$ there  holds
    \begin{equation}\label{eq: inverse bound}
        C_1 d_{\cU\cH}(\rho,\rho')^{(1\vee\alpha_0)J-1} \leq d_{\TV}(p_{n,\omega}, p_{n,\omega'}) + 6V \exp\left[-\frac{n}{8V} d_{\cU\cH}(\rho,\rho')^2\right],
    \end{equation}
    where $J$ is the depth of the DRT $\cT$,  $C_1=C_1(c_1,J,\alpha_0)>0$ is a constant, and $c_1$ is the constant in Assumption \ref{assume: bound path probability}.
\end{theorem}

\textit{Proof sketch}: The proof uses the probabilistic model $\eta\sim G=\sum_i \pi_i G_i$ (where $G_i$s are the component measures, each supported on a component polytope $\cS_i$) and $\doc|\eta$ are i.i.d. from a multinomial distribution with $\eta$ as the parameter. By a standard concentration inequality, for large $n$, the distribution of the document means $\hat{\eta}_i\in\Delta^{V-1}$ concentrate in small balls around that of the corresponding $\eta$, under the joint distribution of $(\eta, \doc)$. Finally, we argue that if the union Hausdorff metric is $\epsilon$, then there exists a set $A^*\subset\Delta^{V-1}$, which is in either $\cS=\cup_i \cS_i$ and well-separated from $\cS'=\cup_i \cS_i'$ or in $\cS'$ and well-separated from $\cS$. In either case, this set enables one to obtain a lower bound on the total variation distance between the document distributions, since $G, G'$ place different probability on a sufficiently small fattening of $A^*$. The complete proof of this theorem is given in the Appendix.

\begin{remark}
    We make a few remarks about the above theorem:
\begin{enumerate}
    \item Special case: If each of the component polytopes have the same dimension $p$, and $\alpha\leq 1$, then the inverse bound improves to
    $$C_1 d_{\cU\cH}(\rho,\rho')^{\alpha J + (1-\alpha) p -1} \leq d_{\TV}(p_{n,\omega}, p_{n,\omega'}) + 6V \exp\left[-\frac{n}{8V} d_{\cU\cH}(\rho,\rho')^2\right].$$
    In the case the topics in each component are affinely independent (which is the case when $J<V$ and the topics are in general positions), then the exponent of the union Hausdorff metric in the left-hand side becomes simply $(J-1)$.
    \item The only assumption required for the inverse bound is that the path probabilities for each model are bounded from below by a constant ($c_1$ in the lemma).
    \item Although the above theorem is presented using the Dirichlet distribution $p_\beta$ on the probability simplex, this is not strictly required. In fact, the Dirichlet distribution may be replaced by any distribution so long as the conclusion of Lemma \ref{lemma:lb_mass_small_ball} holds for the pushforward probability measure $(p_{\beta})_{\#}L$. Specifically, the theorem continues to hold for any probability distribution $p_{\beta}$ arising from an $\alpha$-regular family (as defined in \cite{NguyenPosteriorContractionPopulation2015}). 
    
    \item Due to the presence of the additive term (the second quantity) in the right hand side of Eq. \eqref{eq: inverse bound}, we cannot directly obtain $d_{\cU\cH}\lesssim d_{\TV}$ in this model and require $n\to\infty$, since that term decays exponentially with $n$. However, if $d_{\cU\cH}(\rho,\rho')>\epsilon$, then taking $n\geq \frac{8V}{\epsilon^2}\log \left[\frac{12V}{C_1\epsilon^{(1\vee\alpha)J-1}}\right]$ ensures that $d_{\TV}(p_{n,\omega}, p_{n,\omega'})\geq \frac{C_1}{2}\epsilon^{(1\vee\alpha)J-1}$. As a consequence, we have showed that as long as the Hausdorff distance between $\cup_i \cS_i$ and $\cup_i \cS_i'$ are bounded away from 0, the total variation is bounded away in terms of this for sufficiently high $n$, irrespective of the path probabilities subject to Assumption \ref{assume: bound path probability}.
\end{enumerate}
\end{remark}

The inverse bound allows us to transfer the density estimation rate we derived in the last section to derive estimation rates for the parameter $\rho$ in terms of the $d_{\cU\cH}$ metric.

\begin{theorem}\label{thm :parameter contraction rate}
    Suppose $\Pi$ is a prior on $\Omega(\cT)$ satisfying the assumptions in Theorem \ref{thm: density contraction rate} and $\omega_0=(\rho_0,\pi_0)$ is any point in the support of $\Pi$. Then, given data $\doc^1,\dots,\doc^m\overset{iid}{\sim} P_{n,\omega_0}$, the posterior distribution of $\rho$ contracts to $\delta_{\rho_0}$ as 
    \begin{equation}
        \Pi\left(d_{\cU\cH}(\rho, \rho_0) > C\epsilon_{m,n}\Bigg| \doc^1,\dots,\doc^m\right)\to 0\text{ in } P_{n,\omega_0}^\infty \text{ probability as } m,n\to\infty 
    \end{equation}
    such that $\log n = o(m)$, for some suitably large constant $C>0$, for the choice
    \begin{equation}
        \epsilon_{m,n} = \left[\log (m\vee n) \left(\frac{1}{m} + \frac{1}{n}\right)\right]^{\frac{1}{2[(1\vee\alpha)J - 1]}}
    \end{equation}
    where $J$ is the depth of the tree $\cT$.
\end{theorem}

\textit{Proof sketch:} The inverse bound allows us to control $\liminf_{n\to\infty} h^2(p_{n,\omega_0}, p_{n,\omega})$ in terms of $\epsilon>0$ for $\omega$ in the support of $\Pi$ such that $d_{\cU\cH}(\rho,\rho_0)\geq \epsilon$. In particular, as long as $\epsilon\gtrsim \sqrt{\log n /n}$, for such $\omega$, we have $\liminf_{n\to\infty} h^2(p_{n,\omega_0}, p_{n,\omega}) \geq C\epsilon^{2[(1\vee\alpha)J-1]}$ for some constant $C>0$. Thus, the posterior probability on $\{d_{\cU\cH}(\rho,\rho_0)\geq \epsilon_{m,n}\}$ can be upper bounded by the posterior mass on $\{h(p_{n,\omega_0}, p_{n,\omega})\geq C'\epsilon_{m,n}^{(1\vee\alpha)J-1}\}$ as $n\to\infty$, which we further know goes to 0 as $m\to \infty$ because of the density contraction result.

\begin{remark}
    We make a few remarks regarding the preceding theorem.
\begin{enumerate}
    \item $\epsilon_{m,n}$ is merely an upper bound for the contraction rate in terms of the union Hausdorff metric. We require both $m,n\to \infty$ for this result. This is in contrast to Theorem \ref{thm: density contraction rate}, which only requires $m\to\infty$ and $n$ could be fixed. The situation arises due to the need to dominate the second term in Eq. \eqref{eq: inverse bound} in Theorem \ref{thm: inverse bound}. This can be taken as a consequence of the proof technique, where we study parameter estimation through density estimation — for fixed $n$, as $m\to\infty$, the density can be estimated at the usual parametric rate (with respect to the number of iid documents); however, the component polytopes can be estimated from this density at a rate given in the inverse bound as $n\to\infty$.
    \item In the special case $\alpha\leq 1$ and each component consists of affinely independent topics, the rate boils down to $\epsilon_{m,n}=[\log m (1/m + 1/n)]^{1/2(J-1)}$ in the typical case when $m>n$.
    \item The only restrictions required are those on the prior, same as the ones for the density estimation result. They arise in our attempt to upper bound $K_2$, defined as $K_2(p, q) = \int (\log(p/q) - \KL(p\Vert q))^2 dP$, which plays a crucial role in the Bayesian asymptotic analysis for density estimation. 
    \item The result does not explicitly require the use of the Dirichlet distribution, and is applicable to any distribution $p_\beta$ on the probability simplex subject to a certain restriction. Specifically, under the well-specified setting as considered here, the restriction is related to placing sufficient mass near the boundary of the simplex, as discussed in the third remark after Theorem \ref{thm: inverse bound}.
\end{enumerate}
\end{remark}

\section{Inference algorithm}\label{sec:inference}

In this section, we describe an inference algorithm for estimating the topics in this model. We consider a collapsed Gibbs sampler, which is derived based on the collapsed Gibbs sampler for the LDA \citep{griffiths2004finding}. We utilize the equivalent model with the additional latent variables $C\in [I]^m$ and $L \in [J]^{m\times n}$, where $C_i$ is the label of the path of the underlying DRT associated to document $\doc^i, i\in[m]$ and $L_{i,j}$ is the depth associated to word $X_{i,j}$ in document $\doc^i$ for $i\in[m], j\in[n]$. Given a DRT $\cT$ of size $(I,J,K)$,  topics $\theta_1,\dots,\theta_K$ (where $\theta_k = \rho(v_k)$ for topic map $\rho$), path-probability vector $\pi\in\Delta^{I-1}$ and document-specific topic mixture $\beta_i\in\Delta^{J-1}$ for $i\in[m]$, the model specification is given as
\begin{align*}
    C_i | \pi &\overset{i.i.d.}{\sim} \text{Cat}(\pi) \\
    L_{i,j} | \beta_i &\overset{i.i.d.}{\sim} \text{Cat}(\beta_i) \\
    X_{i,j} | C_i, L_{i,j} &\sim \text{Cat}(\theta_{z(c_i, L_{i,j}))}.
\end{align*}
Here, we use the following parametrization of the tree — for a path label $c$ and depth label $\ell$, $z(c,\ell)$ is the index of the node among $\{1,\dots,K\}$, such that $v_k$ is at depth $\ell$ from the root along path $\varphi^c$ in the DRT. Let $z^{-1}(k) := \{(c,\ell): c\in [I], \ell\in[J], z(c,\ell)=k\}$ denote the set of all tuples of path and depth such that it locates node $k$. The following priors are used for simplicity, using symmetric Dirichlet parameterized by a scalar parameter
\begin{align*}
    \pi &\sim \text{Dir}_I(\pi_0), \\
    \beta_m &\overset{i.i.d.}{\sim} \text{Dir}(\alpha), \\
    \theta_k &\overset{i.i.d.}{\sim} \text{Dir}(\eta).
\end{align*}
The joint distribution of $\corpus, C, L, \Theta, \beta, \pi$ (where $\Theta$ includes all the topics $\theta_1,\dots,\theta_K$ and $\beta$ includes all the document allocations $\beta_1,\dots,\beta_m$) is given by
\begin{align*}
    p\left(\corpus, C, L, \Theta, \beta, \pi\right) &= p\left(\corpus|C,L\right) p(C,L|\Theta,\beta,\pi)p(\Theta|\eta)p(\pi|\pi_0)p(\beta|\alpha)\\
    &=p(\pi|\pi_0)p(\Theta|\eta)\prod_{i\in[m]} p(C_i|\pi)p(\beta_i|\alpha)\prod_{j\in[n]} p(L_{i,j}|\beta_i)p(X_{i,j}|\Theta, C_i, L_{i,j}).
\end{align*}
Employing the Dirichlet-multinomial conjugacy, one can marginalize out $\Theta, \beta$ and $\pi$ to obtain
$p\left(\corpus, C, L\right) = p(\corpus|C, L) p(C, L)$ where
\begin{align}
    p\left(\corpus\Big| C, L\right) &= \left(\frac{\Gamma(V\eta)}{\Gamma(\eta)^V}\right)^K \prod_{k\in[K]} \frac{\prod_{v\in [V]} \Gamma(N_{vk}+\eta)}{\Gamma(N_{\cdot k} + V\eta)} \label{eq: X given C,L}\\
    p(C, L) &= \left[\frac{\Gamma(I\pi_0)}{\Gamma(\pi_0)^I}\frac{\prod_{i\in[I]} \Gamma(M_i + \pi_0)}{\Gamma(M_\cdot + I\pi_0)}\right] \times \left[\left(\frac{\Gamma(J\alpha)}{\Gamma(\alpha)}\right)^m \prod_{i\in [m]} \frac{\prod_{k\in \varphi^{c_i}} \Gamma(\tilde{N}_{ik} + \alpha)}{\Gamma(\sum_{k\in \varphi^{c_i}} \tilde{N}_{ik} + J\alpha)}\right]. \label{eq: C,L}
\end{align}
In the above expressions, $k\in \varphi^i$ indicates iterating only over those $k$ for which node $v_k$ is in the $i$-th path in the DRT; and for each such $i$, there are $J$ such $k'$s. The count matrices $M\in \bbN_0^{I}, N\in \bbN_0^{V\times K}$ and $\tilde{N}\in \bbN_0^{m\times K}$ (where $\bbN_0=\bbN\cup \{0\}$) are defined as follows
\begin{align*}
    M_\ell &= \sum_{i\in[M]} 1(C_i = \ell),  \\
    N_{vk} &= \sum_{i\in[m], j\in[n]} 1\left(X_{i,j}=v, (C_i, L_{i,j})\in z^{-1}(k)\right), \\
    \tilde{N}_{ik} &= \begin{cases}
        0, &\text{if } k \text{ not in path } C_i \\
        \sum_{j\in[n]} 1(L_{i,j}=\ell), &\text{if node } k \text{ is at depth } \ell \text{ in path } C_i.  
    \end{cases} 
\end{align*}

Next, we shall describe the Gibbs sampler over just $C$ and $L$. For $L$, we have the following update
\begin{align}\label{eq: update L}
    p\left(L_{i,j}=\ell|L_{-(i,j)}, C, \doc^m\right) \propto \frac{N_{x_{i,j}, z(c_i,\ell)}^{-(i,j)} + \eta}{N_{\cdot z(c_i,\ell)}^{-(i,j)} + V\eta} \times \left(\tilde{N}_{i,z(c_i,\ell)}^{-(i,j)} + \alpha\right),
\end{align}
where $-(i,j)$ just indicates the count without taking the current assignment of $L_{i,j}$ into account.

The update for $C$ is a bit more involved since it affects multiple documents together. Note that if we change $C_i$ from $c'$ to $c$, then the second (product) term in Eq. \eqref{eq: C,L} stays the same (since $L_i$ stays the same — the topics along path $\varphi^{c_i}$ might change, but the counts along the depths stay the same).

\begin{align}\label{eq: update C}
    p\left(C_i=c|C_{-i}, L, \corpus\right) \propto (M_{c}^{-i} + \pi_0) \times \prod_{k\in[K]} \left(\frac{\Gamma(N_{\cdot k}^{-c_i} + V\eta)}{\Gamma(N_{\cdot k}^{(c_i=c)} + V\eta)} \times \prod_{v\in[V]} \frac{\Gamma(N_{v k}^{(c_i=c)} + \eta)}{\Gamma(N_{v k}^{-c_i} + \eta)}\right),
\end{align}
where $N_{vk}^{-c_i}$ is the corresponding count without considering the assignment for $c_i$ while $N_{vk}^{(c_i=c)}$ considers the particular assignment for $c_i$ being $c$. These count terms can be computed efficiently by simply saving the word-topic count matrix for each document.

\textit{Overall Sampler:} The Gibbs sampler performs the following updates per iteration:
\begin{enumerate}
    \item Update $L_{i,j}$ using Eq. \eqref{eq: update L} for $i\in[m], j\in[n]$
    \item Update $C_i$ using Eq. \eqref{eq: update C} for $i\in[m]$.
\end{enumerate}
Given $C, L$, we can estimate the parameters by using the following equations (posterior means of the corresponding variables given others)
\begin{align*}
    \hat{\pi}_\ell &= \frac{M_\ell + \pi_0}{m + I\pi_0}, \\
    \hat{\beta}_{ik} &= \frac{\tilde{N}_{ik} + \alpha}{\sum_{k\in \varphi^{c_i}}\tilde{N}_{ik} + J\alpha}, \\
    \hat{\theta}_{kv} &= \frac{N_{vk} + \eta}{N_{\cdot k} + V\eta}.
\end{align*}

To monitor the Gibbs sampler, we can approximate the likelihood of the data as the harmonic mean of $p(\corpus|C,L)$ based on a sample of size $S$ from $p(C,L|\corpus)$ \citep{kass1995bayes}, which can be drawn from the Gibbs sampler described above. We have implemented the sampler in Jax. To overcome the sampler getting stuck at local modes, we restarted the sampler several times in our simulations, and report using the chain corresponding to the highest log-likelihood of the data, approximated as discussed above.

\section{Experiments}\label{sec:numerical exp}
\subsection{Simulation studies}

We conducted a number of simulations to study the performance of the algorithms and gain further understanding into the latent structure estimation in our model. The primary objectives for the experiments are understanding the (a) estimation rate and (b) identifiability through model selection. Further details about the experiments are provided in the Appendix \ref{app:numerical exp}. For estimation rates, we consider the metric 
$$d_{L_2}(\rho, \rho') = \min_{\sigma\in\bbS_I} \sum_{i=1}^I \min_{\tau_i\in \bbS_J} \sum_{j\in[J]} \norm{\theta_{i,j} - \theta_{\sigma(i), \tau_i(j)}'}$$
for the ease of computation, compared to the theoretical union Hausdorff metric. Here $\theta_{i,j}$ is the topic corresponding to the node at level $j$ along path $i$ in the fixed DRT $\cT$. The intuition for this choice is clear --- for each pair of component polytopes, their `distance' is the sum of the distances between an optimal matching of their extreme points. Finally, the overall metric is optimally matching the components, following this metric. We show that $d_{\cU\cH}(\rho,\rho')\leq d_{L_2}(\rho,\rho')$ in the Appendix — hence, estimation rate in terms of $d_{L_2}$ provides an upper bound for the rate in terms of $d_{\cU\cH}$. 

\subsubsection*{Experiment 1: } In the first set of simulations, we consider a directed and rooted tree of size $I=2, J=3, K=5$ sharing only the root. Thus, in this model there are 2 component polytopes sharing one extreme point. We take $V=10$ as the vocabulary size. The true topics were drawn from a Dirichlet distribution $\text{Dir}_V(1.0)$ (uniform), the path probabilities were set to be uniform and $\alpha_0=0.8$. Two choices were taken for the number of words per document $n=50, 100$ and for each setting of $n$,  we selected $m$ on an equi-spaced grid from 200 to 5000 (8 values on a log scale). For each choice of $(m,n)$, the experiment was repeated $L=15$ times, each time drawing a corpus of size $(m,n)$ from the model (keeping the same parameter) and using the Gibbs sampler (we ran for 5500 iterations, dropping the first 5000 as burn-in and using a thinning of 10, giving us 50 samples) to estimate the topics. The DRT is shown in Figure \ref{fig: exp1exp2 setting} (top row left column) and the two images on the right show the true topic polytopes and the estimated topic polytope (for an instance in the case $n=50, m=200$), each point is a document (this is the document mean $\hat{\eta}_i$) projected on the first and second principal component directions (middle column) and second and third (right column), the solid triangles are the true component polytopes (sharing one common vertex), while the dotted blue triangles are the estimated polytopes. 

\begin{figure}
    \centering
    \includegraphics[width=0.9\linewidth]{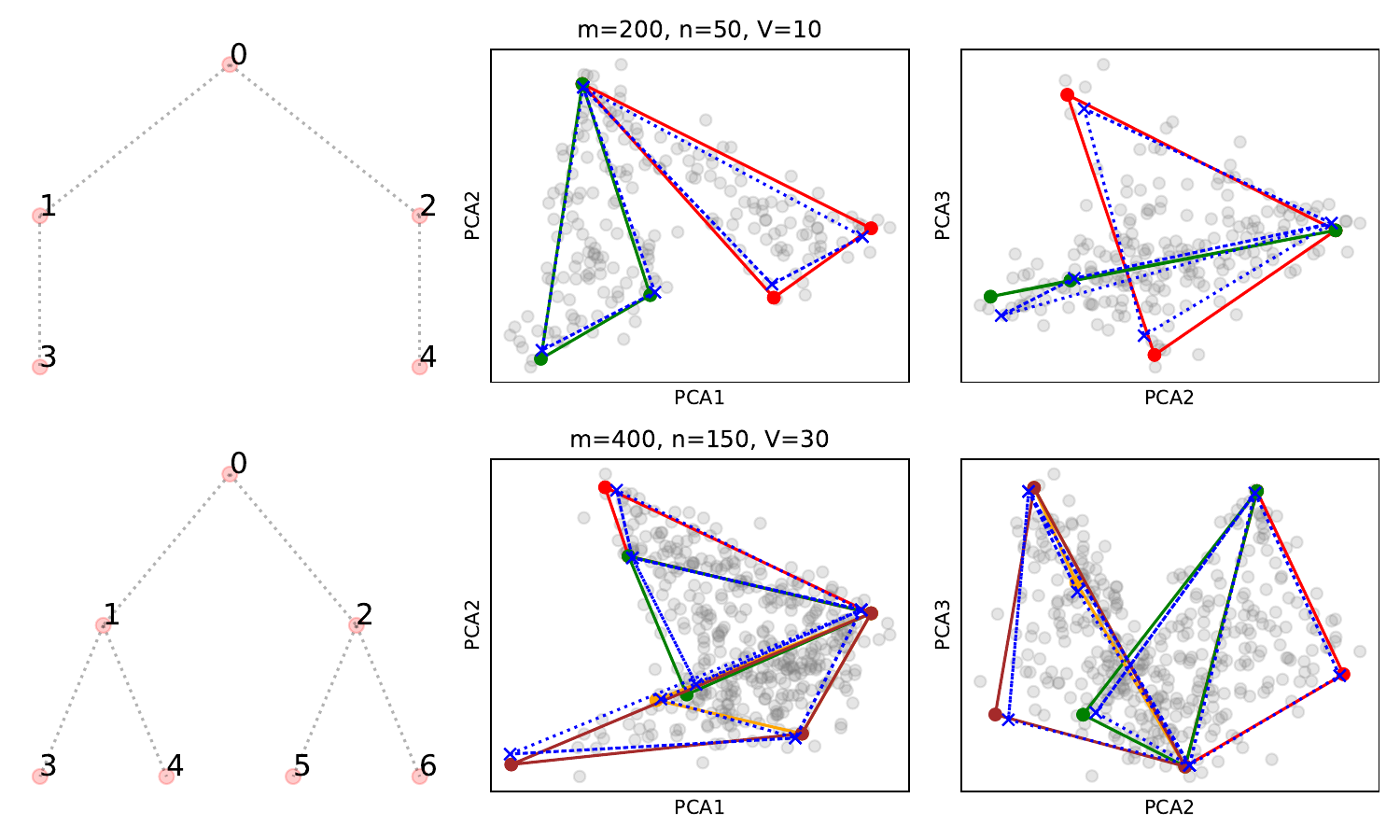}
    \caption{Illustrations for Experiments 1 (top row) and 2 (bottom row): The left-most column shows the underlying DRTs, the middle and right columns show the true polytopes, estimated polytopes and documents projected on the labelled PCA-subspaces (for one particular run of the experiments, $m,n$ given at the top respectively). Each gray circle represents a document (as the document mean), the true polytopes are shown in solid colors and the estimated polytopes in blue dotted lines.}
    \label{fig: exp1exp2 setting}
\end{figure}

For each instance, the performance of our methods was measured by the $d_{L_2}$ metric, computed on the sample of size 50. The result is shown in Figure \ref{fig: exp1 exp2 rates} (Left), where we plot the $\log d_{L_2}$ against $\log m$ (color indicates the two choices of $n$). We can see a similar trend for both $n$ (somehow in this simulation, $n=100$ had higher standard errors compared to $n=50$ with very similar means) — the estimation rate of the topics under $d_{L_2}$ metric is indeed very close to a parametric rate, as the slope suggests (obtained by fitting a linear regression model). Our theory in Section \ref{sec:parameter estimation} showed an upper bound of $[\log m(1/m+1/n)]^{1/4}$ (in the union Hausdorff metric) in this model as $\alpha<1$ and $J-1=2$.  The experiment suggests that this upper bound on the parameter estimates may not be tight. Furthermore, the simulation shows that sending $n\to\infty$ is probably not required in this case.

\subsubsection*{Experiment 2: } In the second set of simulations, we consider another DRT of size $I=4, J=3, K=7$ shown in Figure \ref{fig: exp1exp2 setting} (bottom row left column). Thus, in this model there are 4 component polytopes, with one topic shared across all 4 components and two different pairs of polytopes share one common edge — this is shown as the solid colored triangles in the middle and right columns of the bottom row in Figure \ref{fig: exp1exp2 setting}. We take $V=30$ as the vocabulary size in this experiment, $n\in\{150, 300\}$ and $m$ increasing from $400$ to $4000$ (5 values chosen equi-spaced in log scale). The true parameters are generated similarly as in experiment 1, with $\alpha_0=1.0$. We perform a similar study as in experiment 1 (with 15 repetitions for each choice of $m,n$) and Figure \ref{fig: exp1 exp2 rates} (middle) shows that the results, in terms of the rate, are also similar to experiment 1 — the plot shows $\log d_{L_2}$ against $\log m$, having a slope of around $-1/2$, verifying a parametric rate. However, we note that in this case, there is a marked difference between $n=150$ and $n=300$ case — this tells that while $n\to\infty$ might not be necessary for consistent parameter estimation, increasing $n$ improves the estimation rate of the parameters. 

We like to highlight an interesting challenge for this class of model. Although we fit the model using the true underlying DRT, it does not ensure that the Gibbs sampler converges to a \textit{correct sharing structure} within our prescribed computational limit. As an example, compare the plots in Figures \ref{fig: exp1exp2 setting} (bottom row middle column) and \ref{fig: exp1 exp2 rates} (right most column) — the blue dotted triangles are the estimated components polytopes. The Gibbs sampler was run for 5500 iterations for the former, while even with 10000 iterations, the Gibbs sampler in the later case did not improve beyond the shown structure. Note that in the latter plot, although the underlying DRT is the same and there is an estimated topic nearby each of the true topics, the sharing structure is not at all close to the true sharing structure. For our simulations, we used 8 parallel chains and chose the one based on highest log likelihood. This also demonstrates that $d_{L_2}$ metric by itself does not give a full picture of the accuracy of the estimation of the latent structure. In experiments 1 and 2, we checked for this accuracy in the following way: we first find a permutation $\sigma\in\bbS_K$ minimizing $\sum_k \Vert\theta_k - \hat{\theta}_{\sigma(k)}\Vert$, i.e., an optimal matching of the estimated topics to the true topics. Then we checked whether estimated topics play the same role in the estimated topic hierarchy as the true counterparts (based on the optimal matching) — i.e., for example, if $\theta_1$ is the root of the true topic hierarchy (the vertex shared by all the 4 polytopes), then whether $\hat{\theta}_{\sigma(1)}$ is the vertex shared by all the estimated polytopes (also similar checks for all the other topics). In Figure \ref{fig: exp1 exp2 rates} (right), this is not true since the true root of the topic hierarchy is the rightmost point, which is closest to an estimated topic (right-most blue topic) which is a topic not shared by any two polytopes (hence a leaf). In all of the instances in experiments 1 and 2 (using 8 parallel chains), the true sharing structure was estimated correctly in this sense.

\begin{figure}
    \centering
    \includegraphics[width=0.96\linewidth]{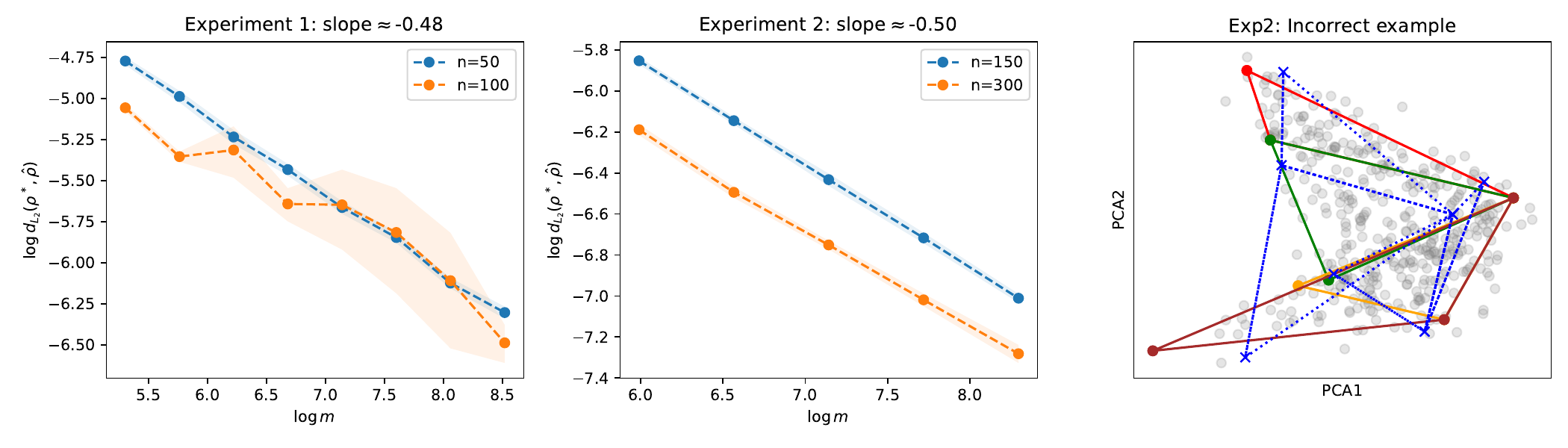}
    \caption{(Left) Results for Experiment 1, (Middle) Results for Experiment 2 and (Right) One example of Gibbs sampler getting stuck at a local mode (compare with Figure \ref{fig: exp1exp2 setting} bottom row middle column).}
    \label{fig: exp1 exp2 rates}
\end{figure}

These two experiments validate that the latent topic hierarchy can indeed be consistently estimated from the corpus under our model. While we provide an upper bound for the estimation rate for this latent structure, the experiment demonstrate that it is not tight. We conjecture that the estimation rate is parametric for any fixed $n$, i.e., $\epsilon_{m,n}\approx c_n / \sqrt{m}$, where the constant $c_n$ depends on $n$ — in particular, higher $n$ would correspond to a lower $c_n$, however, the estimation is consistent as long as $n\geq n_0$ for some minimal $n_0$. The nature of this dependence remains unknown, but there is some hint about such dependence based on recent results for simpler models such as the LDA \citep{anandkumar2012spectral} and finite mixtures of product distributions \citep{wei2022convergence}.


\begin{figure}
    \centering
    \includegraphics[width=0.8\linewidth]{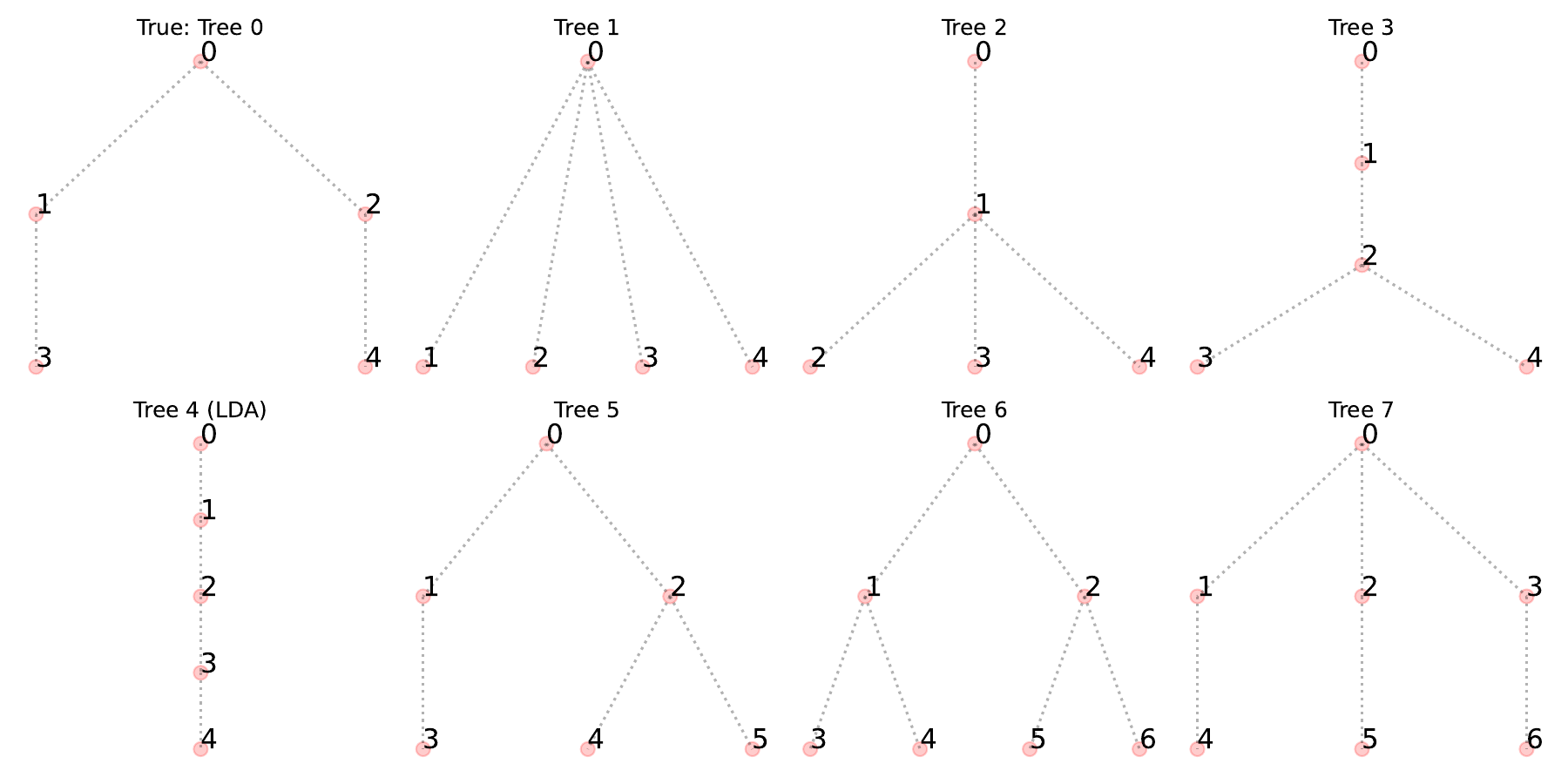}
    \caption{The collection of DRTs considered for Experiment 3.}
    \label{fig: exp3 DRTs}
\end{figure}

\begin{figure}
    \centering
    \includegraphics[width=0.98\linewidth]{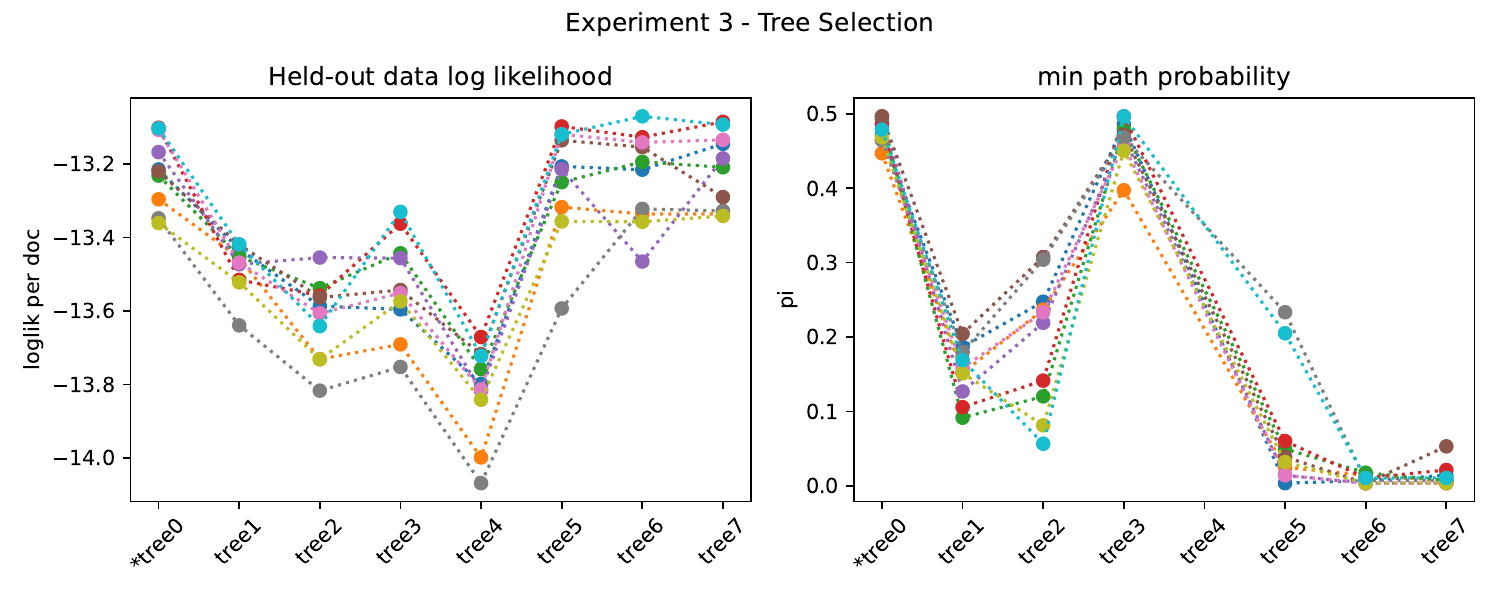}
    \caption{Results from Experiment 3 with \textit{tree0} as the true DRT: (Left) held-out log likelihood per document for each tree (each color represents a particular corpus); (Right) minimum path probability for each tree (\textit{tree4} is left out, since it is an LDA, equivalent to a tree with a single path and hence, the path probability is 1).}
    \label{fig: exp3 result}
\end{figure}

\subsubsection*{Experiment 3: } The third set of simulations is designed to demonstrate a DRT selection problem. The true underlying DRT is the same as in Experiment 1. However, we use 8 different choices of DRT (the whole collection is shown in Figure \ref{fig: exp3 DRTs}) — out of these, the first 5 all have $K=5$ vertices, while the last 3 have $K>5$. The tree named \textit{tree4} has just $I=1$ path and is equivalent to a LDA model with $K=5$ topics (notice that the true DRT has 5 topics). We generate data using the true DRT and fit the model using each of the considered trees. In this experiment, we fix $n=60, m=400$ and $V=10$. For each simulation (which is repeated $10$ times), we train the model with $m=280$ ($70\%$ documents) randomly chosen documents and keep the remaining $30\%$ as held-out data — finally, we compute the log-likelihood per document for the held-out dataset (which is often called perplexity for topic models). The held-out log likelihood for each tree across each experiment is shown in Figure \ref{fig: exp3 result} (left) — each colored line represents results from a particular corpus instance (out of the 10). We make the following observations: among the first 5 trees, the true DRT (\textit{tree0}) has the highest held-out log likelihood, however, the last 3 have almost similar held-out log likelihood — this makes sense, since the first 5 all have $K=5$ topics, while the remaining have higher $K$ (they are all trees with the true DRT being a sub-tree of it) and higher $I$ (number of paths) than the true DRT. As a result, it might be possible that some of the paths are assigned very low probability — to check this, we also plot $\min_i \hat{\pi}_i$, the minimum value of the estimated path-probabilities. We note that the last 3 trees have significantly smaller values compared to \textit{tree0}. Furthermore, although \textit{tree1} and \textit{tree2} have same $K$ as \textit{tree0}, they have $I=4$ and $I=3$ respectively (true $I=2$), and thus, they too have significantly low minimum path probabilities. This experiment shows that methods involving held-out log likelihood can be used for the DRT selection task. The results in this experiment support our study of identifiability in Section \ref{sec:identifiability}.

\subsection{Real data example}



In this section, we provide an analysis of a subset of the 2016 New York Times articles using our model to explore the hierarchical structure among latent topics. The original dataset contains 8887 documents across multiple categories (the categories were extracted from the news articles URLs) from articles during April to June, 2016. For this study, to keep the presentation simple and interpretable, a small subset of these documents were selected corresponding to 4 categories (\textit{world-middleeast, world-americas, sports-hockey, sport-soccer}). This resulted in a corpus of size $m=641$. As pre-processing, all words were converted to lower case and lemmatized, punctuation and words which occurred in less than 5 documents or more than $60\%$ of the documents were removed. Finally, the top 500 words were selected from the remaining vocabulary corresponding to highest document frequencies. This effectively reduced the vocabulary size from over 100,000 to $V=500$. Each document was still quite large, the average size of a document turned out to be around $n\approx 200$. This resulted in a corpus with $131,415$ words. 

A visualization of the data is given in Figure \ref{fig: nyt example intro} (left two figures). Each document is represented as a point in the two dimensional space spanned by the first and second principal components (left) and second and third (right) — the points are colored by their true category. Note that for model fitting, information about such category labels are not used in any way. In Figure \ref{fig: nyt example intro}, observe that the documents lie on roughly two triangular polytopes, sharing a common vertex. We also note that the two sides (one side contains the orange and red point and one contains the blue and green) clearly represent two very different topics — world and sports. Within a topic, it is harder to differentiate visually between the subtopics. Nonetheless, this provides evidence into the geometric viewpoint that we discuss throughout this work — namely, to understand the latent structure as arising from a collection of topic polytopes, restricted to certain vertices being shared.

The data set was randomly splitted into $70\%$ for training and $30\%$ as held-out. We considered a variety of DRTs (given in the Appendix) with $J\leq 4$ and $I\leq 6$. For each DRT, model fitting was achieved using the collapsed Gibbs sampler with 8 restarts —- we dropped first 5000 iterates as burn-in and then retained 50 samples at intervals of 10 (hence, a total of 5500 iterations for the sampler). For DRT selection, we considered a few criteria motivated by interpretability and considerations arising from this work. Firstly, the log-likelihood per document was computed on the held-out data for each DRT. We also collect measures to check (i) redundancy of a topic within a component polytope, (ii) separation across the component polytopes and (iii) path probabilities. For (i), we computed the \textit{width} of each component polytope, defined as the minimum projection distance of a vertex to the affine hull spanned by the remaining vertices, and minimum edge length. For (ii), we computed minimum minimal matching distance $d_{M}$ between distinct component polytopes, where 
$$d_{M}(\cS_i, \cS_j) = \max_{\theta\in\extr \cS_i}\min_{\theta'\in\extr \cS_j} \norm{\theta-\theta'} \vee \max_{\theta'\in\extr \cS_j}\min_{\theta\in\extr \cS_i} \norm{\theta-\theta'},$$
which is shown to be equivalent to the Hausdorff distance between the component polytopes in \cite{NguyenPosteriorContractionPopulation2015}. Since each of the component polytopes has the same dimension for each DRT, we also computed the Grassmanian distance between the affine hulls of the component polytopes (angle between the affine spaces). For (iii), we collected the minimum estimated path probability for each DRT. The full results are shown in Appendix. Based on the results, the DRT was selected, see in Figure \ref{fig: nyt topic hierarchy}, with top 10 words for each estimated topic shown in the corresponding vertices, and also the corresponding path probabilities.

\begin{figure}
    \centering
    \includegraphics[width=0.98\linewidth]{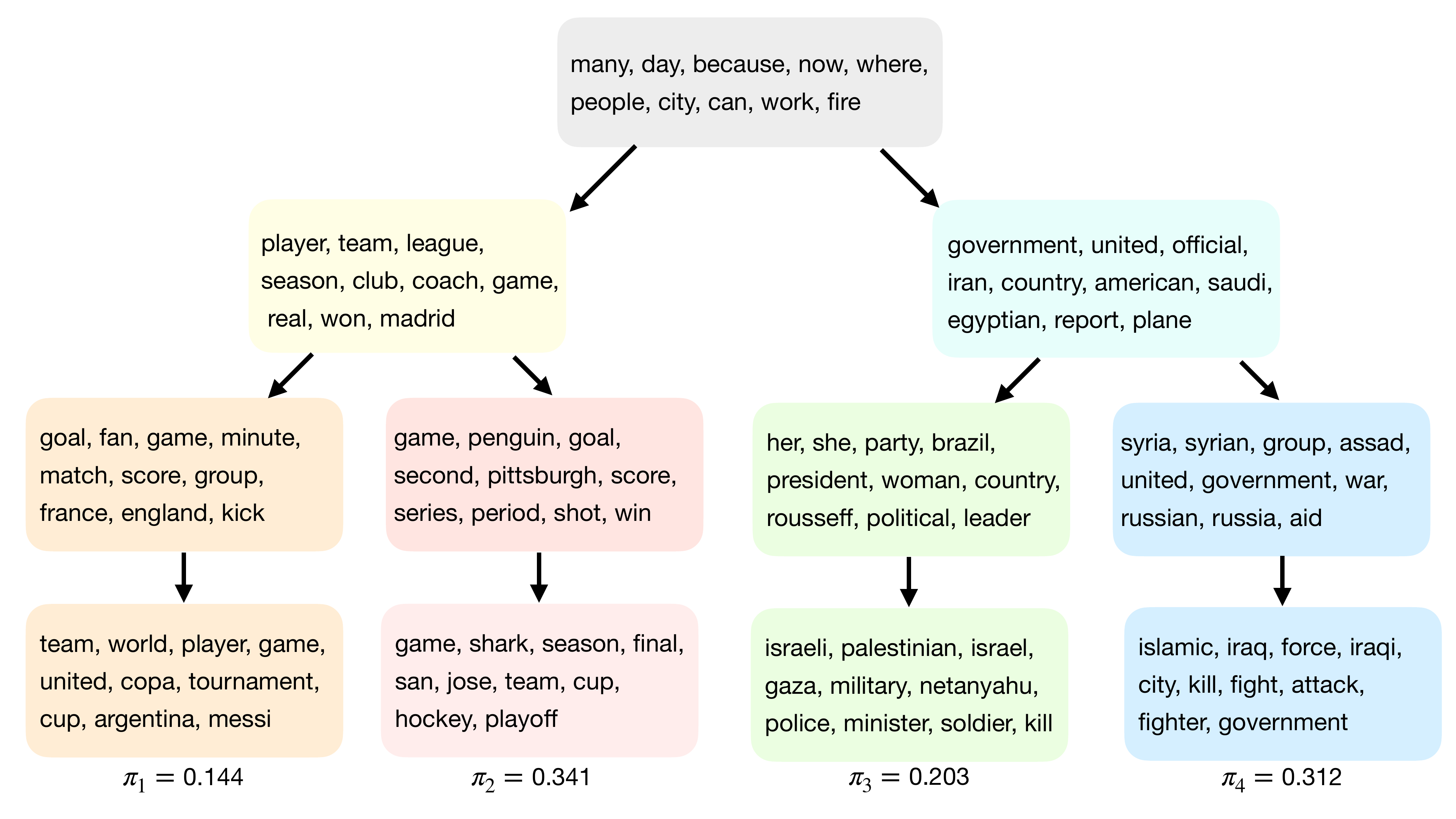}
    \caption{Estimated Topic Hierarchy for NYT articles (showing 10 top words for each topic) — DRT with $I=4, J=4, K=11$.}
    \label{fig: nyt topic hierarchy}
\end{figure}

The estimated topic hierarchy shown in Figure \ref{fig: nyt topic hierarchy} captures quite nicely the corpus structure. We emphasize that the news categories (shown in colors in Figure \ref{fig: nyt example intro}) were not used while fitting the model — the tree-directed topic model captured the topics in a completely unsupervised manner from the corpus. The root node naturally carries words that can occur commonly in all types of documents. The topics in level 1 (light yellow and light blue) evidently capture the two major topics in the corpus, namely, \textit{sports} and \textit{world}. For the sports topic, the top 10 words are mostly common to both soccer or hockey, except \textit{real, madrid}, which originate from the famous soccer club Real Madrid. For the \textit{world} topic, we see however, that there is slightly higher influence of middle-eastern countries (e.g., we find the words \textit{iran, egyptian, saudi}). The topics in the bottom two layers of hierarchy capture the unique words for each suptopic. For this DRT, since each node in the penultimate level only has 1 child each, these two levels are exchangeable from the model's perspective. Indeed, the child nodes for the sports topic divide into soccer (light orange boxes) and hockey (light red boxes). The topics for hockey capture common team names such as \textit{Pittsburgh Penguins} and \textit{San Jose Sharks} while the \textit{soccer} topics contain words for famous European teams such as \textit{England, France} and also famous player like \textit{Lionel Messi} from \textit{Argentina}. We note that words \textit{goal, game, score} appear in both of these paths -- ideally, we would want the parent topic to put more probability to these words since they appear in both. We note that the word \textit{united} appears in both the soccer path, as well as the \textit{world} topic -- however, upon closer inspection, these two in fact have very different origin. The former \textit{united} comes from \textit{Manchester United} (a popular team in soccer), while the latter comes from the \textit{United States of America} and \textit{United Arab Emirates}. On the right side in the \textit{world} sub-tree, we find the paths behaving slightly differently. The left child of \textit{world} is clearly talking about the political situation in \textit{Brazil}  as the at-the-time woman president \textit{Dilma Rousseff} was succeeded by \textit{Michael Temer} in 2016. The child node of this is however again from the middle-east, representing the \textit{Israel-Palestine} conflict about the \textit{Gaza} strip (\textit{Benjamin Netanyahu} has been the Prime Minister of Israel). Closer inspection revealed many of the articles where about struggle of Palestinian women during these troubled times, and hence possibly the model placed this node in the same path as its parent was owing to the fact that the latter puts high probability to words like \textit{her, she, woman}. The right child of the \textit{world} topic is mostly about \textit{Syrian} civil war (\textit{Bashar Al-Assad} is the President of Syria) and \textit{Russian} aid. Its child is mostly about the \textit{War in Iraq}, between Iraq and its allies and the \textit{Islamic State}, making this fourth path a very distinct \textit{middle-east} sub-topic of \textit{world}. However, the third path has a mix of both \textit{American} and \textit{middle-east} world politics. The estimated topic polytopes are shown in Figure \ref{fig: tree20} — each of the component polytopes are of dimension $3$ (tetrahedrons) and we can see (middle figure) how it captures the two distinct clouds corresponding to the sports and world categories, sharing a vertex in the middle.
\begin{figure}
    \centering
    \includegraphics[width=0.97\linewidth]{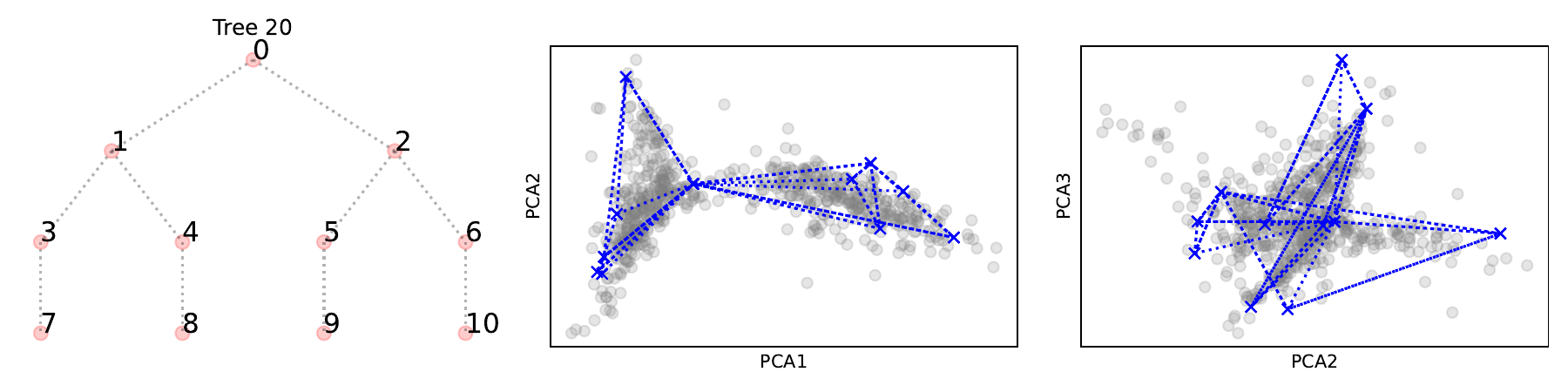}
    \caption{DRT and estimated topic polytopes for the model whose hierarchy is shown in Figure \ref{fig: nyt topic hierarchy}.}
    \label{fig: tree20}
\end{figure}

Next, we present the results corresponding to a DRT which resembles our existing knowledge of the news category's \emph{structure} in the corpus (i.e., a binary tree of 7 vertices without further information). The estimated topics are shown in Figure \ref{fig: nyt hierarchy tree7} and the estimated component topic polytopes were shown in Figure \ref{fig: nyt example intro} (right two figures). In this case, we find that the four paths majorly capture the four news categories and correctly captures the hierarchy. Compared to the previous tree, the top level topics are comparable, while there are also a few notable differences: firstly, for the \textit{soccer} path, the previous tree had a distinction between \textit{European soccer} (penultimate node) and \textit{South American soccer} (node in the bottom level), while for this tree, they are combined together. For the path corresponding to \textit{hockey}, it seems both the last two level nodes in the previous tree were about the Stanley Cup and Pittsburgh Penguins winning over San Jose Sharks in the final. The third path in the previous tree was unique since it combined documents from both American and Middle-Eastern world politics; in the present tree, they are put in different paths. Furthermore, there is no clearly identifiable topic about the War in Iraq, as we have seen in the previous tree. Even when a single best model is challenging to determine for real corpus with complex and latent contextual information such as the current example, we have demonstrated that interpretability of the learned models is feasible according to the varying model constraints. It is quite interesting that the latter tree captures correctly the news category tag assigned to the documents, and moreover discovers predominant topics from these categories of NYT articles from April to June 2016. 
The former tree due to its larger size also captured additionally meaningful topics and information of interest (including the war in Iraq, the distinction between European and South American soccer and connection between Syria-Palestine-related articles and Brazil politics-related articles).

\begin{figure}
    \centering
    \includegraphics[width=0.98\linewidth]{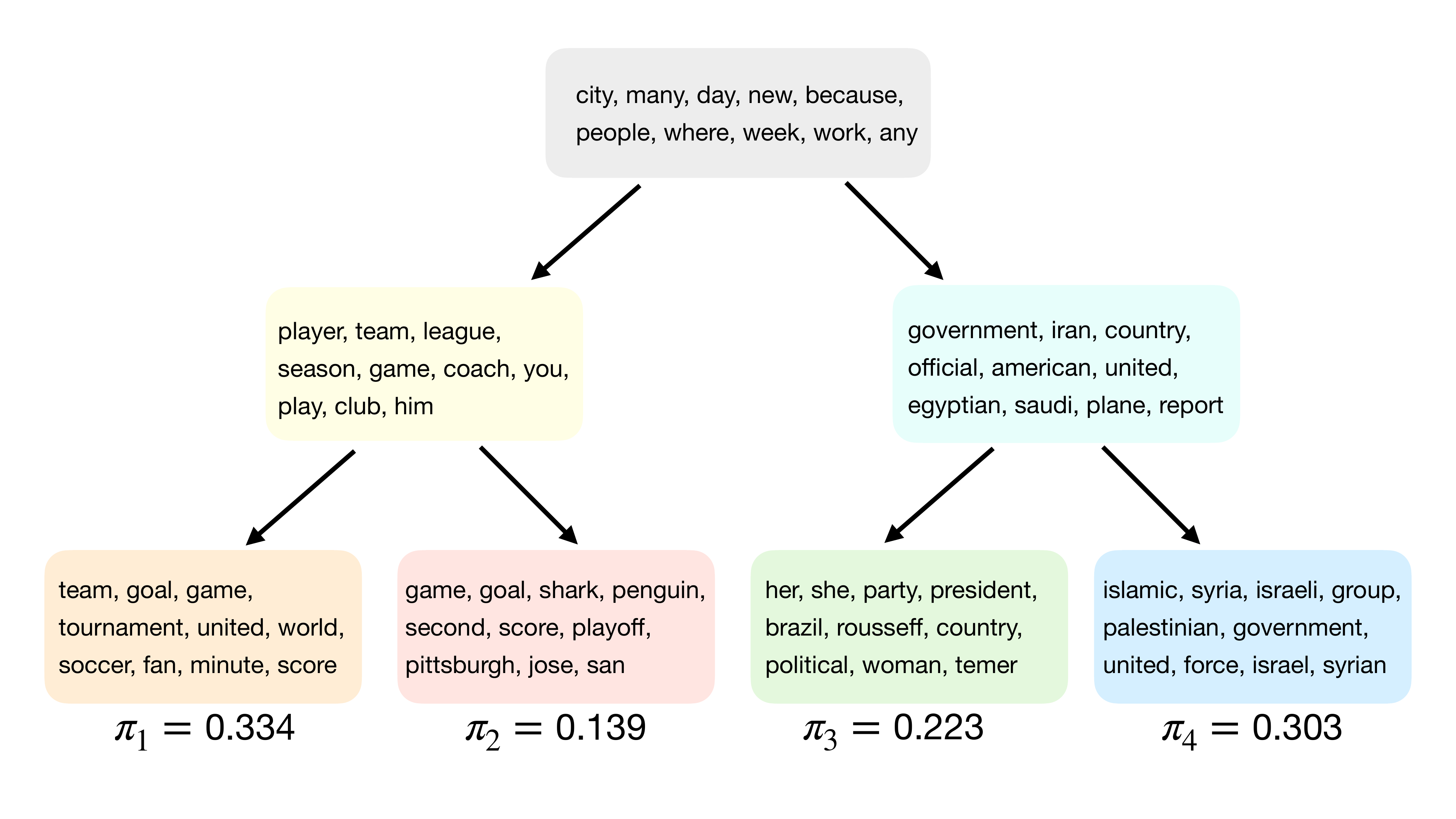}
    \caption{Estimated Topic Hierarchy for NYT articles (showing 10 top words for each topic) — DRT with $I=4,J=3,K=7$}
    \label{fig: nyt hierarchy tree7}
\end{figure}
\section{Conclusion}
\label{sec:conclude}
We studied a class of topic models which enables the learning of a latent hierarchy of topics enforced through a directed rooted tree (DRT) structure. The probabilistic model was formulated using the underlying tree and a topic map, which provide the embedding of tree-structured hierarchy into the space of distributions on vocabulary (the vocabulary simplex). It was established in this work that the latent topic hierarchy is identifiable under suitable regularity conditions, and can be consistently estimated from the text corpus to provide meaningful insight into the latent structure of the corpus. Examples of non-identifiability were provided when the regularity conditions are not satisfied. Our approach relies on exploiting the geometric structures underlying the model, which can be seen as representing a collection of topic polytopes sharing particular extreme points or faces constrained by a tree-structured hierarchy. As we have demonstrated, such a structure adapts quite well to complex text corpora, where the existing topic models such as the LDA were rather inadequate. 
While our model may be seen as a mixture of the LDA-induced distributions, it is the restriction on the sharing patterns among the LDA components that makes the model significantly richer and more interesting, and the theoretical analysis more challenging.
We validated our model by numerical experiments and demonstrated its usefulness in the analysis of New York Times articles, for which the latent topics captured by the topic hierarchy were learned and meaningfully interpreted. 

This paper opens up several venues for future work. On the computational side, we used collapsed Gibbs sampler with a known DRT — utilizing the geometric insights obtained from this paper should allow developing more efficient geometric algorithms for estimating the topic polytopes.  In terms of theory, deriving tighter upper bounds than what were presented in Section \ref{sec:parameter estimation} and obtaining the minimal $n_0$ for which this model is identifiable, thereby extending the theory in Section \ref{sec: identifiability2}, are interesting problems; we believe that new techniques would be required to address such questions. It is also of interest to consider estimating the tuning parameter, i.e., parameter $\alpha$, for the distribution on vocabulary simplex, or at least investigating into the effect of this hyperparameter on estimation. Given the highly general identifiability results presented herein, consistent estimation for the underlying DRT is also a direction worth pursuing further. 

\bibliographystyle{apalike} 
\bibliography{ref}   

\begin{thebibliography}{}

\bibitem[Anandkumar et~al., 2012]{anandkumar2012spectral}
Anandkumar, A., Foster, D.~P., Hsu, D.~J., Kakade, S.~M., and Liu, Y.-K.
  (2012).
\newblock A spectral algorithm for latent dirichlet allocation.
\newblock {\em Advances in neural information processing systems}, 25.

\bibitem[Arora et~al., 2012]{arora2012learning}
Arora, S., Ge, R., and Moitra, A. (2012).
\newblock Learning topic models--going beyond svd.
\newblock In {\em 2012 IEEE 53rd annual symposium on foundations of computer
  science}, pages 1--10. IEEE.

\bibitem[Bing et~al., 2020a]{bing2020fast}
Bing, X., Bunea, F., and Wegkamp, M. (2020a).
\newblock A fast algorithm with minimax optimal guarantees for topic models
  with an unknown number of topics.

\bibitem[Bing et~al., 2020b]{bing2020optimal}
Bing, X., Bunea, F., and Wegkamp, M. (2020b).
\newblock Optimal estimation of sparse topic models.
\newblock {\em Journal of machine learning research}, 21(177):1--45.

\bibitem[Blei and Lafferty, 2006a]{blei2006correlated}
Blei, D. and Lafferty, J. (2006a).
\newblock Correlated topic models.
\newblock {\em Advances in neural information processing systems}, 18:147.

\bibitem[Blei et~al., 2010]{BleiEtAlNestedChineseRestaurant2010}
Blei, D.~M., Griffiths, T.~L., and Jordan, M.~I. (2010).
\newblock The {{Nested Chinese Restaurant Process}} and {{Bayesian
  Nonparametric Inference}} of {{Topic Hierarchies}}.
\newblock {\em Journal of the ACM}, 57(2):1--30.

\bibitem[Blei and Lafferty, 2006b]{BleiLaffertyDynamicTopicModels2006}
Blei, D.~M. and Lafferty, J.~D. (2006b).
\newblock Dynamic topic models.
\newblock In {\em Proceedings of the 23rd International Conference on
  {{Machine}} Learning - {{ICML}} '06}, pages 113--120, {Pittsburgh,
  Pennsylvania}. {ACM Press}.

\bibitem[Blei et~al., 2003]{BleiEtAlLatentDirichletAllocation2003}
Blei, D.~M., Ng, Andrew, Y., and Jordan, M.~I. (2003).
\newblock Latent {{Dirichlet Allocation}}.
\newblock {\em Journal of Machine Learning Research}, 3(January):30.

\bibitem[Chen et~al., 2023]{chen2023learning}
Chen, Y., He, S., Yang, Y., and Liang, F. (2023).
\newblock Learning topic models: Identifiability and finite-sample analysis.
\newblock {\em Journal of the American Statistical Association},
  118(544):2860--2875.

\bibitem[Choo et~al., 2013]{6634167}
Choo, J., Lee, C., Reddy, C.~K., and Park, H. (2013).
\newblock Utopian: User-driven topic modeling based on interactive nonnegative
  matrix factorization.
\newblock {\em IEEE Transactions on Visualization and Computer Graphics},
  19(12):1992--2001.

\bibitem[Das et~al., 2015]{das2015gaussian}
Das, R., Zaheer, M., and Dyer, C. (2015).
\newblock Gaussian lda for topic models with word embeddings.
\newblock In {\em Proceedings of the 53rd Annual Meeting of the Association for
  Computational Linguistics and the 7th International Joint Conference on
  Natural Language Processing (Volume 1: Long Papers)}, pages 795--804.

\bibitem[Ghosal and van~der Vaart,
  2017]{GhosalVaartFundamentalsNonparametricBayesian2017}
Ghosal, S. and van~der Vaart, A.~W. (2017).
\newblock {\em Fundamentals of Nonparametric {{Bayesian}} Inference}.
\newblock Number~44 in Cambridge Series in Statistical and Probabilistic
  Mathematics. {Cambridge University Press}, {Cambridge ; New York}.

\bibitem[Griffiths and Steyvers, 2004]{griffiths2004finding}
Griffiths, T.~L. and Steyvers, M. (2004).
\newblock Finding scientific topics.
\newblock {\em Proceedings of the National academy of Sciences},
  101(suppl\_1):5228--5235.

\bibitem[Guha et~al., 2021]{guha2021posterior}
Guha, A., Ho, N., and Nguyen, X. (2021).
\newblock On posterior contraction of parameters and interpretability in
  bayesian mixture modeling.
\newblock {\em Bernoulli}, 27(4):2159--2188.

\bibitem[Harary, 1969]{harary1969graph}
Harary, F. (1969).
\newblock {\em Graph Theory}.
\newblock Addison-Wesley Series in Maethematics.

\bibitem[Ho and Nguyen, 2016]{ho2016strong}
Ho, N. and Nguyen, X. (2016).
\newblock On strong identifiability and convergence rates of parameter
  estimation in finite mixtures.

\bibitem[Hoffman et~al., 2010]{hoffman2010online}
Hoffman, M., Bach, F., and Blei, D. (2010).
\newblock Online learning for latent dirichlet allocation.
\newblock {\em advances in neural information processing systems}, 23.

\bibitem[Javadi and Montanari, 2019]{javadi2019nonnegative}
Javadi, H. and Montanari, A. (2019).
\newblock Nonnegative matrix factorization via archetypal analysis.
\newblock {\em Journal of the American Statistical Association}.

\bibitem[Kass and Raftery, 1995]{kass1995bayes}
Kass, R.~E. and Raftery, A.~E. (1995).
\newblock Bayes factors.
\newblock {\em Journal of the american statistical association},
  90(430):773--795.

\bibitem[Ke and Wang, 2024]{ke2024using}
Ke, Z.~T. and Wang, M. (2024).
\newblock Using svd for topic modeling.
\newblock {\em Journal of the American Statistical Association},
  119(545):434--449.

\bibitem[Kherwa and Bansal, 2017]{8455018}
Kherwa, P. and Bansal, P. (2017).
\newblock Latent semantic analysis: An approach to understand semantic of text.
\newblock In {\em 2017 International Conference on Current Trends in Computer,
  Electrical, Electronics and Communication (CTCEEC)}, pages 870--874.

\bibitem[Li and McCallum, 2006]{li2006pachinko}
Li, W. and McCallum, A. (2006).
\newblock Pachinko allocation: Dag-structured mixture models of topic
  correlations.
\newblock In {\em Proceedings of the 23rd international conference on Machine
  learning}, pages 577--584.

\bibitem[Lijoi et~al., 2023]{lijoi2023flexible}
Lijoi, A., Pr{\"u}nster, I., and Rebaudo, G. (2023).
\newblock Flexible clustering via hidden hierarchical dirichlet priors.
\newblock {\em Scandinavian Journal of Statistics}, 50(1):213--234.

\bibitem[Nguyen, 2013]{nguyen2013convergence}
Nguyen, X. (2013).
\newblock Convergence of latent mixing measures in finite and infinite mixture
  models.

\bibitem[Nguyen, 2015]{NguyenPosteriorContractionPopulation2015}
Nguyen, X. (2015).
\newblock Posterior contraction of the population polytope in finite admixture
  models.
\newblock {\em Bernoulli}, 21(1):618--646.

\bibitem[Nguyen, 2016]{NguyenBorrowingStrenghHierarchical2016}
Nguyen, X. (2016).
\newblock Borrowing strengh in hierarchical {{Bayes}}: {{Posterior}}
  concentration of the {{Dirichlet}} base measure.
\newblock {\em Bernoulli}, 22(3).

\bibitem[Pritchard et~al., 2000]{PritchardEtAlInferencePopulationStructure2000}
Pritchard, J.~K., Stephens, M., and Donnelly, P. (2000).
\newblock Inference of {{Population Structure Using Multilocus Genotype Data}}.
\newblock {\em Genetics}, 155(2):945--959.

\bibitem[Shalit et~al., 2013]{shalit2013modeling}
Shalit, U., Weinshall, D., and Chechik, G. (2013).
\newblock Modeling musical influence with topic models.
\newblock In {\em International Conference on Machine Learning}, pages
  244--252. PMLR.

\bibitem[Tang et~al., 2014]{TangEtAlUnderstandingLimitingFactors2014}
Tang, J., Meng, Z., Nguyen, X., Mei, Q., and Zhang, M. (2014).
\newblock Understanding the {{Limiting Factors}} of {{Topic Modeling}} via
  {{Posterior Contraction Analysis}}.
\newblock In {\em Proceedings of the 31st {{International Conference}} on
  {{Machine Learning}}}, pages 190--198. {PMLR}.

\bibitem[Teh et~al., 2006]{TehEtAlHierarchicalDirichletProcesses2006}
Teh, Y.~W., Jordan, M.~I., Beal, M.~J., and Blei, D.~M. (2006).
\newblock Hierarchical {{Dirichlet Processes}}.
\newblock {\em Journal of the American Statistical Association},
  101(476):1566--1581.

\bibitem[Valle et~al., 2020]{valle2020topic}
Valle, F., Osella, M., and Caselle, M. (2020).
\newblock A topic modeling analysis of tcga breast and lung cancer
  transcriptomic data.
\newblock {\em Cancers}, 12(12):3799.

\bibitem[Villani et~al., 2009]{villani2009optimal}
Villani, C. et~al. (2009).
\newblock {\em Optimal transport: old and new}, volume 338.
\newblock Springer.

\bibitem[Wang and Blei, 2009]{wang2009variational}
Wang, C. and Blei, D. (2009).
\newblock Variational inference for the nested chinese restaurant process.
\newblock {\em Advances in Neural Information Processing Systems}, 22.

\bibitem[Wang et~al., 2011]{wang2011online}
Wang, C., Paisley, J., and Blei, D.~M. (2011).
\newblock Online variational inference for the hierarchical dirichlet process.
\newblock In {\em Proceedings of the fourteenth international conference on
  artificial intelligence and statistics}, pages 752--760. JMLR Workshop and
  Conference Proceedings.

\bibitem[Wang, 2019]{WangConvergenceRatesLatent2019}
Wang, Y. (2019).
\newblock Convergence rates of latent topic models under relaxed
  identifiability conditions.
\newblock {\em Electronic Journal of Statistics}, 13(1):37--66.

\bibitem[Wei and Nguyen, 2022]{wei2022convergence}
Wei, Y. and Nguyen, X. (2022).
\newblock Convergence of de finetti’s mixing measure in latent structure
  models for observed exchangeable sequences.
\newblock {\em The Annals of Statistics}, 50(4):1859--1889.

\bibitem[Wong and Shen, 1995]{wong1995probability}
Wong, W.~H. and Shen, X. (1995).
\newblock Probability inequalities for likelihood ratios and convergence rates
  of sieve mles.
\newblock {\em The Annals of Statistics}, pages 339--362.

\bibitem[Yurochkin et~al., 2019]{YurochkinEtAlDirichletSimplexNest2019}
Yurochkin, M., Guha, A., Sun, Y., and Nguyen, X. (2019).
\newblock Dirichlet {{Simplex Nest}} and {{Geometric Inference}}.
\newblock In {\em Proceedings of the 36th {{International Conference}} on
  {{Machine Learning}}}, volume~97, pages 7262--7271, {Long Beach, CA}.

\bibitem[Yurochkin and Nguyen,
  2016]{YurochkinNguyenGeometricDirichletMeans2016}
Yurochkin, M. and Nguyen, X. (2016).
\newblock Geometric {{Dirichlet Means}} algorithm for topic inference.
\newblock In {\em Advances in {{Neural Information Processing Systems}}
  ({{NIPS}} 2016)}, pages 2505--2513, {Barcelona, Spain}. {Curran Associates}.

\end{thebibliography}

\newpage
\begin{appendices}
\section{DRTs and Isomorphisms}
\begin{figure}
    \centering
    \includegraphics[clip, trim=6cm 12cm 8cm 10cm, width=0.98\textwidth]{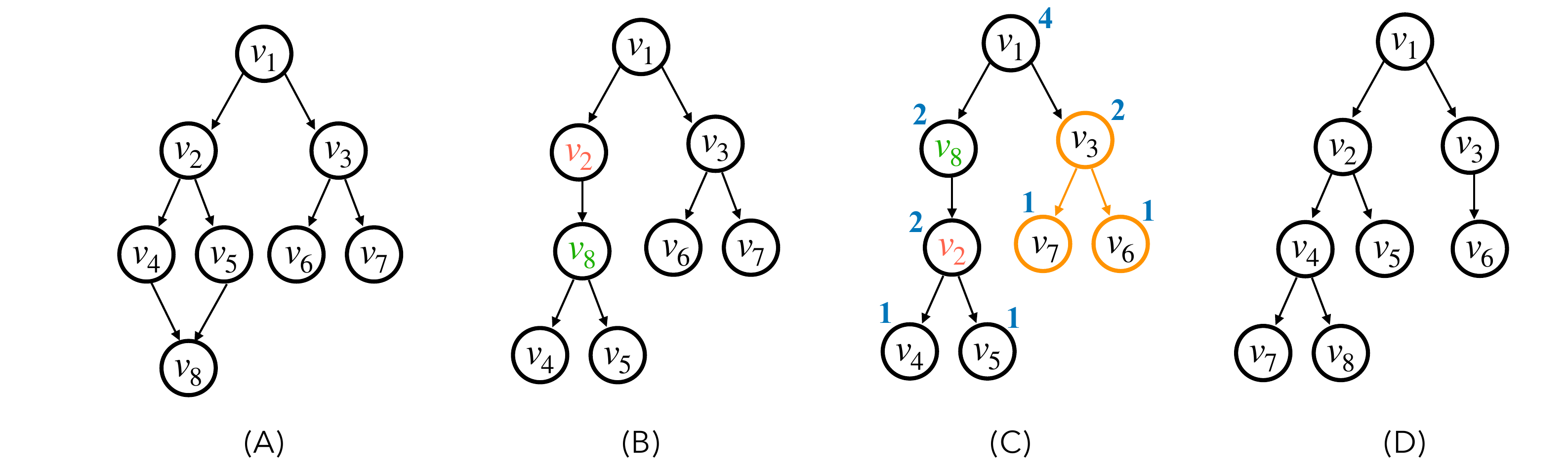}
    \caption{Directed Rooted Tree: (A) is not a DRT since the underlying undirected graph has a cycle, although the graph has no directed cycle. (B), (C), (D) all are DRTs of size $(I=4, \bJ=\{2,2,3,3\}, K=8)$. Notice how (B), (C) have the same intuitive structure (nodes $v_2$ and $v_8$ are swapped) while (D) has a totally different structure - the first two are isomorphic, while (B) and (D) are not. (C) also shows membership value of each node (blue number beside each node) - e.g. if we look at node $v_3$, the subtree starting from it is highlighted in yellow and has 2 leaves.}
    \label{fig:DRT-defn}
\end{figure}

We recall a few standard definitions from graph theory \citep{harary1969graph}. A \textit{graph} consists of a set of vertices and a set of edges connecting pairs of vertices (in this paper vertex and node are used interchangeably). A directed edge is an ordered pair of vertices $(u,v)$ which links from a \textit{parent} vertex $u$ toward a \textit{child} vertex $v$ while an undirected edge is a set of two vertices $\{u,v\}$ which connect $u$ and $v$ (there is no sense of direction). If the edges in a graph are directed, we call it a directed graph. A \textit{trail} is a sequence of vertices $v_1,\dots,v_{\ell}$ in an undirected graph such that consecutive pairs are connected and no edge is repeated. A cycle is a trail with same starting and ending vertices. A tree is an undirected graph with no cycle. In a directed graph, a vertex is called a \textit{root} if it has no parent and a vertex is called a \textit{leaf} if it has no children. The distance to a vertex from a root is defined as the length of the shortest trail from the root to that vertex. Finally, the undirected underlying graph of a directed graph $(V,E)$ is an undirected graph with the same set of vertices $V$ and set of edges $\{\text{set}(e): e\in E\}$.
 
\subsubsection*{Membership function} We define a function \textit{membership}, denoted by $\mathfrak{m}$, on the vertices of a DRT, $\mathfrak{m}:\cV\to \bbN$. Suppose $u\in\cV$ is a node in a DRT $\cT=(\mathcal{V},\mathcal{E},v_0)$. Then, the subtree starting at $u$ is a DRT. Suppose it is of size $(I',\bJ',K')$, then we define membership of $u$ in the original DRT as $I'$. This function (illustrated in blue in (C) in Figure \ref{fig:DRT-defn}) plays an important role in the sequel. Some trivial properties of $\mathfrak{m}$ include that $\mathfrak{m}(v_0)=I$ and $\mathfrak{m}(v)=1$ for all leaf nodes $v\in\cV$. Furthermore, due to the DRT structure $\mathfrak{m}(v)=\sum_{u: (v,u)\in \mathcal{E}} \mathfrak{m}(u)$, which basically says that the membership of a node is the sum of memberships of its children. This implies that if $u$ is a descendant of $v$ (i.e. there is a directed trail from $v$ to $u$), then $\mathfrak{m}(v)\geq \mathfrak{m}(u)$. We refer to this as the monotonicity of the membership function across trails. Suppose $(v,w_1,\dots,w_k,u)$ is a directed trail in the tree, then $\mathfrak{m}(u)=\mathfrak{m}(v)$ if and only if $v$ and each of $w_1,\dots,w_k$ has no other child apart from the one in this trail. As we shall see, the trails, as a set, in such a DRT starting from the root to a leaf gives rise to the topic hierarchy. In this context, some directed trees, while not DRT themselves (for example (A) in Figure \ref{fig:DRT-defn}), can be expressed as a DRT (e.g. (B) or (C) in Figure \ref{fig:DRT-defn}) which preserves these trails. In Figure \ref{fig:DRT-defn}, any of (A), (B) and (C) has exactly 4 trails (as sets) $\{v_1,v_2,v_4,v_8\}, \{v_1,v_2,v_5,v_8\}, \{v_1,v_3,v_6\}$ and $\{v_1,v_3,v_7\}$.

\subsection{Proof of Lemma \ref{lemma:tree}}
\begin{proof}
 Given a collection of subsets, each of which arise as the set of nodes along a maximal path of a DRT, we explicitly construct the DRT and show that such a construction is unique for the tree structure. Let $\{C_1,\dots,C_I\}$ be these subsets arising from some DRT $\cT$, which immediately imply that $\cT$ must have $I$ leaves. The algorithm would proceed by constructing the root and adding a child to one of the existing nodes at each step. It is worth reminding that if $(u,w_1,\dots,w_k,v)$ is a directed trail in $\cT$ with $\mathfrak{m}(u)=\mathfrak{m}(v)$, then each of $u,w_1,\dots,w_k$ has no other child apart from the one in this trail — this also indicates that any other tree where the nodes $\{u,w_1,\dots,w_k,v\}$ are permuted is still isomorphic to $\cT$ while preserving the collection of subsets along maximal paths. 

 \begin{figure}
    \centering
    \includegraphics[clip, trim=5cm 12cm 4cm 7cm, width=0.98\textwidth]{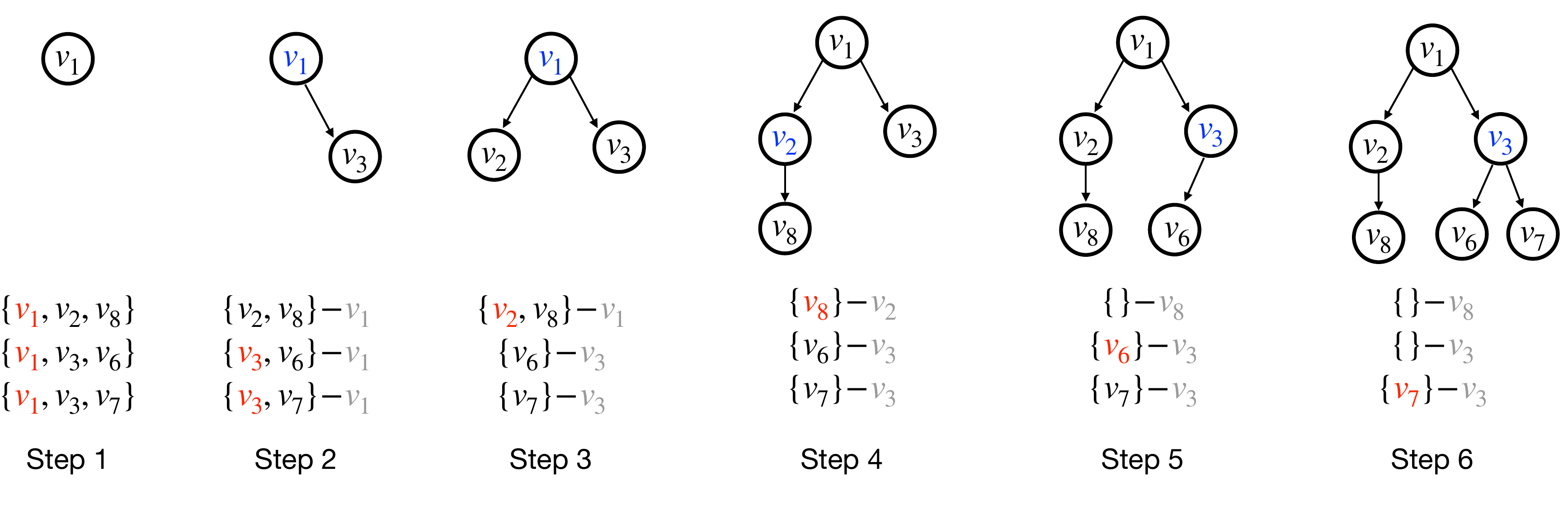}
    \caption{Illustration of the construction in the proof of Lemma \ref{lemma:tree}. Given a collection of subsets, construct the underlying DRT: Consider Step 2, $v_1$ (in blue) is the desired parent, we look at histories of each current subset (in gray outside each subset) and select the element which occurs in a maximal number of times, here $v_3$ (in red). Hence $v_3$ is added as a child of $v_1$ in this step. Now, $v_3$ is removed from all subsets which contained it and the histories are updated for the next step. The only ambiguous choices in this procedure occur in Steps 3 and 5: for Step 3, any of $v_2$ or $v_8$ could have been chosen, this would have only resulted in the nodes $v_2$ and $v_8$ being swapped, keeping the tree structure same. For Step 5, whether $v_6$ is chosen first or $v_7$ does not matter, both would have ended up as children of $v_3$ and the trees are same (not just isomorphic).}
    \label{fig:unique_tree_from_subsets}
\end{figure}

 Firstly, $\cap_i C_i\neq \emptyset$, since this collection arise from a DRT which has a root. Note that there might be more than one element in this intersection (e.g. consider $\cT$ to be the subtree starting from $v_2$ in (B) in Figure \ref{fig:DRT-defn}, $C_1\cap C_2 = \{v_2,v_8\}$). Pick any one of them, say $v_0$ and start a tree with $v_0$ as the root. Each successive step in this algorithm proceeds as follows: once a new child for an existing node in the tree constructed so far is identified, it is attached at the parent node and this element is removed from all current subsets which shared this vertex and had the parent node as history. A history $h$ is maintained for each of the subsets which keep track of the last element removed from this set. Thus when the root is fixed, we set $h_i=v_0$ for all $i$.

 For each node $v$ in the tree so far, gather all $C_i$ whose history is $v$ and select $u\in \cap_{i:h_i=v} C_i$ which appears a maximal number of times in these subsets. Note that the number of times it appears must be the membership number for that particular node in any tree which preserves the subsets. If there is a unique $u$ satisfying the above, it must be a child of $v$. If not, firstly $u$ must be a descendant of $v$ (since there were subsets which had both $v$ and $u$) — suppose $(v,w_1,\dots,w_k,u)$ be such a trail, then $\mathfrak{m}(w_1) \geq \mathfrak{m}(u)$, which contradicts the unique maximality property of $u$. If there are more than one such $u$ (say the maximal count is $n$), pick any of them. We note that any other node which is shared by the current $C_i$ with history $v$ and count less than $n$ cannot be a child of $v$ (since membership must be monotone). Attach this as a child of $v$ and make the changes to the subsets (remove $u$ from each subset which had $v$ as history and contained $u$) and their history (change history of each of these subsets to $u$) accordingly. Repeat until all subsets are emptied (for an illustration, see Figure \ref{fig:unique_tree_from_subsets}). 

 \textbf{Step 1:} Suppose at a stage we select $u$ as the child of a $v$ already in the tree and decide to append $u$ to the tree for the next step. We first note that $u$ cannot be present in any other subset whose history was not $v$. If not, suppose $u$ is also present in some other subset with history $w$. From the tree so far, get the smallest subtree containing both $v$ and $w$ and let $r$ be the root of this subtree with at least two children $r_1, r_2$ (since there are two different trails from $r$ to $v$ and $w$ and this is the smallest possible such subtree). This means that either there was a subset with $r,r_1$ and no $r_2$ or there was a set with $r,r_2$ and no $r_1$ (otherwise based on the construction, $r_1,r_2$ would not be separate children of $r$), lets call this $\star$. Now we note that since $u$ appears in subsets with history $v$ (descendant of $r_1$ say) and $w$ (descendant of $r_2$), there were two subsets one containing $r,r_1,v,u$ and one containing $r,r_2,w,u$. Now we recall that the original collection comes from some DRT $\cT$ and hence in that DRT, there is either a unique trail from the root to $u$ containing $r$ or a unique trail from the root to $r$ containing $u$. In this tree if either $r_1$ or $r_2$ is in this trail, then $\star$ is violated since, say $r_1$ is in the trail, any subset containing $r,r_2$ would contain $r_1$. Thus, both $r_1, r_2$ must be descendants of $r$ — one cannot be descendant of the other since again $\star$ is violated. Thus, there must exist two distinct trails from the root to $r_1$ and to $r_2$ containing both $r,u$. This means $\mathfrak{m}(r_1)<\mathfrak{m}(u)$ and $\mathfrak{m}(r_2)<\mathfrak{m}(u)$, which contradicts the choice that $r_1$ and/or $r_2$ were chosen as children of $r$ in the construction when $u$ had a higher count. Thus, we establish that after a node $u$ is attached to a parent $v$ by this process and the node $u$ is removed from all subsets containing it with history $v$, then $u$ is not present in any other subset -- this implies that the process does not lead to any tree where the underlying undirected graph has a cycle.
 
  \textbf{Step 2: }We argue that particular decisions at stages where are multiple choices do not affect the final tree structure. Suppose in this process we come to a stage for the first time when there are more than one such maximal elements. Suppose at this stage, the current subsets are $\tilde{C}_1,\dots,\tilde{C}_I$ and the target parent is $u$ in the tree (which must appear in at least one of the $\tilde{C}_i$'s histories, otherwise this cannot be a parent). Suppose $w_1,\dots,w_L$ all have the maximal count. There must be a partition of these $w_\ell$ into $E_1,\dots,E_Q$ such that $w_i,w_j\in E_q$ if and only if $w_i$ and $w_j$ both appear together in a subset. To prove this, start with $w_1$ and start filling a set $E_1$ with $w_1$ and all other $w$'s which appear together with $w_1$ in some subset. If this covers all $w$'s, we terminate. Else, there is some other $w_2$ not in this set and do the same process. In this process, suppose we come across a $w\in E_1$ that is present with $w_2$ in some subset. We argue that is not possible. If so, it means there is a subset with $w_1,w$ and not $w_2$ (otherwise $w_2$ would be in $E_1$) and a subset with $w_2,w$ and not $w_1$ (again else, $w_2$ would be in $E_1$). Then we can append $w_1$ as a child of $w$ and then $w_2$ as a child of $w$ (still one of the maximal elements) and end up with $w$ being in two subsets with different histories ($w_1$ and $w_2$), which by Step 1, is not possible. Thus, we are sure to have such a partition.

  Now we note that for each such $E_q$ of size $L_q$, the tree must have a trail $u\rightarrow v_1\rightarrow\dots\rightarrow v_{L_q}$, where the $v_{\ell}$ is any permutation of $E_q$, such that each of $v_1,\dots,v_{L_Q-1}$ has only one child. If any of these has another child, then it contradicts the fact that they all had the same maximal count (recall membership strictly decreases among children iff there are more than 1). Thus, all possible trees that incorporate these nodes can only differ up to a permutation of the nodes within a trail of the above form and it is easy to see that a bijection can be defined between any two of such trees which preserve the connectivity. Thus, all such trees are isomorphic. 
    
 This process terminates in finite time since $\cup_i C_i$ is finite to start with. This construction not only ensures that each $C_i$ can be attained from the constructed tree's maximal paths but also $\mathfrak{m}$ is monotone across levels of a trail. Furthermore, every step of this process is fixed and necessary, except possibly stages where the choice of $u$ is not unique, where we showed that all possible different trees must be isomorphic. Thus, the tree structure is unique.

Now, for the last part of the lemma, we can use the same argument as above. Starting with the collection of subsets obtained from $\cT$, we repeat the construction as in the proof. In this case, there will be no ambiguity at any stage regarding the order in which to add children, since $\mathfrak{m}(u)<\mathfrak{m}(v)$ (strict inequality) whenever $u$ is a descendant of $v$ in $\cT$. Thus, if we find multiple nodes having same membership at a stage as potential children, they all have to be children of the parent node and hence, the underlying DRT can be reconstructed uniquely.

\end{proof}

\textbf{Remark.} We add a remark on the uniqueness of topic tree's levels. The level of a node (i.e., the depth of the node from the root) represents the position of the corresponding topic in the topic hierarchy. In practical applications, starting from the root, which represents common topic shared by all documents, the parent node has more than one children at each level (other than leaf). This incorporates the idea that each topic has several subtopics with each considered as a further specialization within that topic. Other than this interpretation, this assumption makes the level of a particular node unique. As an illustration, consider trees (B) and (C) in Figure \ref{fig:DRT-defn} and suppose we use a common $\rho, \pi$ for these two trees. Clearly, they lead to the same model. However the level of $\rho(v_2)$ and $\rho(v_8)$ are not uniquely identified. The last part of Lemma \ref{lemma:tree} gives a sufficient condition under which levels of each topic are unique.

We note that \textit{any} collection of such subsets may not arise from a DRT. The condition that we know that \textit{there exists a DRT} which gives rise to this collection is necessary. Lemma \ref{lemma:tree} implies that if we know that there is at least one DRT which generates this collection, then any other DRT which also generates it must be isomorphic. An example of a collection of subsets which cannot arise from a DRT would be $\{\{v_1,v_2,v_4\}, \{v_1,v_2,v_3\},\{v_1,v_3\}\}$. Any directed tree which generates these must contain an undirected cycle.

\section{Moments of the Model}\label{app:moments}
In this section we discuss the moments of $\bX\in[V]^n\sim p_{\doc|\omega}$. We note that $\bX$ is a conditionally iid draw from a Categorical distribution with parameter $\sum_{k=1}^K h_k \theta_k$, where the distribution of $h\in\Delta^{K-1}$ follows a mixture of sparse Dirichlet distributions, which we delve into in this section.

We say $h$ follows a mixture of sparse Dirichlet distributions (MSD) with respect to DRT $\cT$ and parameters $\pi\in \Delta^{I-1}$ and $\alpha\in\bbR^I_+$, if $h$ has a mixture distribution: with probability $\pi_i$, $h_{\varphi_i}\sim \text{Dir}_{J_i}(\alpha_i)$ and $h_{\varphi_i^c} = 0$, where $h_{\varphi_i}= (h_k : v_k \in \varphi_i)$. Thus, with probability $\pi_i$, only a part of $h$ (as dictated through the paths of $\cT$) is non-zero and follows a Dirichlet distribution, while the remaining part is degenerate with all coordinates 0. To illustrate, consider the left DRT in Figure \ref{fig:DRT_example}. Since $K=7$, we have $h\in \Delta^{6}$. Let $P_1$ be the distribution of $h$ such that $(h_1,h_2,h_4)\sim \text{Dir}_3(\alpha_1)$ and $h_3=h_5=h_6=h_7=0$ - this corresponds to the first path (whose leaf is $v_4$). As another example, $P_3$ (corresponding to path ending at $v_6$) is the distribution of $h$ such that $(h_1,h_3,h_6)\sim \text{Dir}_3(\alpha_3)$ and $h_2=h_4=h_5=h_7=0$. Then $h\sim \sum_i \pi_i P_i$. Thus, for our model, $\doc|h \sim \otimes^n \text{Cat}(\Theta^\top h)$ with $h\sim \text{MSD}(\pi,\alpha|\cT)$, where $\Theta$ is the $K\times V$ matrix with topics $\theta_k$ as the rows (note that it includes all the topics in the model). We look at the moments of the distribution of $h$ first.

For the first moment, $\bbE[h_k] = \sum_i \pi_i \bbE_{P_i}[h_k] = \sum_i \pi_i \frac{\alpha_i}{J_i\alpha_i}1(v_k\in\varphi_i) = \sum_{i\in[I]: v_k\in\varphi_i} \pi_i/J_i$. Note that this implies
$$v_{k_1},\dots,v_{k_L} \text{ are children of } v_k \Rightarrow \bbE[h_k] = \sum_{\ell\in[L]} \bbE[h_{k_\ell}].$$
In particular, for the leaves of the DRT, if a leaf $v_k$ is in path $\varphi_i$, then $\bbE[h_k] = \pi_i/J_i$ and starting at the roots, $\bbE[h_k]$ for other nodes can be computed as the sum of the values for its children, as given in the above display. We remark that this property arises because of the symmetric Dirichlet distribution being used ($\alpha$ depends only on the path, not the depth). Similarly, for the second moment, we have
\begin{align*}
    \bbE[h_{k_1}h_{k_2}] = \sum_i \pi_i \bbE_{P_i}[h_{k_1}h_{k_2}] = \begin{cases}
        \sum_{i:v_{k_1}\in\varphi_i} \pi_i \frac{\alpha_i(\alpha_i+1)}{J_i\alpha_i(J_i\alpha_i+1)} &\text{if } k_1=k_2 \\
        \sum_{i:v_{k_1}, v_{k_2}\in\varphi_i} \pi_i \frac{\alpha_i^2}{J_i\alpha_i(J_i\alpha_i+1)} &\text{if } k_1\neq k_2
    \end{cases}
\end{align*}
using the second moments for the Dirichlet distribution (sum over an empty set is 0). Similarly, we can use the $n$-order moments of the Dirichlet distribution to compute the $n-$order moments for $h$. We illustrate this with the example in case of the left DRT in Figure \ref{fig:DRT_example}, with $\rho_1$ as in the figure. As before, $h\in\Delta^6$ in this case and suppose $\pi_1,\pi_2,\pi_3,\pi_4$ correspond to paths $\varphi_1,\dots,\varphi_4$, ending in leaves $v_4, v_5, v_6, v_7$ respectively. The first moments for $h$ and hence $X_1\sim p_{\bX_{[1]},\omega}$ are as follows (note that $J_i=J (=3)$ in this case):
\begin{align*}
    M^{(1)} := \bbE [h] = \left(
        \frac{1}{J},
        \frac{\pi_1+\pi_2}{J},
        \frac{\pi_3+\pi_4}{J},
        \frac{\pi_1}{J},
        \frac{\pi_2}{J},
        \frac{\pi_3}{J},
        \frac{\pi_4}{J}\right)^\top, \quad \bbE [X_1] = \Theta^\top M_{(1)}.
\end{align*}
The second moment for $h$ is the $K\times K$ matrix $M^{(2)}$ with $M^{(2)}_{k_1,k_2} = \bbE[h_{k_1},h_{k_2}]$. The form of $M^{(2)}$ and corresponding second moment of the observations for this case takes the form
\begin{align*}
    M^{(2)} = \bbE[h\otimes h] &= \begin{pmatrix}
        a_{1,2,3,4} & b_{1,2} & b_{3,4} & b_{1} & b_{2} & b_{3} & b_{4} \\
        b_{1,2} & a_{1,2} & 0 & b_{1} & b_{2} & 0 & 0 \\
        b_{3,4} & 0 & a_{3,4} & 0 & 0 & b_{3} & b_{4} \\
        b_{1} & b_{1} & 0 & a_{1} & 0 & 0 & 0 \\
        b_{2} & b_{2} & 0 & 0 & a_{2} & 0 & 0 \\
        b_{3} & 0 & b_{3} & 0 & 0 & a_{3} & 0 \\
        b_{4} & 0 & b_{4} & 0 & 0 & 0 & a_{4} \\
    \end{pmatrix},\\
    \bbE [X_1 \otimes X_2] &= \Theta^\top M^{(2)} \Theta.
\end{align*}
where
\begin{align*}
    a_{\cI} &= \sum_{i\in \cI} \frac{\pi_i (\alpha_i+1)}{J_i(J_i\alpha_i + 1)} \\
    b_{\cI} &= \sum_{i\in \cI} \frac{\pi_i \alpha_i}{J_i(J_i\alpha_i + 1)}.
\end{align*}
We note that $M^{(2)}_{k_1,k_2}$ is non-zero only if there is a path in $\cT$ such that $v_{k_1}$ and $v_{k_2}$ both are on it. Thus, the sparse-structure of $M^{(2)}$ inherits from the tree structure of $\cT$. In a similar manner, higher order moment tensors $M^{(n)}$ can also be expressed in terms of the parameters, allowing us to compute $\bbE[X_1\otimes\cdots\otimes X_n]$ for $\doc \sim p_{\doc|\omega}$. We remark that $M^{(n)}$ will also be sparse, with $M^{(n)}_{k_1,\dots,k_n}$ only non-zero if there exists a path in $\cT$ such that $v_{k_1},\dots,v_{k_n}$ are all on it (here $k_1,\dots,k_n\in[K]$). 

As opposed to the study of identifiability in \cite{anandkumar2012spectral}, the moment tensors $\bbE[X_1\otimes X_2\otimes X_3], \bbE[X_1\otimes X_2]$ and $\bbE[X_1]$ cannot be `diagonalized' (algebraically combined to form $\Theta^\top D \Theta$, with $D$ as a diagonal tensor) easily, even with the knowledge of $\alpha$. Applying techniques in tensor decomposition, including Kruskal decomposition, is also not straightforward, precisely because of this challenge in diagonalising the moments. We hope to resolve this issue in the future - also, we expect to recover the sparsity pattern in $\bbE[h\otimes h]$ from the data moments, which would help in estimating the tree structure as well. We postpone these to future work.

\section{Proofs in Section \ref{sec:identifiability}}

\subsection{Proof of Lemma \ref{lemma: dh+ triangle} - augmented tree-directed Hausdorff metric satisfies the triangle inequality}
\begin{proof}
    Let $\varphi^1_j,\dots,\varphi^I_j$ be enumerations of the maximal paths of the DRTs in $\omega_j$ for $j\in\{1,2,3\}$. Let $\sigma_{12}$ be the optimal permutation for $d_{\mathbb{H}+}(\omega_2,\omega_3)$ as in Equation \eqref{def: augmented tree directed Hausdorff metric} and similarly $\sigma_{23}$ for $d_{\mathbb{H}+}(\omega_2,\omega_3)$. Let $\sigma = \sigma_{23}\circ \sigma_{12}$, then clearly $\sigma\in \mathbb{S}(I)$. Then we have
    \allowdisplaybreaks
    \begin{align*}
        d_{\mathbb{H}+}(\omega_1, \omega_3) &\leq \sum_i \Big[d_{\mathcal{H}}(G^1_i, G^3_{\sigma(i)}) + \abs{\pi^1_i - \pi^3_{\sigma(i)}}\Big] \\
        &\leq \sum_i \left[d_{\mathcal{H}}(G^1_i, G^2_{\sigma_{12}(i)}) + d_{\mathcal{H}}(G^2_{\sigma_{12}(i)}, G^3_{\sigma(i)}) + |\pi_i^1 - \pi_{\sigma_{12}(i)}^2| + |\pi_{\sigma_{12}(i)}^2 - \pi_{\sigma_(i)}^3|\right] \\
        &= \sum_i \left[d_{\mathcal{H}}(G^1_i, G^2_{\sigma_{12}(i)}) + |\pi_i^1 - \pi_{\sigma_{12}(i)}^2|\right] + \sum_i \left[d_{\mathcal{H}}(G^2_{\sigma_{12}(i)}, G^3_{\sigma(i)}) + |\pi_{\sigma_{12}(i)}^2 - \pi_{\sigma_(i)}^3|\right] \\
        &= d_{\mathbb{H}+}(\omega_1, \omega_2) + \sum_i \left[d_{\mathcal{H}}(G^2_{\sigma_{12}(i)}, G^3_{\sigma_{23}\circ\sigma_{12}(i)}) + |\pi_{\sigma_{12}(i)}^2 - \pi_{\sigma_{23}\circ\sigma_{12}(i)}^3|\right] \\
        &= d_{\mathbb{H}+}(\omega_1, \omega_2) + \sum_{i'} \left[d_{\mathcal{H}}(G^2_{i'}, G^3_{\sigma_{23}(i')}) + |\pi_{i'}^2 - \pi_{\sigma_{23}(i')}^3|\right] \\
        &= d_{\mathbb{H}+}(\omega_1, \omega_2) + d_{\mathbb{H}+}(\omega_2, \omega_3).
    \end{align*}
    Here the first inequality is by definition and the second by the fact that both $d_{\mathcal{H}}$ and $|\cdot|$ are metrics by their own right. The following chain of equalities merely use the definitions of $\sigma_{12}$, $\sigma_{23}$ and $\sigma$.
\end{proof}

\subsection{Proof of Proposition \ref{lemma: tree-directed identifies tree structure}}
\begin{proof}
\textbf{First part:} For fixed $\omega_0$, consider 
\begin{align*}
    r_1(\omega_0) &:= \min \left\{\min_i \min_{\substack{\theta, \theta'\in\extr \cS_i\\\theta\neq\theta'}}\norm{\theta-\theta'}, \min_{i\neq j} \min_{\substack{\theta\in \extr \cS_i \\ \theta'\in\extr \cS_j \\ \theta\neq \theta'}}\norm{\theta-\theta'}\right\}\\
    r_2(\omega_0) &:= \min_{i\neq j \in[I]} d_{\cH}(\cS_i, \cS_j)
\end{align*}
which considers the minimum inter-topic distance between topics from the same component or different topics (not the shared ones) across different components in $r_1$ and the minimum inter-component Hausdorff distance in $r_2$, by our assumptions, $r_1>0, r_2>0$ (first one by injectivity of $\rho$ and the second one by assumption \ref{assume:distinct_components}).

Furthermore, by assumption \ref{assume:all_exposed}, there is $r_3(\omega_0)>0$ such that for all $i\in[I]$ and $\theta\in \extr \cS_i$, there is a hyperplane passing through $\theta$ that makes at least angle $r_3(\omega_0)$ with each of the edges of $\cS_i$ incident at that extreme point $\theta$. Fix $\epsilon_0(\omega_0)$ sufficiently small (to be decided later).

Condition 1 \label{proof:condition1}: $\epsilon_0$ should be such that for all $\omega<\omega_0$, if $d_{\bbH+}(\omega_0,\omega')<\epsilon$, then $r_1(\omega)>r_1(\omega_0)/2$,  $r_2(\omega)>r_2(\omega_0)/2$ and $\omega'$ also satisfies the non-obtuse (existence of hyperplane) property with $r_3(\omega) > r_3(\omega_0)/2$.

Now, we have $d_{\bbH+}(\omega,\omega')<\epsilon \leq \epsilon_0$, this shows that there is a permutation $\sigma \in \bbS(I)$ such that
$$\sum_i \left[d_{\cH}(\cS_i, \cS_{\sigma(i)}') + |\pi_i - \pi_{\sigma(i)}'|\right] < \epsilon$$
without loss of generality, assume $\sigma$ is the identity permutation (otherwise relabel the components). Hence, in particular, $d_{\cH}(\cS_i, \cS_i')<\epsilon$ for all $i$, which by Lemma 1(b) \cite{NguyenPosteriorContractionPopulation2015} yields that $d_{M}(\cS_i, \cS_i')<C_0\epsilon$ where $d_M$ is the minimal matching metric defined in \ref{eq: minimal matching distance} and $C_0$ is a constant depending on $r_3(\omega_0)$. This shows that for every $\theta\in \extr \cS_i$ (because of assumption \ref{assume:all_exposed}, each topic $\theta$ from $\omega$ is an extreme point of some component polytope), there is a $\theta'\in \extr \cS_i'$ such that $\norm{\theta-\theta'}<C_0\epsilon$ (and also vice-versa). Thus, if $\epsilon\leq \epsilon_0 < r_1(\omega_0)/3C_0,$ then $\norm{\theta-\theta'}<r_1(\omega_0)/3$ and hence the balls $\{B(\theta,\epsilon):\theta\in \extr S_i \text{ for some } i\in [I]\}$ are all disjoint. Thus, we have the partitions $\cP_1= \{\{u\} : u\in \cT\}$ for $\omega$ and $\cP_2=\{\{u'\in \cT: \norm{\rho(u)-\rho'(u')}<C_0\epsilon\} : u \in \cT\}$ as a valid partition satisfying that for any $P\in \cP_i$, there is a unique $P'\in \cP_2$ such that for any $u\in P$ and $u'\in P'$, $\norm{\rho(u)-\rho'(u')}<C_0\epsilon$. Let $\tau:\cP_1\to\cP_2$ denote this map $P\mapsto P'$ satisfying the above.

Now consider a component polytope $\cS_1$ of $\omega_0$, with extreme points $\theta_1=\rho(v_1),\dots,\theta_k=\rho(v_k)$ with $v_1,\dots,v_k$ on a maximal path of $\cT$ such that $v_i\in P_i\in \cP_1$ for all $i$ (note that we allow $P_i=P_{i'}$ for $i\neq i'$). Consider $\cS_1'$, a component polytope of $\omega'$ with $d_{\cH}(\cS_1,\cS_1')<\epsilon$. Suppose $\theta_1'=\rho'(v_1'), \dots, \theta_{k'}=\rho'(v_{k'}')$ be its extreme points. We show that $v_{j}'\in \cup_{i\in[k]} P_{\tau(i)}'$ for all $j\in[k']$, i.e. the partitions corresponding to $P_1,\dots,P_k$ (under $\tau$) contain all the $v_1',\dots,v_{k'}'$. If not, suppose $v_1'\in P_1'$ such that $P_1'\neq P_{\tau(i)}'$ for any $i\in[k]$. 

We argue that $\rho'(v_1')$ must be separated from each extreme point of $\cS_1$. Firstly, for any $v\in P_{\tau^{-1}(1)}$, we have $\norm{\rho(v) - \rho'(v_1')}<C_0\epsilon$ by construction of the partitions. Now, for any such $v$, since it is not an extreme point of $\cS_1$, $\norm{\rho(v) - \rho(v_i)}\geq r_1$ for all $i\in[k]$. Hence, $\norm{\rho'(v_1') - \rho(v_i)} > r_1 - C_0\epsilon > 2r_1/3$ (by choice $\epsilon_0<r_1/3C_0$ as earlier) for all $i\in[k]$. This implies that $d_{\cH}(\cS_1, \cS_1') \geq 2r_1/3$, which contradicts $d_{\cH}(\cS_1,\cS_1')<\epsilon$ for $\epsilon < 2r_1/3$. Thus, overall $\epsilon_0$ satisfying Condition 1 above and $\epsilon_0<\min\{r_1(\omega_0)/3C_0(\omega_0), 2r_1(\omega_0)/3\}$ works, all depending on $\omega_0$.

Now, if the number of nodes of $\cT$ and $\cT'$ are same, say $K$, each set in the partition $\cP_2$ must be a singleton containing only one node (note that $|\cP_1|=K$). By the condition proved in the last part, this means that $v_1,\dots,v_k$, with $v_i\in P_i\in \cP_1$, are on a maximal path (equivalently, $\rho(v_1),\dots,\rho(v_k)$ are extreme points of a component polytope $\cS_j$) iff $v_1',\dots, v_k'$ are on a maximal path of $\cT'$, with $v_i'\in P_{\tau(i)}'\in\cP_2$. Since both $\cV=\cV'=[K]$, hence this means that the collection of sets of nodes along maximal paths of $\cT$ and $\cT'$ are the same - by application of lemma \ref{lemma:tree}, we conclude that $\cT\cong \cT'$. Furthermore, $\tau$ identifies a unique isomorphism from $\cT$ to $\cT'$ and $\norm{\rho(v)-\rho'(\tau(v))}<\epsilon$,  which completes the proof. Of course, we can also conclude that corresponding $\pi_i$'s are also $\epsilon-$close.

\textbf{In particular part:} The `if' part of the proof is trivial. We show the `only if' part. Assume $d_{\mathbb{H}+} = 0$. Enumerate the maximal paths of $\cT$ and $\cT'$ to get $\varphi^1,\dots,\varphi^I$ and $\varphi'^1,\dots,\varphi'^I$ respectively. By definition, we have a permutation $\tau$ of $[I]$, such that $\conv(\rho(\varphi^i)) = \conv(\rho'(\varphi'^{\tau(i)}))$  and $\pi(\varphi^i) = \pi'(\varphi'^{\tau(i)})$ for all $i$. By assumption \ref{assume:all_exposed}, this means that for every $i$, the set $\rho(\varphi^i)$ is the same as the set $\rho'(\varphi'^{(\tau(i)})$. This immediately gives that $K=K'$ ($K$ is the number of vertices in $\cT$).

We wish to apply Lemma \ref{lemma:tree}, but that applies only to DRT, not the topic maps. To resolve it, for the tree $\cT=(\mathcal{V},\mathcal{E},v)$ we consider a new DRT $\tilde{\cT}=(\tilde{\mathcal{V}},\tilde{\mathcal{E}},\tilde{v})$, where $\{\rho(u):u\in \mathcal{V}\}$ is the set of vertices and $(\rho(u),\rho(u'))\in\tilde{\mathcal{E}}$ iff $(u,u')\in\mathcal{E}$ and $\tilde{v}=\rho(v)$. Having constructed $\tilde{\cT}$ from $\cT$ and $\rho$ and $\tilde{\cT}'$ from $\cT'$ and $\rho'$, we can now apply Lemma \ref{lemma:tree} to guarantee that $\tilde{\cT}\cong \tilde{\cT}'$. However, by construction, $\cT\cong \tilde{\cT}$ and $\cT'\cong \tilde{\cT}'$, which together gives $\cT\cong \cT'$. 

We are only left to find the unique isomorphism $\sigma$ as in the statement of the lemma. We claim that it suffices to show that one exists. If there were two such isomorphisms $\sigma$ and $\sigma'$, then $\rho(u)=\rho'(\sigma(u))=\rho'(\sigma'(u))$ for every vertex $u$ of $\cT$, this implies that $\sigma(u) = \sigma'(u)$ since $\rho'$ is one-to-one (definition of topic map). Thus, $\sigma=\sigma'$.

Recall the existence of a permutation $\tau$ of $[I]$. Consider a matched pair of maximal paths, $\varphi^1$ of $\cT$ and $\varphi'^{\tau(1)}$ of $\cT'$. Let $\varphi^1=(\varphi^1_1,\dots,\varphi^1_{J_1})$, then since $\conv(\rho(\varphi^1)) = \conv(\rho'(\varphi'^{\tau(1)}))$, by assumption \ref{assume:all_exposed}, $\rho(\varphi^1)$ must be a permutation of $\rho'(\varphi^{\tau(1)})$. Define the unique map $\sigma$ which maps $u\in \varphi^1$ to $u'\in \varphi^{\tau(1)}$ such that $\rho(u)=\rho'(\sigma(u))$. Now, we keep repeating and expanding the domain of $\sigma$. At step $i=2,\dots,I$, take another pair of maximal paths $\varphi^i$ and $\varphi'^{\tau(i)}$. For all $u\in \varphi^i\setminus \cap_{1\leq i <I}\varphi^i$, define $\sigma(u)$ as previously based on the fact that $\rho(\varphi^i)$ must be a permutation of $\rho'(\varphi'^{\tau(i)})$. It is easy to check that the constructed $\sigma$ satisfies all conditions in the result.
\end{proof}

\subsection{Proof of Lemma \ref{lemma: intersection affine space}}
In the lemma, the ambient space is irrelevant. Thus, take $\cA,\cA',\cS$ as subsets of $\bbR^d$, say.
\begin{proof}
    \textit{Part 1:} An affine space $\cA$ can be represented as $\{x\in\bbR^d : Ax=b\}$, i.e., as the solution of a system of linear equations where $A\in \bbR^{m\times d}, b\in\bbR^m$. A convex polytope $\cS$ is the intersection of a set of hyperplanes and can be represented as $\{x\in\bbR^d: Cx\leq d\}$, where the inequality is coordinate-wise, $C\in\bbR^{n\times d}, d\in\bbR^n$. By appending $A$ and $-A$ as rows in $C$ (to form $\tilde{C}\in\bbR^{(n+2m)\times d}$ and appending $b$ and $-b$ to $d$ (to form $\tilde{d}\in\bbR^{n+2m}$, we get the intersection of $\cA$ and $\cS$ as the solution set of $\{x\in\bbR^d : \tilde{C}x\leq \tilde{d}\}$, which by definition is another convex polytope.

    \textit{Part 2:} Similar to part 1.

    \textit{Part 3:} Let $\cA=S + v$, where $S$ is a linear subspace of $\bbR^d$ and $v\in \cA$, similarly let $\cA'=S'+v'$. Then, by assumption, $\dim S=\dim S'$. If $\cA\cap \cA'\neq \emptyset$, let $x\in \cA\cap\cA'$. Then, $\cA=S+ x$ and $\cA'=S'+x$. We claim that $\cA\cap \cA' = (S\cap S') + x$. Any $y\in (S\cap S') + x$ is of the form $y=v + x$ where $v\in S\cap S'$ and hence, $y\in \cA$ and $y\in \cA'$ (as $\cA=S + x$). Conversely, suppose $y\in \cA\cap \cA'$, then $y-x\in S$  and $y-x\in S'$, hence $y-x\in S\cap S'$, thus $y = (y-x) + x \in (S\cap S') + x$. Thus, $\dim \cA\cap\cA' = \dim S\cap S'$. Now, either $\dim S\cap S' = \dim S$ (iff $S=S'$) or $\dim S\cap S' < \dim S$ (otherwise), which proves the statement of the lemma.

    \textit{Part 4:} Similar to part 3. Note that for subspaces $S, S'$ with $\dim S < \dim S'$, either $\dim S\cap S' = \dim S$ (iff $S$ is a subspace of $S'$) or $\dim S\cap S' < \dim S$.
\end{proof}

\subsection{Proof of Theorem \ref{thm:identifiability1}}\label{app:proof of identifiability1}

Before we look at the proof, we make the following observations regarding mixture models of the form $\sum_k \pi_k G_k$, where the mixture components $G_k$ are supported on simplices $\cS_k$ with affine hulls $\cA_k$. For a measure $G$ on $\Delta^{V-1}$ and a measureable $H \subset \Delta^{V-1}$ such that $G(H)>0$, we define the restricted and normalized measure $G|_{H}$ by $G|_{H}(A) = G(A\cap H) / G(H)$. Using Lemma \ref{lemma: intersection affine space}, we have
\setlist{nolistsep}
\begin{enumerate}
    \item For two component measures $G_i, G_j$ supported on $\cS_i, \cS_j$ respectively with dimensions $d_i < d_j$, $\left(\pi G_i + (1-\pi) G_j\right)|_{\cA_i} = G_i$. This occurs since $G_j$ is absolutely continuous with respect to the Hausdorff measure on $\cA_j$ and hence any measurable subset of $\mathcal{A}_i$ is a null set with respect to this measure, since $d_i<d_j$. Note that the other case is not true, i.e. $\left(\pi G_i + (1-\pi) G_j\right)|_{\mathcal{A}_j}$ is, in general, not $G_j$.
    \item For two components $G_i$ and $G_j$ with $d_i = d_j$, by Assumption \ref{assume:diff_affine_space}, $\mathcal{A}_i \cap \mathcal{A}_j$ is a \textit{strictly} lower dimensional flat (by Lemma \ref{lemma: intersection affine space}). Thus, $\left(\pi G_i + (1-\pi) G_j\right)|_{\cA_i} = G_i$ and $\left(\pi G_i + (1-\pi) G_j\right)|_{\cA_j} = G_j$. To prove this, say the first one, we only note that for any measurable $A\subset \cA_i$, $G_j(A)=0$ since either $A\cap \cS_j=\emptyset$ or if they have non-empty intersection, $A\cap \cS_j$ is a null set with respect to $G_j$ (since it is a subset of the strictly lower dimensional flat $\cA_i\cap\cA_j$).
    \item For $G=\sum_k \pi_k G_k$ and affine space $\cA$ such that $G(\cA)>0$, $G|_{\cA}(B) = \sum_k \pi_k G_k(B\cap \cA)/G(\cA)$. This follows from the definition of restricted and normalized measure. Let us denote $G_k(B\cap \cA)/G(\cA)=\tilde{G}_k(B)$ and think of $\pi_k \tilde{G}_k$ as the contribution from component $k$ towards $G|_{\cA}$.  For an affine space $\cA$ and a convex polytope $\cS$, $\cA \cap \cS$ is either $\emptyset$ or $\cS$ or a strictly lower dimensional $(< \dim \cS$) convex polytope (by lemma \ref{lemma: intersection affine space}). For $k$ such that $\cA\cap\cS_k=\emptyset$, clearly, $\tilde{G}_k$ is the 0 measure. For $k$ such that $\dim (\cA\cap\cS_k) < \dim \cS_k$, also $\tilde{G}_k$ is the 0 measure, since for any $B$, $B\cap \cA\cap \cS_k$ is a null set with respect to $G_k$ (recall $G_k$ is absolutely continuous with respect to the Hausdorff measure on $\cS_k$). 
\end{enumerate}
Combining these, we obtain that for a mixture of the form $G = \sum_i \pi_i G_i$, for any affine space $\cA$ of dimension $d$, the only measures among the components which may contribute to $G|_{\mathcal{A}}$ are those for which $\dim G_i \leq d$. Moreover, for a component $G_i$ with $d_i=d$ which contributes to this, it must have $\cA_i=\cA$.

\begin{proof}
    Firstly, let $\doc=(x_1,\dots,x_n)$ with $x_j\in [V]$ and consider the sequence of random variables $\bar{\bX}_n\in\Delta^{V-1}$ where $(\bar{\bX}_n)_v= \sum_{j\in[n]} 1(x_j=v)/n$ being the sample mean. Let $\mu_n$ be the law of $\bar{\bX}_n$, then $\mu_n$ is a probability measure on $\Delta^{V-1}$. Recall from the model (e.g. see Section \ref{sec:model-geom} and Equation \ref{eq: document density2}), that conditional on $\eta \sim G$, $x_j$'s are i.i.d. from the measure $\text{Cat}(\cdot|\eta)$. Thus, conditionally given $\eta$, by Strong Law of Large Numbers, $\bar{\bX}_{n}|\eta \to \eta$ a.s. as $n\to\infty$ which yields $\mu_n \Rightarrow G$. Now, since $V(P_{n,\omega,\alpha}, P_{n,\omega',\alpha'})\to 0$ as $n\to \infty$, then it shows that $G(\omega, \alpha) = G(\omega', \alpha')$. We now show that this implies that the tree topology and the topic map are both uniquely identified.

    Arguing rigorously, we have
    \begin{enumerate}
        \item $\mu_n\Rightarrow G$: $\bar{X}_n$ can be seen as $\bar{X}_n=\sum_j Y_j/n$ where $Y_1,\dots,Y_n|\eta\sim \text{Mult}(1,\eta)$, i.e. $Y_j=e_v$ with probability $\eta_v$ ($\eta\in\Delta^{V-1}$), where $e_k=(0,\dots,0,1,0,\dots,0)^\top$, with $1$ in the $k$th coordinate. We prove that the characteristic function of $\bar{X}_n$ converges to that of $G$ pointwise, which will ensure weak convergence. To start,
        $$\varphi_{Y|\eta}(t) = \bbE_{Y\sim \text{Mult}(1,\eta)} [e^{it^\top Y}] = \sum_v \eta_v e^{it_v}.$$
        Now,
        \begin{align*}
            \varphi_{\bar{\bX}_n}(t) = \bbE \bbE\left[e^{it^\top \sum_j Y_j/n} \mid \eta\right] &= \bbE \left[\varphi_{Y|\eta}(t/n)^n\right] \\
            &=\int \left(\sum_v\eta_v e^{it_v}\right)^n G(d\eta) \\
            &= \int \left(\sum_v \eta_v\left[1 + \frac{i t_v}{n} + O\left(\frac{1}{n^2}\right)\right]\right)^n G(d\eta) \\
            &= \int \left(1 + i\frac{t^\top \eta}{n} + O\left(\frac{1}{n^2}\right)\right)^n G(d\eta) \\
            &\to \int e^{it^\top \eta} G(d\eta), \quad n\to \infty
        \end{align*}
        where the last limit is by dominated convergence theorem. Thus, noting that the last expression is indeed the characteristic function of $\eta\sim G$, we have the desired result. Furthermore, $\bbE[\bar{\bX}_n] = \bbE\bbE[\bar{\bX}_n|\eta] = \eta$, which shows $W_1(\mu_n, G) \to 0$.
        \item Let $\mu_{n|\omega}$ and $\mu_{n|\omega'}$ be the law of $\bar{X}_n$ where $X_1,\dots,X_n|\eta \sim \text{Mult}(1,\eta), \eta\sim G(\omega,\alpha)$ and $G(\omega',\alpha')$ respectively. $\mu_{n|\omega}, \mu_{n|\omega'}$ are measures supported on $\Delta^{V-1}$, a bounded subset of $\bbR^V$. By Theorem 6.15 from \cite{villani2009optimal}, we know that for every fixed $n$, $W_1(\mu_{n|\omega}, \mu_{n|\omega'}) \leq C V(\mu_{n|\omega}, \mu_{n|\omega'})$, where $C=\text{diam} \Delta^{V-1}$ and furthermore, by data processing inequality, we have $V(\mu_{n|\omega}, \mu_{n|\omega'}) \leq V(P_{n,\omega}, P_{n,\omega'})$ (since $\bar{X}_n$ is a statistic of $\bX=(X_1,\dots,X_n)$). Thus, 
        \begin{align*}
            W_1(P,Q) &\leq W_1(P,P_n) + W_1(P_n, Q_n) + W_1(Q_n,Q) \\
            &\leq W_1(P,P_n) + C V(P_{n,\omega,\alpha}, P_{n,\omega',\alpha'}) + W_1(Q_n,Q) \to 0
        \end{align*} showing that $W_1(P,Q)=0$, and hence $P=Q$.
    \end{enumerate}
    
    From our model, $G(\omega,\alpha) = \sum_{i\in[I]} \pi_i G_i$ where $G_i=\text{Dir}_{J_i}(\alpha_i)_{\#} L_i$ with $L_i(\beta)=\Theta_i^\top\beta$ is the component measure corresponding to path $\varphi_i$ of $\cT$ (fix an enumeration) and $\cS_i = \conv \Theta_i$ is the support of $G_i$. Similarly we have $G(\omega',\alpha') = \sum_{i\in[I']} \pi_i' G_i'$, with $G_i'$ supported on $\cS_i'$ (fix enumeration for paths of $\cT'$). Recall each of the constituent measures $G_i$ (resp. $G_i'$) is supported on the polytope $\cS_i$ (resp. $\cS_i'$) and is absolutely continuous with respect to Hausdorff measure on $\cA_i = \aff \cS_i$ (resp. $\cA_i'$). For brevity of notations, we denote $G(\omega,\alpha), G(\omega',\alpha')$ as $G, G'$ respectively. Without loss of generality, we permute the components (relabel) of $G$ so that $d_1\leq d_2 \leq \dots \leq d_I$, where $d_i = \dim \cS_i$, and similarly for components in $G'$.

Now we start the main argument for the proof. We have $G=G'$ and we want to identify the components in the first step. Firstly, it is easy to see that $d_1=d_1'$. If not, suppose $d_1<d_1'$, then taking $B=\cS_1$, clearly $G(B)>0$ while $G'(B)=0$. Thus, we have $d_1=d_1'$. Now we consider $G|_{\cA_1}$. By our arguments in the digression and assumption \ref{assume:diff_affine_space}, we know that $G|_{\cA_1} = G_1$. Thus, $G'|_{\cA_1}$ must also be $G_1$. However, we know the only components contributing to $G'|_{\cA_1}$ are $G_i'$ with $d_i'\leq d_1$. Since $d_1$ is the minimum, hence any such $d_i'=d_1$, and thus any such $i$ must satisfy $\cA_i'=\cA_i$. But by assumption \ref{assume:diff_affine_space} on $\rho'$, we know that no two of such measures (components of $\omega'$) have the same affine hull. This implies that there is a unique $G_i'$ with $d_i'=d_1$ such that $G_i'=G_1$, which would also imply $d_{\cH}(\cS_1, \cS_i') = 0$. Furthermore, since $\pi_1 = P(G_1) = P'(G_i') =\pi_i'$, we conclude that the corresponding mixture probabilities are same too. Thus, we identify a unique component in $G'$, such that it exactly matches $G_1$ (for symmetric Dirichlet as in our case, $G_1=G_i'$ also implies $\alpha_1 = \alpha_i'$ - see ). 

Having identified this component, we remove it, i.e. we consider $\tilde{G} = \sum_{i\neq 1} \frac{\pi_i}{1-\pi_1} G_i$ and $\tilde{G}' = \sum_{j\neq i} \frac{\pi_j'}{1-\pi_i'} G_j'$, since we had $G=G'$ and showed $G_1=G_i', \pi_1=\pi_i'$, we now have $\tilde{G}=\tilde{G}'$. Now, we can apply the same argument iteratively. This process terminates in finite time since at each step we are removing a component from the finite mixture. 
 Note that Assumption \ref{assume:pos} ensures that each step in this process is valid and there are no extraneous components left at the end. Furthermore, in this process, we have ensured that $G'$ also has exactly $I$ components (otherwise the process terminates with one of $\tilde{G}$ or $\tilde{G}'$ having at least one non-trivial component left) and each component of $G$ matches a unique component of $G'$. This implies that $d_{\bbH+}(\omega,\omega')=0$ by definition. 
\end{proof}

\subsection{Proof of Theorem \ref{thm:identifiability2}}
Following the same strategy as in the proof of Theorem \ref{thm:identifiability1}, we arrive at $G(\omega,\alpha) = G(\omega',\alpha')$. Using the same definitions as before, permute the labels of the components of $\omega$ (without loss of generality) such that $d_1\leq \dots \leq d_I$

Firstly, arguing as in the preceding theorem, we  have $\min \{d_i'\} = d_1$. Considering $G|_{\cA_1}$, we get that $G'_{\cA_1} = G_1$ (note that $\omega$ satisfies assumption \ref{assume:diff_affine_space}). However, at this stage we cannot argue that this identifies a unique $G_i'$ as before since $\omega'$ does not necessarily satisfy the assumption (in particular, there might be two components sharing the same affine hull). Here we utilize the knowledge of $I$ to recover the components. Suppose, if possible, $G_1' G_2'$ be two component measures of $G'$ such that $G_1 = \pi_1'G_1' + \pi_2'G_2'$. Then, by argument as in the preceding theorem, it must hold that $\cA_1'=\cA_2'=\cA_1$ - for any other $G_i$ (from $\omega$), either $d_i > d_1$ or $G_i$ is supported on a different affine space (assumption \ref{assume:diff_affine_space}). In either case, there are $I-1$ distinct affine spaces for the $\{G_i : i\neq 1\}$, while there are only $I-2$ possible $\{G_i':i\neq 1, i\neq 2\}$. This cannot be possible since each $G_i'$ accounts for exactly one affine space. Thus, there cannot be two $G_1', G_2'$ accounting for $G_1$. Thus, we conclude that there is a unique $G_i'$ with $\cA_i'=\cA_1$ and $\pi_i' G_i'= \pi_1 G_1$. This gives $d_{\bH}(G_1,G_i')=0$ and $\pi_i'=\pi_1$. The rest of the proof is identical to that of Theorem \ref{thm:identifiability1}.

\section{Proofs in Section \ref{sec:rate}}

\subsection{Proofs in Section \ref{sec:density estimation}}

\subsubsection{Proof of Lemma \ref{lemma: kl_upper_W}}
\begin{proof}
Associate each sample $\bX:=\doc=(X_1,\dots,X_n)$ with a $V$ dimensional vector $\eta(\bX)\in \Delta^{V-1}$, where $\eta(\bX)_v = \frac{1}{n}\sum_{j\in[n]} 1(X_j=v)$ for $v=1,\dots,V$. The density of $\doc$ given $\omega$ takes the form
    $$p_{n,\omega}(\doc) = \sum_{i\in[I]} \pi_i\int_{\cS_i} p(\doc|\eta) G_i(d\eta)=\sum_{i\in[I]} \pi_i \int_{\cS_i} \exp\left(n\sum_{v\in[V]} \eta(\bX)_v \log \eta_v\right) G_i(d\eta).$$
    Given a fixed permutation $\sigma$ of $[I]$, we denote $\cS_i'$ to indicate $\cS_{\sigma(i)}'$ and $\pi_i$ to indicate $\pi_{\sigma(i)}'$ for notational simplicity. For $\bX\in[V]^n$, let 
    \begin{align*}
        a_i(\bX)&=\pi_i\int_{\cS_i} p(\bX|\eta)G_i(d\eta)\quad\text{and}\quad        b_i(\bX)=\pi_{\sigma(i)}'\int_{\cS_i'} p(\bX|\eta')G_i'(d\eta').
    \end{align*} 
    so that 
    $$p_{n,\omega}(\bX) = \sum_i a_i(\bX), \quad p_{n,\omega'}(\bX) = \sum_i b_i(\bX).$$
    By Jensen's inequality, for any $\bX\in [V]^n$,
    $$\sum_i a_i(\bX) \log \frac{a_i(\bX)}{b_i(\bX)} \geq \left(\sum_i a_i(\bX)\right) \log \frac{\sum_i a_i(\bX)}{\sum_i b_i(\bX)}.$$
    Adding over all $\bX\in[V]^n$, we get
    \begin{align*}
        &\KL(p_{n,\omega}\Vert p_{n,\omega'}) \\
        &= \sum_{\bX\in [V]^n} p_{n,\omega}(\bX)\log \frac{p_{n,\omega}(\bX)}{p_{n,\omega'}(\bX)} \\
        &= \sum_{\bX} \left(\sum_i a_i(\bX)\right)\log \frac{\sum_i a_i(\bX)}{\sum_i b_i(\bX)} \\
        &\leq \sum_{\bX} \sum_i a_i(\bX)\log \frac{a_i(\bX)}{b_i(\bX)} \\
        &= \sum_i \sum_{\bX} \pi_i \int_{\cS_i} p(\bX|\eta)G_i(d\eta) \left[\log \frac{\pi_i}{\pi_i'} + \log \frac{\int_{\cS_i} p(\bX|\eta)G_i(d\eta)}{\int_{\cS_i'} p(\bX|\eta')G_i'(d\eta')}\right]
    \end{align*}
    which gives
    \begin{align*}
        \KL(p_{n,\omega}\Vert p_{n,\omega'}) &\leq \sum_i \pi_i \log\frac{\pi_i}{\pi_i'}\underbrace{\sum_{\bX} \int_{\cS_i} p(\bX|\eta)G_i(d\eta)}_{=1} + \sum_i \pi_i \KL\left(\int_{\cS_i} p(\bX|\eta)G_i(d\eta)\Bigg\Vert \int_{\cS_i'} p(\bX|\eta')G_i'(d\eta')\right) \\
        &\leq \sum_i \pi_i \log \frac{\pi_i}{\pi_i'} + \KL\left(\int_{\cS_i} p(\bX|\eta)G_i(d\eta)\Bigg\Vert \int_{\cS_i'} p(\bX|\eta')G_i'(d\eta')\right)\\
        &= \sum_i \pi_i |\log \pi_i - \log \pi_i'| + \KL\left(\int_{\cS_i} p(\bX|\eta)G_i(d\eta)\Bigg\Vert \int_{\cS_i'} p(\bX|\eta')G_i'(d\eta')\right) \\
        &\leq \frac{1}{c_1}\sum_i \pi_i|\pi_i - \pi_i'| + \sum_i \pi_i \left[\frac{n}{c_0}W_1(G_i, G_i')\right],
    \end{align*}
    where the last inequality follows from Lemma 6 in \cite{NguyenPosteriorContractionPopulation2015} (for the second term) and $\min \pi \geq c_1$ (for the first term).
\end{proof}

\subsubsection{Proof of Corollary \ref{corollary: kl upper dH+}}
\begin{proof}
Lemma 7 in \cite{NguyenPosteriorContractionPopulation2015} (which applies since each component has $J$ topics and non-obtuse property \ref{def: non-obtuse corner} is assumed to hold for all associated component polytopes) states that $W_1(G_i, G_i') \leq C_{\delta} d_{\cH}(\cS_i, \cS_i')$ for any two component measures $G_i, G_i'$ supported on polytopes $\cS_i, \cS_i'$, where $C_{\delta}>0$ is an absolute constant depending on $\delta$ as in \ref{def: non-obtuse corner}. Together with Lemma \ref{lemma: kl_upper_W}, we have for any permutation $\sigma$
\begin{align*}
    \KL(p_{n,\omega}\Vert p_{n,\omega'}) &\leq \frac{n}{c_0}\sum_i \pi_i W_1(G_i, G_{\sigma(i)}') + \frac{1}{c_1}\sum_i \pi_i |\pi_i - \pi_{\sigma(i)}'| \\
    &\leq \frac{nC_{\delta}}{c_0}\sum_i d_{\cH}(\cS_i, \cS_i') + \frac{1}{c_1}\sum_i |\pi_i - \pi_{\sigma(i)}'| \\
    &\leq \left(\frac{nC_{\delta}}{c_0} \vee \frac{1}{c_1}\right)\sum_i \left(d_{\cH}(\cS_i, \cS_{\sigma(i)}') + |\pi_i - \pi_{\sigma(i)}'|\right)
\end{align*}
Since the above holds for any permutation, it holds true for $\sigma^*$, the minimizer of the above sum, which gives the desired result because of the definition of $d_{\bbH+}$.
\end{proof}

\subsection{Prior Concentration Property}

We recall the definition of the KL-ball $B_K$ to be used for stating the prior concentration property. Let $\cP$ denote the space of probability densities on $\Delta^{V-1}$. For a density $p_0\in \cP$ and radius $\delta>0$, define
$$B_K(p_0, \delta) = \left\{p\in\cP : \KL(p_0\Vert p)\leq \delta^2, K_2(p_0\Vert p)\leq \delta^2\right\}$$
where 
$$K_2(p\Vert q) = P\left(\log \frac{p}{q} - \KL(p\Vert q)\right)^2 = \int \left(\log \frac{p(x)}{q(x)}\right)^2 P(dx) - \KL(p\Vert q)^2$$
is the KL-variation. We verify that under appropriate conditions, the prior puts enough mass on all sufficiently small KL-balls centered at the truth (this condition is typically referred to as the KL-property of the prior).
\begin{proposition}\label{prop: KL property prior}
    Suppose $\omega^*=(\rho^*,\pi^*)\in \Omega(\cT)$, where $\cT\in \mathfrak{T}(I,J)$. $\Pi$ is a prior distribution on $\Omega(\cT)$ satisfying conditions (1), (2), (4) from Theorem \ref{thm: density contraction rate}. Then, for all sufficiently small $\delta>0$,
    \begin{equation}
        \Pi\left(p_{n,\omega} \in B_{K}(p_{n,\omega^*, \delta}\right) \geq c\left(\frac{\delta^2}{n[n\vee \log (1/\delta)]^2}\right)^{K(V-1)+I-1}
    \end{equation}
    where the constant $c$ depends on $c_0, c_1, c_2, c_3$.
\end{proposition}
\begin{proof}
    Denoting $G_i$ to be the component measures supported on component polytope $\cS_i$ and $\pi_i$ as the corresponding path probability (and similarly $G_i^*, \cS^*, \pi^*$). Consider the set
    $$\cE_{\epsilon} = \left\{\norm{\rho(u) - \rho^*(u)}\leq \epsilon, \forall u\in \cV, |\pi_i - \pi^*_i|\leq \epsilon, \forall i\in[I]\right\}.$$
    Under the prior (by condition 3), we have
    \begin{equation}\label{eq: prior mass}
        \Pi(\cE_{\epsilon}) \geq \tilde{c}\epsilon^{K(V-1)+I-1}
    \end{equation}
    where $\tilde{c}$ is a constant depending on $c_2, c_3$. Now, we argue that for all $\omega\in \cE_{\epsilon}$, we have an upper bound (in terms of $\epsilon$) for $\KL(p_{n,\omega^*}\Vert p_{n,\omega})$ and $K_2(p_{n,\omega^*}\Vert p_{n,\omega})$. Take such a $\omega\in \cE_\epsilon$. First, for any $i$, construct the coupling $Q_i$ between $G_i^*$ and $G_i$ as the law of $(\eta_{i1}, \eta_{i2})$ where $\eta_{i1}=\sum_{j\in [J]} \beta_j \rho^*(\varphi_j^i)$ and $\eta_{i2}=\sum_{j\in[J]} \beta_j \rho(\varphi_j^i)$, where $\beta=(\beta_1,\dot,\beta_J)\sim \text{Dir}_J(\alpha_0)$. Since $\norm{\rho(u)-\rho^*(u)}\leq \epsilon$ for all nodes, it is clear that for this coupling, the expected $\norm{\eta_{i1}-\eta_{i2}}\leq \epsilon$, hence $W_1(G_i^*, G_i)\leq \epsilon$. Then, using $h^2 \leq K/2$ and Lemma \ref{lemma: kl_upper_W}, we have for any $\omega\in\cE_\epsilon$,
    \begin{align*}
        h^2(p_{n,\omega^*}, p_{n,\omega}) &\leq \frac{1}{2}\KL(p_{n,\omega^*}\Vert p_{n,\omega}) \\
        &\leq \frac{n}{2c_0}\sum_i \pi_i^*W_1(G_i^*, G_i) + \frac{1}{2c_1}\sum_i \pi_i^*|\pi_i^* - \pi_i| \\
        &\leq \epsilon\left(\frac{n}{2c_0} + \frac{1}{2c_1}\right).
    \end{align*}
    Now, let $p_{n,G_i}(\bX)=\int_{\cS_i} p(\bX|\eta)G_i(d\eta)$ be the conditional density of a document, given that it is generated from the $i$th component, thus $p_{n,\omega}(\bX)=\sum_i \pi_i p_{n,G_i}(\bX)$. For every $i$, the conditional density ratio $p_{n,G_i^*}/p_{n,G_i}\leq 1/c_0^n$ and hence,
    $$\frac{p_{n,\omega^*}}{p_{n,\omega}} = \frac{\sum_i \pi_i^* p_{n,G_i^*}}{\sum_i \pi_i p_{n,G_i}} < \sum_i \frac{\pi_i^*}{\pi_i}\frac{p_{n,G_i^*}}{p_{n,G_i}}\leq \frac{1}{c_1 c_0^n},$$
    where the inequality is trivial given all terms are positive (by expanding the terms). This implies that 
    $$\sum_{\bX \in [V]^n} p_{n,\omega^*}(\bX)^2 / p_{n,\omega}(\bX) \leq 1/c_1c_0^n,$$ where the sum is taken over all possible realizations of documents in $[V]^n$.

    Now, we invoke a bound from \cite{wong1995probability} Theorem 5 on the KL divergence. According to this result, if $\int p^2/q < M$ for two densities $p,q$, then for some universal constant $\epsilon_0>0$, as long as $h(p,q)\leq \epsilon<\epsilon_0$, we have $\KL(p\Vert q)=\cO(\epsilon^2\log (M/\epsilon))$ and $K_2(p\Vert q) =\cO(\epsilon^2[\log (M/\epsilon)]^2)$., where the big $\cO$ constants are universal.

    In our case, we have shown that $\omega\in \cE_\epsilon\Rightarrow h^2(p_{n,\omega^*}, p_{n,\omega}) \leq \epsilon \left(\frac{n}{2c_0}+\frac{1}{2c_1}\right)$, and $\int p_{n,\omega^*}^2 / p_{n,\omega} \leq 1/c_1 c_0^n$. Hence, we have for $\omega\in \cE_{\epsilon}$
    $$K_2\left(p_{n,\omega^*}\Vert p_{n,\omega}\right) = \cO\left(\frac{\epsilon(nc_1 + c_0)}{2c_0c_1}\left[\log n\log \frac{1}{c_0} + \frac{1}{2}\log \frac{2c_0}{\epsilon(nc_1+c_0)c_1}\right]^2\right).$$
    Now, if we choose $\epsilon = \delta^2/n^3$, the term inside the big $\cO$ is of the order $\delta^2[n + \log n \log (1/\delta)]^2/n^2$, in which the dominating term is $\delta^2$ for all small $\delta$ if $n > \log (1/\delta)$. On the other hand, if we set $\epsilon = \delta^2/n(\log (1/\delta)^2$, then the term in the big $\cO$ is of the order $\delta^2[\log\log(1/\delta) + \log(1/\delta) + n)^2 / [\log(1/\delta)]^2$, which is dominated by $\delta^2$ if $n\leq \log 1/\delta$. Thus, in either case, depending on $n< \log(1/\delta)$ or otherwise, we have $K_2=\cO(\delta^2)$ (note that by the result, control on $K_2$ is enough for control on $\KL$). Hence, comined with Eq \ref{eq: prior mass} we obtain
    $$\Pi(p_{n,\omega}\in B_K(p_{n,\omega^*},\delta)) \geq c\left(\frac{\delta^2}{n\left[n \vee \log (1/\delta)\right]^2}\right)^{K(V-1)+ I - 1}$$
    for some constant $c$, which depends on $c_0,c_1,c_2,c_3$.
\end{proof}

\subsection{Model Entropy}
For a metric space $(\cX, d)$, the covering number $N(\epsilon, \cX, d)$ is defined as the minimal number of balls of radius $\epsilon$, i.e. sets of the form $\{x\in \cX: d(x,x_0)\leq \epsilon\}$ which cover $\cX$. In our case we are interested in $N(\epsilon, \{P_{n,\omega}|\omega\in \Omega(\cT)\}, h)$, i.e., the covering number for our model space in terms of the Hellinger metric. It captures the model space complexity and plays a fundamental role in constructing tests to differentiate two densities from this space. However, it is easier to bound $N(\epsilon, \Omega(\cT), d_{\bbH+})$, i.e. covering number for the parameter space in terms of a suitable metric (in our case the augmented tree-directed Hausdorff metric), since it allows us to use bounds on covering number for nicer Euclidean spaces like probability simplex. Result in Corollary \ref{corollary: kl upper dH+} allows us to translate these bounds to covering number for the density space.

\begin{proposition}\label{prop: model entropy}
    For $\tilde{\Omega} = \{\omega \in \Omega(\cT)\mid \text{associated polytopes satisfy Property \ref{def: non-obtuse corner}}\}$, we have
    $$\log N(\epsilon, \{P_{n,\omega}: \omega\in\tilde{\Omega}\}, h) \lesssim \left[K(V-1)+I-1\right] \log \left(\frac{C_n}{\epsilon^2}\right)$$
    where $C_n = \frac{n}{2c_0} \vee \frac{1}{2c_1}$ and $\lesssim$ hides additive constants in terms of $I,K,V$, free of $\epsilon$.
\end{proposition}
\begin{proof}
    We argue in two steps as follows. 
    
    \textbf{Entropy for parameter space:} Suppose $\tilde{\pi}_1,\dots,\tilde{\pi}_L$ be an optimal $\delta-$cover of $\Delta^{I-1}$ with $L=N(\delta, \Delta^{I-1}, \norm{\cdot}_1)$ and $\tilde{\theta}_1,\dots,\tilde{\theta}_M$ be an optimal $\delta'-$cover of $\Delta^{V-1}$ with $M=N(\delta', \Delta^{V-1}, \norm{\cdot}_2)$. We argue that this allows us to construct a cover for $\Omega(\cT)$, and hence also for $\tilde{\Omega}\subset \Omega(\cT)$. Define
    $$\cC = \{(\theta_1,\dots,\theta_K, \pi) | \theta_k\in \{\tilde{\theta}_j\}_{j\in [M]}, \pi\in\{\tilde{\pi}_\ell\}_{\ell\in[L]}\}$$
    with $|\cC| = M^K L$. For each $\tilde{\omega}=(\theta_1,\dots,\theta_k,\pi)\in \cC$, define $\omega=(\rho,\pi)$ as $\rho(v_k)=\theta_k$ for all $k\in [K]$ and $\pi(\varphi_i)=\pi_i$ for all $i$ (fixing an enumeration $v_1,\dots,v_K$ of $\cV$ and $\varphi_1,\dots,\varphi_I$ of $\Phi(\cT)$). Thus, every $\tilde{\omega}\in \cC$ can be mapped to a unique $\omega\in \Omega(\cT)$, call this map $h$.

    Now, given an arbitrary $\omega\in \Omega(\cT)$, let $\theta_1,\dots,\theta_K$ be the topics with $\theta_k = \rho(v_k)$ and $\pi_i = \pi(\varphi_i)$. By construction, there is an element $\tilde{\omega}=(\tilde{\theta}_1,\dots,\tilde{\theta_K},\tilde{\pi})\in \cC$ such that $\norm{\theta_k - \tilde{\theta}_k}_2 \leq \delta'$ for all $k$ and $\norm{\pi - \tilde{\pi}}_1 \leq \delta$. Hence we have
    $$d_{\bbH+}(\omega, h\circ \tilde{\omega}) \leq \sum_i d_{\cH}(\cS_i, \cS_i') + |\pi_i - \pi_i'| < \sum_i \delta' + \delta = I\delta' + \delta$$
    where we used the fact that
    $$d_{\cH}(\cS, \cS') \leq \max_{\theta\in \text{extr}\cS}\min_{\theta'\in\text{extr}\cS'} \norm{\theta - \theta'}_2 \vee \max_{\theta'\in\text{extr}\cS'}\min_{\theta\in\text{extr}\cS}\norm{\theta-\theta'}_2$$
    which follows from Lemma 1(a) \cite{NguyenPosteriorContractionPopulation2015}. Hence, for every $\omega\in \Omega(\cT)$, there exists $\tilde{\omega}\in \cC$ such that $d_{\bbH+}(\omega, h\circ\tilde{\omega}) < I(\delta' + \delta)$. Thus, taking $\delta'=\epsilon/2I$ and $\delta=\epsilon/2$, we get
    $$N(\epsilon, \Omega(\cT), d_{\bbH+}) \leq N\left(\frac{\epsilon}{2}, \Delta^{I-1}, \norm{\cdot}_1\right) \times N\left(\frac{\epsilon}{2I}, \Delta^{V-1}, \norm{\cdot}_2\right)^K.$$
    From existing upper bounds on covering numbers and using the fact that $\norm{x}_2 \leq \norm{x}_1$ for all $x\in \bbR^V$, we have
    \begin{align*}
        N\left(\frac{\epsilon}{2}, \Delta^{I-1}, \norm{\cdot}_1\right) &\leq \left(\frac{10}{\epsilon}\right)^{I-1}\\
        N\left(\frac{\epsilon}{2I}, \Delta^{V-1}, \norm{\cdot}_2\right) &\leq N\left(\frac{\epsilon}{2I}, \Delta^{V-1}, \norm{\cdot}_1\right) \leq \left(\frac{10I}{\epsilon}\right)^{V-1}
    \end{align*}
    which together with the previous display gives
    $$N(\epsilon, \Omega(\cT), d_{\bbH+}) \leq C \left(\frac{1}{\epsilon}\right)^{K(V-1)+I-1}$$
    where $C=10^{I-1}\times (10I)^{K(V-1)}$ is a constant in terms of $\epsilon$. Thus, we have
    $$\log N(\epsilon, \tilde{\Omega}, d_{\bbH+}) \leq \log C + \left[K(V-1)+I-1\right]\log (1/\epsilon).$$

    \textbf{Cover for density space:} Let $\omega_1,\dots,\omega_H$ be an $\epsilon-$optimal cover for $\tilde{\Omega}$. We argue that $\{p_{n,\omega_j} | j \in [H]\}$ is a $\delta-$cover for $\{P_{n,\omega}|\omega\in \tilde{\Omega}\}$, for an appropriate $\delta$ depending on $\epsilon$. Indeed, thanks to Corollary \ref{corollary: kl upper dH+} and the fact that any $\omega\in \tilde{\Omega}$ satisfies non-obtuse corner property, for every such $\omega$, there is $\omega_h$ with $d_{\bbH+}(\omega,\omega_h)\leq \epsilon$ and hence
    $$h^2(p_{n,\omega}, p_{n,\omega_h}) \leq \frac{1}{2}\KL(p_{n,\omega}\Vert p_{n,\omega_h}) \leq \left(\frac{nC_\delta}{2c_0}\vee \frac{1}{c_1}\right)d_{\bbH+}(\omega,\omega_h) < C_n \epsilon.$$
    Thus, $\{p_{n,\omega_j} | j \in [H]\}$ is a $\sqrt{C_n \epsilon}-$cover, i.e. $\delta=\sqrt{C_n \epsilon}$, hence $\epsilon = \delta^2/C_n$. Thus,
    $$N(\epsilon, \{P_{n,\omega}|\omega\in \tilde{\Omega}\}, h) \leq N(\epsilon^2/C_n, \tilde{\Omega}, d_{\bbH+})$$
    which gives, using the covering number bound from the first part of the proof
    $$\log N(\epsilon, \{P_{n,\omega}|\omega\in \tilde{\Omega}\}, h) \leq \log C + \left[K(V-1)+I-1\right]\log (C_n / \epsilon^2).$$
\end{proof}

\subsection{Proof of Theorem \ref{thm: density contraction rate}}\label{app: proof of density rate}

We apply Theorem 8.9 from \cite{GhosalVaartFundamentalsNonparametricBayesian2017} with $\cP_{m,1}=\{p_{n,\omega}| \omega\in \tilde{\Omega}\}$, where $\tilde{\Omega}$ is defined in Proposition \ref{prop: model entropy} and $\cP_{m,2}=\emptyset$ and verify the conditions in the theorem. Trivially $\Pi(\cP_{m,2})=0$, so the last condition is satisfied. Also, we note that $m\epsilon^2_{m,n}\to \infty$ with $\epsilon_{m,n}=\sqrt{\log m / m}$ in part (1) and $\epsilon_{m,n}=\sqrt{\log (m\vee n)/m}$ in part (2) and moreover, $\epsilon_{m,n}\to 0$ in both cases.

\textbf{Case 1:} (fixed $n$ case) For condition 1, we use Proposition \ref{prop: KL property prior} to obtain (note that as $\epsilon_m=\sqrt{\log m/m}\to 0$, $n<\log (1/\epsilon_{m})$ eventually)
    \begin{align*}
        \log \Pi(p_{n,\omega}\in B_K(p_{n,\omega^*},\epsilon_{m})) &\gtrsim (K(V-1)+I-1)\log\left(\frac{\epsilon_{m}^2}{n \log (1/\epsilon_{m})}\right)\\
        &= -c'\left( 2\log(1/\epsilon_{m}) + \log \log (1/\epsilon_{m}) + \log{n}\right)\\
        &\gtrsim  -C' \log m = -C' m \epsilon_{m}^2
    \end{align*}
    for sufficiently large constant $C'$ depending on $K,V,I$, since $\log (1/\epsilon_m) = \frac{1}{2}(\log m - \log\log m)$ is dominated by $\log m$ as $m\to\infty$.

    For condition 2, we use Proposition \ref{prop: model entropy} along with a similar argument to obtain
    \begin{align*}
        \log N(\epsilon_{m}, \cP_{m,1}, h) &\lesssim c'\log \frac{\sqrt{C_n}}{\epsilon_{m}} \\
        &= c'\left(\log (1/\epsilon_{m}) + \tilde{C}\right)\\
        &\lesssim C'' m \epsilon_{m}^2.
    \end{align*}
    for some sufficiently large $C''$, following a similar argument as before. Finally, applying Theorem 8.9 from \cite{GhosalVaartFundamentalsNonparametricBayesian2017}, we have the desired result.

    \textbf{Case 2:} (Increasing $n$) We again use Proposition \ref{prop: KL property prior} noting that $n<\log(1/\epsilon_{m,n})$ iff $n<\frac{1}{2}(\log m - \log\log (m\lor n))$ if $n\lesssim \log m$, otherwise $n > \log (1/\epsilon_{m,n})$ eventually. In the first case, i.e. $n\lesssim \log m$,
    \begin{align*}
        \log \Pi(p_{n,\omega}\in B_K(p_{n,\omega^*},\epsilon_{m,n})) &\gtrsim (K(V-1)+I-1)\log\left(\frac{\epsilon_{m,n}^2}{n \log (1/\epsilon_{m,n})}\right)\\
        &= -c'\left( 2\log(1/\epsilon_{m,n}) + \log \log (1/\epsilon_{m,n}) + \log{n}\right)\\
        &\gtrsim  -C' \log (m\lor n) = -C' m \epsilon_{m,n}^2
    \end{align*}
    Since $\log(1/\epsilon_{m,n})=\frac{1}{2}(\log m - \log \log (m\lor n))$ and the dominant term in the parenthesis in the second line in the above display is either $\log m$, if $m>n$ or $\log n$, if $n>m$ - which controls the lower bound in the last line. 

    In the second case, 
    \begin{align*}
        \log \Pi(p_{n,\omega}\in B_K(p_{n,\omega^*},\epsilon_{m,n})) &\gtrsim (K(V-1)+I-1)\log\left(\frac{\epsilon_{m,n}^2}{n^3}\right)\\
        &= -c'\left( 2\log(1/\epsilon_{m,n}) + 3\log{n}\right)\\
        &\gtrsim  -C' \log (m\lor n) = -C' m \epsilon_{m,n}^2
    \end{align*}
    again following a similar argument. Thus, in both cases, condition 1 for applying Theorem 8.9 from \cite{GhosalVaartFundamentalsNonparametricBayesian2017} holds. For condition 2, we again use
    Equation \eqref{prop: model entropy} to obtain
    \begin{align*}
        \log N(\epsilon_{m,n}, \mathcal{P}_{m,1}, h) &\lesssim c'\log \frac{\sqrt{C_n}}{\epsilon_{m,n}} \\
        &= c'\left(\log (1/\epsilon_{m,n}) + \frac{1}{2}\log C_n\right)\\
        &\lesssim C'' m \epsilon_{m,n}^2.
    \end{align*}
    where again, the dominating term is either $\log m$ (coming from $\log (1/\epsilon_{m,n})$ or $\log n$ (coming from the second term) and they are both dominated by $\log (m\lor n)$. Thus the desired result follows by applying Theorem 8.9 from \cite{GhosalVaartFundamentalsNonparametricBayesian2017}.

\subsection{Proofs in Section \ref{sec:parameter estimation}}

\subsubsection{Additional Results about the union Hausdorff metric}
\begin{lemma}\label{lemma: union Hausdorff identifies polytopes}
    Given $\rho,\rho'\in\Omega(\cT)$, let $\cS_1,\dots,\cS_I$ (and $\cS_1',\dots,\cS_I'$) be the component topic polytopes for $\rho$ (resp. $\rho'$) where $I=|\Phi(\cT)|$. If  the following holds additionally for either one of $\rho,\rho'$:
    \begin{enumerate}[label={(A\arabic*')}]
        \item The component polytopes satisfy $\dim (\cA_i \cap \cA_{i'}) < \min \{\dim \cA_i, \dim \cA_{i'}\}$, where $\cA_i = \aff \cS_i$. \label{assume:diff_affine_space2}
    \end{enumerate}
    then, $d_{\cU\cH}(\rho,\rho')=0 \Rightarrow \cS_i = \cS_{\sigma(i)}'$ for all $i$, up to a permutation $\sigma$ of $[I]$. 
\end{lemma}

We note that the condition \ref{assume:diff_affine_space2} is slightly stronger than Assumption \ref{assume:diff_affine_space} - Assumption \ref{assume:diff_affine_space} allows the possibility of $\cS_i \subset \cS_{i'}$ with $\dim \cS_i < \dim \cS_{i'}$, while such a case is not allowed under condition (2) of the above result. In the case that each path of $\cT$ has length $J$ and topics in a component are affinely independent, the first condition is true. On the other hand, even if the paths have different depths, it is possible that the component polytopes have the same dimension, thereby satisfying the first condition. In particular, assumption \ref{assume:diff_affine_space} along with the fact that component polytopes have same dimension imply condition \ref{assume:diff_affine_space2}.

\begin{proof}
    Let $\rho$ satisfy Assumption \ref{assume:diff_affine_space2} and $(I,\bJ,K)$ be the size of $\cT$. Since $d_{\cU\cH}(\rho,\rho')=0$, we know that $\cup_i \cS_i = \cup_i \cS_i'$. Let $\cS=\cup_i \cS_i$ and $\cA_i=\aff \cS_i$, then by assumption, we know that $\cA_1,\dots,\cA_I$ are affine spaces such that their pairwise intersections are lower dimensional. Furthermore, $\cA_i\cap \cS = \cup_j (\cA_i\cap \cS_j) = \cS_i \cup (\cup_{j\neq i} \cA_i \cap \cS_j)$, showing that the intersection of $\cS$ with any one of these affine spaces is the union of a convex polytope ($\cS_i$) and other convex polytopes ($\cA_i\cap\cS_i$), of dimension strictly smaller than $\cA_i$ (some or all of these could be empty). This shows that $\{\cA_1,\dots,\cA_I\}$ is a unique set of $I$ affine spaces such that $\cS\subset \cup_i \cA_i$, $\dim \cA_i\cap \cA_{i'} < \min \{\dim \cA_i, \dim \cA_{i'}\}$ and $\cA_i\cap \cS$ contains only one convex polytope of same dimension as $\cA_i$.

    Thus, since $\cS=\cS'$, also consisting of exactly $I$ component polytopes and there are precisely $I$ affine spaces (with pairwise trivial intersection), it shows that each component of $\cS'$ must be lying in precisely one of these affine spaces. As a result, $\cS_i'$ can be identified from $\cS'$.
\end{proof}

We state another result which connects $d_{\bbH}$ and $d_{\cU\cH}$, where we recollect $d_{\bbH}$ to be the tree-directed Hausdorff metric
$$d_{\bbH}(\rho,\rho') = \min_{\sigma\in \bbS_I} \sum_{i\in[I]} d_{\cH}(\cS_i, \cS_{\sigma(i)}').$$

\begin{lemma}\label{lemma: tree-directed and union equivalence}
    Suppose $\{\rho_n\}_{n\geq 0}$ is a sequence of topic maps in $\Omega(\cT)$. Then we have
    \begin{enumerate}
        \item For any $\rho_0\in \Omega(\cT)$
        $$d_{\bbH}(\rho_n,\rho_0) \to 0 \Rightarrow d_{\cU\cH}(\rho_n,\rho_0)\to 0.$$
        \item If $\rho_0$ satisfies Assumption \ref{assume:diff_affine_space2}, then
        $$d_{\cU\cH}(\rho_n,\rho_0) \to 0 \Rightarrow d_{\bbH}(\rho_n,\rho_0)\to 0.$$
    \end{enumerate}
\end{lemma}
The proof uses the fact that $\Omega(\cT)$ is compact and hence, any sequence $\rho_n$ has a convergent subsequence. We note that the above lemma states that if only $\rho_0$ satisfies assumption \ref{assume:diff_affine_space2}, then $d_{\cU\cH}$ and $d_{\bbH}$ are equivalent in terms of local convergence. We also remark that due to the fact that $d_{\bbH}$ is precisely $d_{\bbH+}$ for a fixed DRT $\cT$ and uniform path probabilities, we have $d_{\bbH}(\rho,\rho')=0$ iff there exists a unique automorphism (i.e., isomorphism between $\cT$ and $\cT$) $\sigma$ on $\cT$ such that $\rho(u) = \rho'(\sigma(u))$ for all vertex $u$ of $\cT$. It is important to note that the last two lemmas use the fact that associated topic maps are defined on the same tree, so that the tree size is fixed.  

\begin{proof}
Firstly, for each $n$, $\rho_n:\cV\to\Delta^{V-1}$ where $\cV$ is the set of vertices of the underlying DRT (with finite number of vertices). Enumerating these vertices as $v_1,\dots,v_K$, we can think of each $\rho_n$ as $\theta^{(n)}=(\theta_1^n,\dots,\theta_K^n)\in\left(\Delta^{V-1}\right)^K$ where $\theta_k^n=\rho_n(v_k)$. Since $\left(\Delta^{V-1}\right)^K$ is compact, any such sequence $\theta^{(n)}$ has a convergent subsequence. If we look along that subsequence, say $\theta^{*}$ is the limit point, then we can define $\rho^*:\cV\to\Delta^{V-1}$ as $\rho^*(v_k)=\theta^*_k$. It is easy to check that $d(\rho_n,\rho_0)\to 0$ as $n\to\infty$, implies that $d(\rho^*,\rho_0)=0$ where $d$ can be both $d_{\mathcal{UH}}$ and $d_{\mathbb{H}}$. Thus, we need to relate 
$d_{\mathcal{UH}}(\rho^*,\rho_0)=0$ and $d_{\mathbb{H}}(\rho^*,\rho_0)=0.$

\textit{Part (1)}
Suppose $d_{\bbH}(\rho^*,\rho_0)=0$. Using an enumeration $\varphi^i$ of the maximal paths of $\cT$, this means, by definition of $d_{\bbH}$ that there exists a permutation $\sigma$ of $[I]$ such that $\cS_i^* = \cS_{\sigma(i)}^0$ (here $\cS_i^*$ is the topic polytope corresponding to path $\varphi^i$ using topic map $\rho^*$ and similarly $\cS_i^0$ for $\rho_0$). This trivially implies that $\cup_i \cS_i^* = \cup_i \cS_i^0$, which, by definition of $d_{\cU\cH}$ implies that $d_{\cU\cH}(\rho^*,\rho_0)=0$.

\textit{Part (2)}
Conversely, suppose $d_{\mathcal{UH}}(\rho^*,\rho_0)=0$. Thus, we obtain that $\cup_i \cS_i^0 = \cup_i \cS_i^*$, where we know that $\cS_i^0$ satisfies the conditions in the statement. Hence, from $\cup_i \cS_i^0$, each of the component polytopes can be identified uniquely from Lemma \ref{lemma: union Hausdorff identifies polytopes}. Moreover, by assumption \ref{assume:all_exposed}, each $\cS_i^0$ uniquely identifies the associated topics and $\dim \cS_i^0 = J_i-1$, where $J_i$ is the length of path $\varphi_i$ of $\cT$ associated with this component polytope.

Note that there are exactly $I$ many components in $\cup_i \cS_i^*$ as well and both $\{\cS_i^0\}_i$ and $\{\cS_i^*\}_i$ come from the same underlying tree $\cT$ (hence for each $\varphi\in\Phi(\cT)$ of depth $J_i$, there must be a component polytope with exactly $J_i$ extremal points), there must be a one-to-one correspondence between them. This implies that $d_{\mathbb{H}}(\rho^*,\rho_0)=0$.
\end{proof}

\subsection{Proofs in Section \ref{sec:parameter estimation}}

\subsubsection{Proof of Lemma \ref{lemma:lb_mass_small_ball}}

Since all topics $\theta\in\Delta^{V-1}$, $\norm{\theta}\leq 1$. First consider the case where all topics are affinely independent, hence $\dim\cS=K-1$. Suppose $\eta=\sum_{i=1}^{K} \beta_i^* \theta_i$, for some $\beta^*\in \Delta^{K-1}$. Consider the set $A_1(\epsilon)=\{\beta\in\Delta^{K-1}: |\beta_i - \beta_i^*| <\epsilon/K, i=1,\dots,K-1\}$. For any $\beta\in A_1(\epsilon)$, and $\eta'=\sum_k \beta_k\theta_k$, we have $\norm{\eta-\eta'}=\norm{\sum_k (\beta_k-\beta^*_k)\theta_k} \leq 2\sum_{k\in[K-1]} |\beta_k-\beta_k^*|\leq 2\epsilon$. WLOG assume $\beta_K^*\geq 1/K$. Then, for all $\epsilon<1/K$ we have
\begin{align*}
    G(B(\eta,2\epsilon)\cap\cS) &\geq G(A_1(\epsilon)) \\
    &= \frac{\Gamma(K\alpha)}{\Gamma(\alpha)^K} \int_{A_1(\epsilon)} \prod_{k=1}^{K-1}\beta_k^{\alpha-1}\left(1- \sum_{k=1}^K\beta_k\right)^{\alpha-1} d\beta \\
    &\geq \frac{\Gamma(K\alpha)}{\Gamma(\alpha)^K} \int_{\substack{A_1(\epsilon)\\\beta_k\geq \delta\,\forall k}} \prod_{k=1}^{K-1}\beta_k^{\alpha-1}\left(1- \sum_{k=1}^K\beta_k\right)^{\alpha-1} d\beta \\
    &\geq \min\{1, \delta^{\alpha-1}\} \frac{\Gamma(K\alpha)}{\Gamma(\alpha)^K} \left(\prod_{i\leq K-1} \int_{\max\{\beta_i^*-\epsilon/K, \delta\}}^{\min\{\beta^*+\epsilon/K, 1\}} \beta_i^{\alpha-1}d\beta_i\right)\\
    &\geq \min\{1, \delta^{\alpha-1}\} \frac{\Gamma(K\alpha)}{\Gamma(\alpha)^K} \prod_{i\leq K-1} \min\{1,\delta^{\alpha-1}\} 2\epsilon/K \\
    &=  \frac{2\Gamma(K\alpha)}{K^{K-1}\Gamma(\alpha)^K} \min\{1, \delta^{\alpha-1}\}^{K} \epsilon^{K-1} \\
    &= \begin{cases}
        C(\alpha,K)\epsilon^{K-1} &\alpha\leq 1 \\
        C(\alpha,K) \epsilon^{K\alpha-1} &\alpha>1
    \end{cases}
\end{align*}
where we used for $\delta<x<1$, $x^{\alpha-1}\geq \delta^{\alpha-1}$ for $\alpha>1$ and $x^{\alpha-1}\geq 1$ for $\alpha\leq 1$, thus $x^{\alpha-1}\geq \min\{1, \delta^{\alpha-1}\}$ covering both cases. The last line follows from this and the choice of $\delta=\epsilon/3K$, where $C(\alpha,K)$ is a constant. Moreover, if $\alpha\geq \alpha_1$, $C(\alpha,K)\geq C_0(\alpha_1,K)$ only depending on $\alpha$ (since $\Gamma(K\alpha)/\Gamma(\alpha)^K$ is increasing) and moreover, if $\alpha\leq \alpha_2$, then if $\alpha_2\leq 1$, then $\epsilon^K$ works while if $\alpha_2>1$, then $\epsilon^{K\alpha_2-1}$ works as a uniform lower bound.

Now, if $\dim \cS=p<K-1$, then for $\eta\in \cS$, we can write $\eta$ as the convex combination of some $p+1$ topics out of the extreme points of $\cS$, without loss of generality, suppose $\eta=\sum_{i=1}^{p+1} \beta_i^* \theta_i$, for some $\beta^*\in \Delta^{p}$. Let $\beta_j^*=0$ for $j=p+2,\dots,K$. Consider the set $A_1(\epsilon)=\{\beta\in\Delta^{K-1}: |\beta_i - \beta_i^*| <\epsilon/K, i=1,\dots,K-1\}$. For any $\beta\in A_1(\epsilon)$, and $\eta'=\sum_k \beta_k\theta_k$, we have $\norm{\eta-\eta'}=\norm{\sum_k (\beta_k-\beta^*_k)\theta_k} \leq 2\sum_{k\in[K-1]} |\beta_k-\beta_k^*|\leq 2\epsilon$. WLOG assume $\beta_{p+1}^*\geq 1/(p+1)$. Then, for all $\epsilon<1/(p+1)$, we have
\begin{align*}
    G(B(\eta,2\epsilon)\cap\cS) &\geq G(A_1(\epsilon)) \\
    &= \frac{\Gamma(K\alpha)}{\Gamma(\alpha)^K} \int_{A_1(\epsilon)} \prod_{k=1}^{K-1}\beta_k^{\alpha-1}\left(1- \sum_{k=1}^K\beta_k\right)^{\alpha-1} d\beta \\
    &\geq \delta^{1\vee\alpha -1} \frac{\Gamma(K\alpha)}{\Gamma(\alpha)^K} \left(\prod_{i\leq p} \int_{\max\{\beta_i^*-\epsilon/K, \delta\}}^{\min\{\beta^*+\epsilon/K, 1\}} \beta_i^{\alpha-1}d\beta_i\right)\left(\prod_{i=p+2}^{K} \int_{\delta}^{\epsilon/K} \beta_i^{\alpha-1}d\beta_i\right)\\
    &\geq \frac{\Gamma(K\alpha)}{\Gamma(\alpha)^K} \left(\delta^{1\vee\alpha-1}\right)^{p+1} \left( \frac{\epsilon}{K}\right)^{p} \frac{\left[\left(\frac{\epsilon}{K}\right)^{\alpha} - \delta^\alpha\right]^{K-p-1}}{\prod_{i=p+2}^{K} \alpha}\\
    &= \frac{\Gamma(K\alpha)}{\Gamma(\alpha)^K} \left(\delta^{1\vee\alpha-1}\right)^{p+1} \left( \frac{\epsilon}{K}\right)^{p} \frac{\left[\left(\frac{\epsilon}{K}\right)^{\alpha} - \delta^\alpha\right]^{K-p-1}}{\alpha^{K-p-1}}\\
    &= \frac{\Gamma(K\alpha)}{\Gamma(\alpha)^K} \left(\frac{\epsilon}{3K}^{1\vee\alpha-1}\right)^{p+1} \left( \frac{\epsilon}{K}\right)^{p} \frac{\left[\left(\frac{\epsilon}{K}\right)^{\alpha} - \left(\frac{\epsilon}{3K}\right)^\alpha\right]^{K-p-1}}{\alpha^{K-p-1}}\\
    &\geq C(\alpha,K,p) \epsilon^{(p+1)(1\vee\alpha-1)} \epsilon^{p} \epsilon^{\alpha(K-p-1)} \\
    &=\begin{cases}
        C(\alpha,K,p) \epsilon^{p + \alpha(K-p-1)} &\alpha \leq  1 \\
        C(\alpha,K,p) \epsilon^{\alpha K-1} &\alpha>1
    \end{cases}
\end{align*}
using $\delta=\epsilon/3K$ as before.

\subsection{Proof of Theorem \ref{thm: inverse bound}}
\begin{proof}
    Given a document $\doc=(X_1,\dots,X_n)$, define $\hat{\eta}\in\Delta^{V-1}$, such that $\hat{\eta}_v=\sum_{j\in[n]}1(X_j=v)/n$, to be the observed document mean. By definition of total variation,
    $$d_{\TV}(p_{n,\omega},p_{n,\omega_0})= \sup_{A} \left|P_{n,\omega}(A) - P_{n,\omega'}(A)\right|$$
    where the supremum is taken over all measurable subsets of $\Delta^{V-1}$. Since the conditional distribution of $\doc$ given $\eta$ is a product of i.i.d. multinomials, we have by Hoeffding's inequality
    $$P_{\doc|\eta}\left(\max_{v\in[V]} |\hat{\eta}_v - \eta_v| \geq \epsilon\right)\leq 2\exp\left(-2n\epsilon^2\right)$$
    for every $\epsilon>0$. Now, $\eta\sim G=\sum_k \pi_k G_k$, where $G_k$ are the component latent measures. Treating $\eta\sim G$ as random, the above inequality holds with probability 1, hence
    \begin{align*}
        P_{\eta\times \doc|\omega}\left(\norm{\hat{\eta}-\eta}\geq \epsilon\right) &\leq P_{\eta\times \doc|\omega}\left(\max_{v\in[V]} |\hat{\eta}_v - \eta_v|\geq \frac{\epsilon}{\sqrt{V}}\right) \\
        &\leq 2V\exp\left(\frac{-2n\epsilon^2}{V}\right).
    \end{align*}
    The same bound also holds for $P_{\eta\times \doc|\omega'}$, where $P_{\eta\times \doc|\omega}$ is the joint distribution of $(\eta,\doc)$ for parameter $\omega$, under $\eta\sim G(\omega), \doc|\eta \sim \otimes^n \text{Mult}(\eta)$. Now, let us define $B=\{\norm{\hat{\eta}-\eta}<\epsilon\}$. For any measurable $A\subset\Delta^{V-1}$,
    \begin{align}
        &\left|P_{n,\omega}(\hat{\eta}\in A) - P_{n,\omega'}(\hat{\eta}\in A)\right| \nonumber\\
        &= \left|P_{\eta\times \doc|\omega}(\hat{\eta}\in A; B) + P_{\eta\times \doc|\omega}(\hat{\eta}\in A; B^c) - P_{\eta\times \doc|\omega'}(\hat{\eta}\in A; B) - P_{\eta\times \doc|\omega'}(\hat{\eta}\in A; B^c)  \right| \\
        &\geq \left|P_{\eta\times \doc|\omega}(\hat{\eta}\in A; B) - P_{\eta\times \doc|\omega'}(\hat{\eta}\in A; B) \right| - 4V\exp\left(-\frac{2n\epsilon^2}{V}\right).\label{eq:proof_inv_bound_1}
    \end{align}
    We now show that there exists $A$ for which the first term in the above display is bounded from below. Let $\cS=\cup_i \cS_i$ and $\cS'=\cup_i \cS_i'$ be the union of the component polytopes arising from the two parameters $\omega, \omega'$ respectively. Let $\epsilon_1= d_{\cH}(\cS,\cS')/4$. We prove that for with $\epsilon=\epsilon_1$, one of the following holds:
    \begin{enumerate}[label=(\roman*)]
        \item $\exists A^*\subset \cS\setminus \cS'$ such that $A^*_{\epsilon}\cap \cS_{\epsilon}'=\emptyset$ and $G(A^*)\geq c\epsilon^{(1\vee \alpha)J-1}$ \\
        \item $\exists A^*\subset \cS'\setminus \cS$ such that $A^*_{\epsilon}\cap \cS_{\epsilon}=\emptyset$ and $G'(A^*)\geq c\epsilon^{(1\vee \alpha)J-1}$. 
    \end{enumerate}
    where $A_\epsilon=A + B(0,\epsilon) = \{y:d(y,A)<\epsilon\}$. We first complete the proof assuming this is correct and then prove that it indeed is true.
    
    \textbf{Assuming that (i) above is true: } Firstly, $P_{\eta\times \doc|\omega'}(\hat{\eta}\in A^*_{\epsilon}; B) \leq P_{\eta|\omega'}(\eta\in A^*_{2\epsilon}) = G'(A^*_{2\epsilon})=0$, since $G'$ is supported on $\cS'$ and $A^*_{\epsilon}\cap S_{\epsilon}'=\emptyset$. In addition,
    \begin{align*}
        P_{\eta\times \doc|\omega}(\hat{\eta}\in A^*_{\epsilon}; B) &\geq P_{\eta\times \doc|\omega}(\hat{\eta}\in A; B) \\
        &\geq P_{\eta|\omega}(\eta\in A^*) - P_{\eta\times \doc|\omega}(B^c) \\
        &\geq G(A^*) - 2V\exp\left(-\frac{2n\epsilon^2}{V}\right)\\
        &\geq c\epsilon^{(1\vee\alpha)J-1} - 2V\exp\left(-\frac{2n\epsilon^2}{V}\right).
    \end{align*}
    Combined with equation \eqref{eq:proof_inv_bound_1}, we have
    \begin{align*}
        \left|P_{n,\omega}(\hat{\eta}\in A^*_{\epsilon}) - P_{n,\omega'}(\hat{\eta}\in A^*_{\epsilon})\right| &\geq c\epsilon_1^{(1\vee\alpha)J-1} - 6V\exp\left(-\frac{2n\epsilon^2}{V}\right)
    \end{align*}
    Noting that $\epsilon=d_{\cU\cH}(\rho,\rho')/4$, this concludes the proof. A similar argument works if (ii) was true above.

    \textbf{Why one of (i) or (ii) must hold: } 
    Since $d_{\cH}(\cS,\cS')= 4\epsilon_1$, there either exists $x\in \cS$ with $d(x,\cS')= 4\epsilon_1$ or $x'\in\cS'$ with $d(x',\cS)=4\epsilon_1$. Assume the former. Let $A^*=B(x,\epsilon_1)\cap \cS$. Clearly, $A^*\subset \cS$, also, since $d(x,\cS')=4\epsilon_1$, we have $A^*\cap \cS'=\emptyset$, hence $A^*\subset \cS\setminus \cS'$. Note that $\min_{y\in A^*} d(y, \cS') = 3\epsilon_1$. Hence, for any $\epsilon\leq\epsilon_1$, say $\epsilon=\epsilon_1$, we have $A^*_{\epsilon} \cap \cS_{\epsilon}'=\emptyset$. Finally, $G(A^*)=\sum_i \pi_i G_i(B(x,\epsilon)\cap \cS_i)$. Noting that $x\in \cS_i$ for some $i$, we get $G(A^*) \geq c_1 G(B(x,\epsilon_1)\cap S_i)\geq c_1 C(J,\alpha) \epsilon_1^{(1\vee\alpha)J-1}$, by application of Lemma \ref{lemma:lb_mass_small_ball}.
\end{proof}

\subsection{Proof of Theorem \ref{thm :parameter contraction rate}}
In this proof, we assume that all $\omega$ are in the support of the prior $\Pi$. This entails that for any sequence $\omega_n$, there is a convergent subsequence $\omega_{n_k}$ with $d_{\cU\cH}(\rho_{n_k}, \rho^*)\to 0$ (here $\omega=(\rho,\pi)$ for topic map $\rho$ and path probability $\pi$), where $\omega^*$ is also in the support (and thereby satisfies the properties assumed on the prior). We first show that as $n\to\infty$, the Hellinger distance between densities is lower bounded by a function of the union Hausdorff metric between the corresponding topic maps. This requires combining the inverse bound (local property, since we need $d_{\cU\cH}\leq \epsilon_0$ for this result), along with a global identifiability property (which we can also deduce by modifying the inverse bound proof).

Firstly, we note that for $\epsilon_0$ as in Theorem \ref{thm: inverse bound}, there exists $\delta_0>0$ such that for all $\omega$ in the support of $\Pi$ with  $d_{\cU\cH}(\rho, \rho_0)\geq \epsilon_0$, we have  $\liminf_{n\to \infty} h^2(p_{n,\omega_0}, p_{n,\omega})\geq \delta_0$. This can be seen by modifying the proof of the inverse bound Theorem \ref{thm: inverse bound} - in the very last step, where we used Lemma \ref{lemma:lb_mass_small_ball} under the assumption that $\epsilon_1<\epsilon_0$, if $\epsilon_1>\epsilon_0$, then we can lower bound $G(B(x,\epsilon_1)\cap \cS_i)\geq G(B(x, \epsilon_0)\cap \cS_i) \geq C\epsilon_0^{(1\vee \alpha)J - 1}$  (note that it is a uniform lower bound, not depending on the $d_{\cU\cH}$), which would show $\liminf_{n\to\infty} d_{\TV}(p_{n,\omega}, p_{n,\omega'}) \geq c\epsilon_0^{(1\vee \alpha)J - 1}$, as long as $d_{\cU\cH}(\rho,\rho')\geq \epsilon_0$. Since, $\sqrt{2} h \geq d_{\TV}$, we have for some $\delta_0>0$,
$$\liminf_{n\to\infty} \inf_{\substack{\omega\in\text{Supp}(\Pi)\\ d_{\cU\cH}(\rho, \rho_0)\geq \epsilon_0}} h^2(p_{n,\omega_0}, p_{n,\omega}) \geq \delta>0.$$

Now, for $\omega\in \text{Supp}(\Pi)$, if $d_{\cU\cH}(\rho,\rho_0)\leq \epsilon_0$, we can use Theorem \ref{thm: inverse bound}, to have a lower bound depending on the $d_{\cU\cH}$:
$$\sqrt{2}h(p_{n,\omega_0}, p_{n,\omega}) \geq C_1 d_{\cU\cH}(\rho,\rho_0)^{(1\vee\alpha)J-1} - 6V\exp(-n d_{\cU\cH}(\rho,\rho_0)^2 / 8V)$$
where the right side is further lower bounded by a constant multiple of $\epsilon^{(1\vee\alpha)J-1}$ if $\epsilon \leq d_{\cU\cH}(\rho,\rho_0) \leq \epsilon_0$ and $C_1\epsilon^{(1\vee\alpha)J-1} \gtrsim 6V\exp(-n\epsilon^2/8V)$, which is true is $\epsilon>\tilde{c}\sqrt{\log n/n}$. Thus, in such a case, $h^2(p_{n,\omega_0}, p_{n,\omega})\geq c'\epsilon^{2[(1\vee\alpha)J-1]}$. Combining the above cases, we obtain the following: for a sequence $\epsilon_n\to 0$,
$$\liminf_{n\to \infty} \inf_{\substack{\omega\in\text{Supp}(\Pi)\\d_{\cU\cH}(\rho,\rho_0)\geq \epsilon_n \\ \epsilon_n \gtrsim \sqrt{\log n/n}}} \frac{h^2(p_{n,\omega_0}, p_{n,\omega}) }{\epsilon_n^{2[(1\vee\alpha)J-1]}}\geq C'>0.$$

Hence, with $\epsilon_{m,n}$ as in the statement of the theorem, as $m,n\to\infty$, (which ensures $\epsilon_{m,n}\gtrsim \sqrt{\log n / n}$),
\begin{align*}
    \Pi\left(d_{\cU\cH}(\rho,\rho_0)\geq \epsilon_{m,n}\mid \corpus\right) &\leq \Pi\left(h^2(p_{n, \omega_0}, p_{n,\omega}) \geq C'\epsilon_{m,n}^{2[(1\vee\alpha)J-1]}\mid \corpus\right) \\
    &\leq \Pi\left(h(p_{n,\omega_0}, p_{n,\omega})\geq C_1'\epsilon_{m,n}^{(1\vee\alpha)J-1}\mid \corpus\right) \\
    &\leq \Pi\left(h(p_{n,\omega_0}, p_{n,\omega}\geq C_1'\left[\log (m\vee n) \left(\frac{1}{m} + \frac{1}{n}\right)\right]^{1/2}\mid \corpus\right)\to 0
\end{align*}
in probability, where the last convergence follows from density estimation rate Theorem \ref{thm: density contraction rate}.

\section{Additional Details about Simulations in Section \ref{sec:numerical exp}}\label{app:numerical exp}

\subsection*{Relation between $d_{\cU\cH}$ and $d_{L_2}$}

Recall the Hausdorff metric is defined as
$$d_{\cH}(A, B) = \max_{x\in A} d(x,B)\vee \max_{y\in B} d(y,A) = \inf \left\{\epsilon>0 : A\subset B_{\epsilon}, B\subset A_{\epsilon}\right\},$$
where $A_\epsilon = A + B(0,\epsilon) = \cup_{x\in A} B(x,\epsilon)$. Consider $\cS=\cup_{i\in[I]} \cS_i$ and $\cS'=\cup_{i\in[I]} \cS_i'$, where $\cS_i$ (and $\cS_i'$) are the component polytopes. Fix a permutation $\sigma$, wlog assume the identity permutation (otherwise relabel components of $\rho'$). Firstly
$$d_{\cH}(\cS, \cS') \leq \max \{d_{\cH}(\cS_i, \cS_i'):i\in[I]\}$$ 
since suppose $r$ is the max on the right side, then $\cS_i\subset \cS_i' + B(0,r)$ and $\cS_i'\subset \cS_i+B(0,r)$. Then, $\cS = \cup_i \cS_i \subset \cup_i \left(\cS_i'+B(0,r)\right) = \cS'+B(0,r)$ and also the converse, thereby $d_{\cH}(\cS,\cS')\leq r$. By \cite[Lemma 1(a)]{NguyenPosteriorContractionPopulation2015}, for convex polytopes we have $d_{\cH}(\cS_i, \cS_i') \leq d_M(\cS_i, \cS_i')$, where $d_M$ is the minimum-matching metric
\begin{equation}\label{eq: minimal matching distance}
    d_{M}(\cS_1, \cS_2) = \max_{\theta\in \extr \cS_1} \min_{\theta'\in \extr \cS_2} \norm{\theta-\theta'} \vee \max_{\theta'\in \extr \cS_2}\min_{\theta\in \extr \cS_1} \norm{\theta-\theta'}
\end{equation}
which in turn, for convex polytopes $\cS_1,\cS_2$, each with $J$ extreme points $\theta_1,\dots,\theta_J$ (and $\theta_1',\dots,\theta_J'$), satisfies
$$d_M(\cS_1, \cS_2) \leq \min_{\tau\in\bbS_J}\sum_j \norm{\theta_j- \theta_{\tau(j)}'}.$$
Hence, (as $\max_i a_i \leq \sum_i a_i$ for $a_i\geq 0$) we obtain
$$d_{\cH}(\cS,\cS') \leq \inf_{\sigma\in \bbS_I} \sum_i \min_{\tau_i\in\bbS_J} \sum_j \norm{\theta_{i,j} - \theta_{\sigma(i),\tau_i(j)}'}$$
and we recognize the right side to be precisely $d_{L_2}$. Hence, we have $d_{\cU\cH}(\rho,\rho')\leq d_{L_2}(\rho,\rho')$. Note that in the last display, the extreme points of the convex polytope $\cS_i$ are denoted as $\theta_{i,1},\dots,\theta_{i,J}$.


\begin{table}[]
\centering
\begin{tabular}{cc}

    \begin{minipage}{.4\linewidth}
        \centering
        \begin{tabular}{ccc}
\multicolumn{3}{c}{Experiment 1 ($V=10, K=5$)}                    \\
\multicolumn{1}{l}{}       & \multicolumn{1}{l}{}  & \multicolumn{1}{l}{}                        \\ \hline
\multicolumn{1}{|c|}{$m$} & $n=50$ & \multicolumn{1}{c|}{$n=100$} \\ 
\hline
\multicolumn{1}{|c|}{200}  & 27.40  & \multicolumn{1}{c|}{43.94}  \\
\multicolumn{1}{|c|}{318}  & 39.96  & \multicolumn{1}{c|}{67.65}  \\
\multicolumn{1}{|c|}{503}  & 61.57  & \multicolumn{1}{c|}{107.64} \\
\multicolumn{1}{|c|}{796}  & 95.25  & \multicolumn{1}{c|}{164.10} \\
\multicolumn{1}{|c|}{1261} & 142.16 & \multicolumn{1}{c|}{265.47} \\
\multicolumn{1}{|c|}{1996} & 231.18 & \multicolumn{1}{c|}{406.71} \\
\multicolumn{1}{|c|}{3159} & 349.60 & \multicolumn{1}{c|}{626.76} \\
\multicolumn{1}{|c|}{5000} & 554.15 & \multicolumn{1}{c|}{979.62} \\ \hline
\end{tabular}
\vspace{1mm}

    \end{minipage} &

    \begin{minipage}{.4\linewidth}
        \centering
        \begin{tabular}{ccc}
\multicolumn{3}{c}{Experiment 2 ($V=30, K=7$)}                    \\ 
\multicolumn{1}{l}{}       & \multicolumn{1}{l}{}  & \multicolumn{1}{l}{}                        \\ \hline
\multicolumn{1}{|c|}{$m$} & $n=150$ & \multicolumn{1}{c|}{$n=300$} \\ \hline
\multicolumn{1}{|c|}{400}  & 146.66  & \multicolumn{1}{c|}{259.09}  \\
\multicolumn{1}{|c|}{711}  & 259.63  & \multicolumn{1}{c|}{439.72}  \\
\multicolumn{1}{|c|}{1264}  & 433.12  & \multicolumn{1}{c|}{762.77} \\
\multicolumn{1}{|c|}{2249}  & 747.27  & \multicolumn{1}{c|}{1342.87} \\
\multicolumn{1}{|c|}{4000} & 1320.04 & \multicolumn{1}{c|}{2610.74} \\\hline
\end{tabular}
\vspace{2mm}
    \end{minipage} 
    
\end{tabular}
\caption{Average run time (in seconds) for the Gibbs sampler for experiments 1 (left) and 2 (right). }
\label{table: run time}
\end{table}

\subsection*{Runtime}
The running time for the Gibbs sampler in experiments 1 and 2 are shown in Table \ref{table: run time}. For a choice of $(m,n)$, we recall that the corpus had $mn$ words in total. For $m=5000, n=100$ (total $5\times 10^5$ words in the corpus) case in experiment 1, total time required was around 16 minutes (about $0.0000356$ seconds per iteration per document). For $m=4000, n=300$ (total $1.2\times 10^6$ words in the corpus) case in experiment 2, total time required was around 43 minutes (about $0.0001186$ seconds per iteration per document). The run times shown in the table are for a single chain, in the simulations, 8 chains were run in parallel on a 2.9 GHz Quad-Core Intel Core i7 processor machine. Apart from the total size of the corpus, the run time also depends on the vocabulary size $V$ (rather small here in our toy experiments) and the size of the DRT — for experiment 1, the DRT has two paths while for experiment 2, the DRT has four paths.

\subsection{DRTs for Experiment 3}

We use a total of 8 different DRTs (including LDA with true total number of topics) - these are shown in Figure \ref{fig: exp3 DRTs}. The tree \textit{tree0} in the figure is the true DRT responsible for the data generation in this experiment. The results obtained from this experiment are discussed in the main text, see Section \ref{sec:numerical exp}. We note that the first 5 DRTs all have $K=5$ (which is the actual total number of topics for this experiment), while the last 3 DRTs have more - however, the true DRT \textit{tree0} is a sub-tree of each of these last 3 trees. We have seen that while these tree have similar held-out log-likelihood compared to the first tree, one of the paths for these 3 trees have very small probability assigned to it. On the other hand, out of the first 5 DRTs, \textit{tree0} has a significantly higher held-out log-likelihood for each data instance.

\begin{remark}[Note on initialization]
    For each dataset, we ran the Gibbs sampler across 8 independent chains. Four of these chains were started randomly. The other four were initialized using a combination of usual LDA and hierarchical clustering. We provide a brief description of the latter method. Given the total number of topics $K$ based on the DRT, we fit a LDA model with $K$ topics and a small Dirichlet parameter $\alpha$ (to promote sparsity) - the results from this model provide the topics and also the mixing proportions for each document $\beta_i\in\Delta^{K-1}$. However, in our model, each document is associated to only $J$ topics (out of the $K$), where $J$ is the depth - hence, for each document, we chose the $J$ coordinates of the mixing probabilities with highest values - made the other coordinates 0 and normalized the vector, to obtain $\tilde{\beta}_i\in\Delta^{K-1}$ with $\norm{\tilde{\beta}_i}_0 = J$. Then we perform an agglomerative hierarchical clustering on these modified document-wise mixing probabilities. Finally, the sub-tree from this hierarchical clustering starting from the root, matching the desired DRT was chosen -- this provided the initialization for the collapsed Gibbs sampling for our model. Although a heuristic approach, in most of the cases, this initialization turned out to be better than random initialization, based on the log-likelihood at the end of the chains.
\end{remark}

\subsection{Additional Details about NYT Corpus Data Analysis}
We have already discussed pre-processing of the dataset. The training data has 641 documents with a vocabulary of $V=500$. We use 25 different DRTs (shown in Figure \ref{fig: nyt trees}) and also usual LDA with $K\in [2,12]$. Among the 25 DRTs used, we have all possible DRTs with $J=3, I\leq 5$  and a few others with $J=3,I=6$ and some with $J=4$. 

\begin{figure}
    \centering
    \includegraphics[width=\linewidth]{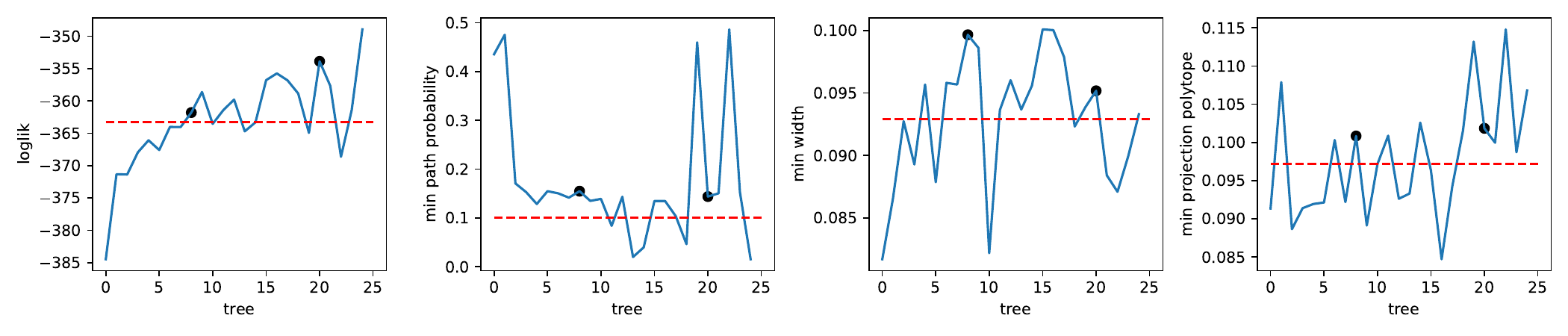}
    \caption{Visual representation of DRT-model selection for NYT data (corresponding to the results in Table \ref{tab: nyt DRTs}) - the red dashed lines are the cut-offs chosen and the black marked points are \textit{Tree8} and \textit{Tree20}, which passes all the filters.}
    \label{fig: nyt model selection}
\end{figure}

\begin{figure}
    \centering
    \includegraphics[width=\linewidth]{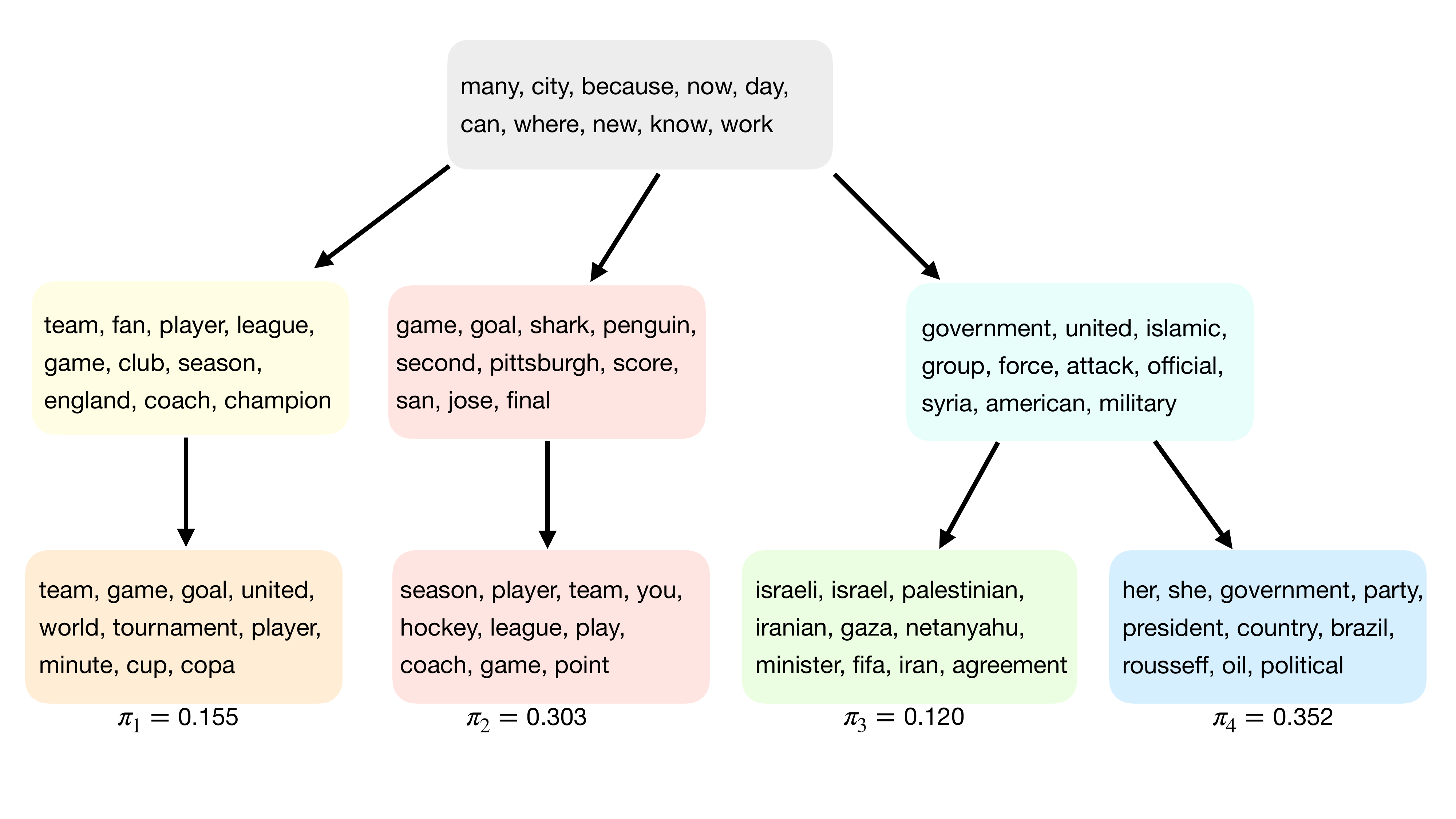}
    \caption{Topic hierarchy for Tree 8}
    \label{fig: nyt hierarchy 8}
\end{figure}

\subsubsection*{DRT Selection:} Model selection in this case was done using held-out log-likelihood along with few other geometric considerations from the model. Table \ref{tab: nyt DRTs} shows the results from the DRTs (where there are at least two component polytopes) and Table \ref{tab: nyt LDAs} shows the corresponding results for the LDA model. Note that since the LDA can be seen as special case with a DRT with $I=1$, there is only one topic polytope - hence the other columns in Table \ref{tab: nyt DRTs} do not make sense for LDA models. All the LDA models have lower held-out log likelihood compared to most of the DRT models. For the DRT models, first we filtered only models where the minimum path probability was more than $0.1$, then we further filtered to cases where the minimum width and the minimum projection distance of topics to other component polytopes was more than the median of all the values. This was to ensure that the components are well-separated and each is given sufficient weight. This left us with trees \textit{Tree8} and \textit{Tree20} (shown in Figure \ref{fig: nyt model selection}, with the latter having the better held-out log likelihood. We chose this tree as the most appropriate for this dataset. However, we also include \textit{Tree8} and \textit{Tree7} (which respects the topics structure based on the corpus news category knowledge) for comparison of the topics. This is not a very in-depth model selection process, which is not the focus of this paper - but, we should keep in mind that for practical applications, interpretation of topics are not often aligned with performance measures like perplexity. We already presented the topic hierarchies for \textit{Tree20} and \textit{Tree7} in the main text, here we include the estimated topic hierarchy from \textit{Tree8}, shown in Figure \ref{fig: nyt hierarchy 8}, which can be interpreted similar to \textit{Tree20} or \textit{Tree7}, as discussed in the paper.

\begin{figure}
    \centering
    \includegraphics[width=\linewidth]{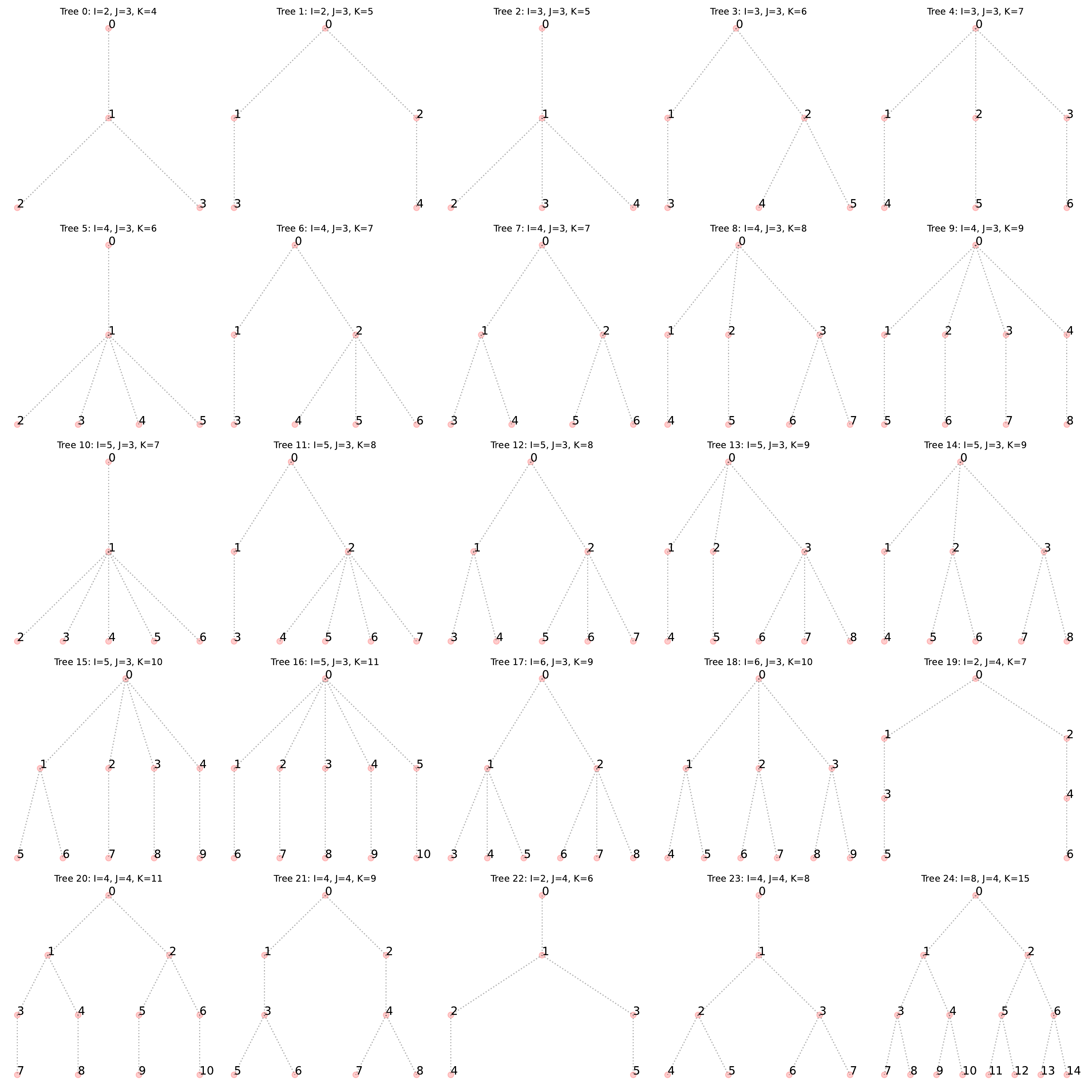}
    \caption{Different DRTs used for NYT dataset}
    \label{fig: nyt trees}
\end{figure}

\begin{table}[]
    \centering
    \begin{tabular}{rrrrrrrr}
Tree & loglik & min dist & min path & min width & min MM & min Gr & min proj \\
\hline
0 & -384.478119 & 0.095829 & 0.435556 & 0.081677 & 0.116825 & 1.413106 & 0.091344 \\
1 & -371.339691 & 0.097044 & 0.475556 & 0.086559 & 0.124632 & 1.967437 & 0.107878 \\
2 & -371.380646 & 0.098162 & 0.170732 & 0.092740 & 0.098162 & 0.827112 & 0.088639 \\
3 & -367.955231 & 0.105501 & 0.152993 & 0.089274 & 0.118787 & 1.088370 & 0.091389 \\
4 & -366.091553 & 0.110773 & 0.128603 & 0.095683 & 0.127080 & 1.470876 & 0.091933 \\
5 & -367.597961 & 0.100855 & 0.154867 & 0.087850 & 0.101692 & 0.838693 & 0.092132 \\
6 & -364.029053 & 0.107039 & 0.150442 & 0.095817 & 0.138978 & 1.346883 & 0.100305 \\
7 & -364.049042 & 0.109462 & 0.141593 & 0.095683 & 0.117186 & 1.070265 & 0.092187 \\
8 & -361.803009 & 0.110093 & 0.154867 & 0.099673 & 0.126911 & 1.296238 & 0.100837 \\
9 & -358.632141 & 0.115678 & 0.134956 & 0.098613 & 0.120622 & 1.463362 & 0.089122 \\
10 & -363.557709 & 0.093382 & 0.139073 & 0.082163 & 0.112251 & 0.869724 & 0.097159 \\
11 & -361.415192 & 0.106969 & 0.083885 & 0.093637 & 0.134366 & 1.307429 & 0.100872 \\
12 & -359.805176 & 0.108728 & 0.143488 & 0.096028 & 0.113834 & 1.082977 & 0.092621 \\
13 & -364.706207 & 0.107942 & 0.019868 & 0.093678 & 0.136534 & 1.140199 & 0.093300 \\
14 & -363.315582 & 0.110165 & 0.039735 & 0.095599 & 0.138494 & 1.101638 & 0.102580 \\
15 & -356.794220 & 0.117294 & 0.134658 & 0.100107 & 0.128460 & 1.058989 & 0.096342 \\
16 & -355.755402 & 0.106236 & 0.134658 & 0.100043 & 0.142102 & 1.645063 & 0.084705 \\
17 & -356.826508 & 0.110622 & 0.103524 & 0.097887 & 0.110622 & 0.973003 & 0.094247 \\
18 & -358.865204 & 0.108103 & 0.046256 & 0.092319 & 0.131394 & 1.065186 & 0.101542 \\
19 & -364.944122 & 0.108742 & 0.460000 & 0.093873 & 0.151502 & 2.445448 & 0.113184 \\
20 & -353.887177 & 0.113148 & 0.143805 & 0.095182 & 0.151048 & 1.769003 & 0.101841 \\
21 & -357.689148 & 0.106585 & 0.150442 & 0.088401 & 0.145929 & 1.375201 & 0.099952 \\
22 & -368.640930 & 0.100454 & 0.486667 & 0.087086 & 0.135501 & 2.027920 & 0.114793 \\
23 & -361.434509 & 0.105401 & 0.152655 & 0.089929 & 0.123659 & 1.075447 & 0.098708 \\
24 & -348.942963 & 0.113665 & 0.015351 & 0.093306 & 0.145915 & 0.623937 & 0.106803 \\
\end{tabular}

    \caption{Results for the different DRTs used: the columns are (i) index of the tree corresponding to Figure \ref{fig: nyt trees}, (ii) held-out log  likelihood per document, (iii) minimum distance between distinct topics, (iv) minimum path probability, (v) minimum width of the component polytopes, (vi) minimum minimum-matching distance among distinct component polytopes, (vii) minimum Grassmanian distance between affine hulls of distinct component polytopes and (viii) minimum projection distance of topics to other component polytopes, not sharing that topic.}
    \label{tab: nyt DRTs}
\end{table}

\begin{table}[]
    \centering
    \begin{tabular}{rrrr}
K & loglik & min topic distance & min width \\
\hline
2 & -397.795074 & 0.082901 & 0.082901 \\
3 & -389.209595 & 0.078906 & 0.067211 \\
4 & -378.412842 & 0.079232 & 0.065750 \\
5 & -376.222534 & 0.100496 & 0.084834 \\
6 & -375.433838 & 0.101611 & 0.085654 \\
7 & -375.300140 & 0.105880 & 0.092065 \\
8 & -374.194305 & 0.115331 & 0.092785 \\
9 & -373.980713 & 0.118708 & 0.078634 \\
10 & -381.322723 & 0.132685 & 0.100577 \\
11 & -374.463074 & 0.125991 & 0.092925 \\
12 & -376.339691 & 0.137340 & 0.099181 \\
\end{tabular}
    \caption{Results for LDA fit on NYT data with different number of topics ($K$): computed the (i) held-out log likelihood per document, (ii) minimum distance between distinct topics and (iii) minimum width of the topic polytope.}
    \label{tab: nyt LDAs}
\end{table}
\end{appendices}

\end{document}